\documentclass[11pt, reqno, a4]{amsart}

\usepackage{textcmds}

 \usepackage{graphics,graphicx}  

\usepackage{textcmds}  
\usepackage{amsmath, amssymb, amsfonts, amstext, verbatim, amsthm, mathrsfs}
\usepackage[mathcal]{eucal}
\usepackage{microtype}
\usepackage[all]{xy}
\usepackage[modulo]{lineno}
\usepackage[usenames]{color}
\usepackage{aliascnt}
\usepackage{enumitem}
\usepackage{xspace}
\usepackage{amsfonts}
\usepackage{amssymb}
\usepackage[centertags]{amsmath}
\usepackage{amsthm}
\usepackage[margin=2.8cm]{geometry}
\usepackage{dsfont}
\usepackage{bm}
\usepackage{xcolor}
\usepackage{subfigure}
\usepackage{amsmath}
\usepackage{array}
\usepackage[all]{xy}
\usepackage{makecell}

\usepackage{parskip}

\let\oldtocsection=\tocsection

\let\oldtocsubsection=\tocsubsection

\renewcommand{\tocsection}[2]{\hspace{0em}\oldtocsection{#1}{#2}}
\renewcommand{\tocsubsection}[2]{\hspace{1em}\oldtocsubsection{#1}{#2}}

\usepackage[backref,colorlinks=true,linkcolor=blue,citecolor=blue,urlcolor=blue,citebordercolor={0 0 1},urlbordercolor={0 0 1},linkbordercolor={0 0 1}]{hyperref} 
\usepackage{amssymb}
\usepackage{hyperref}

\newtheorem{theorem}[equation]{Theorem}

\newtheorem{lemma}[equation]{Lemma}

\newtheorem{proposition}[equation]{Proposition}
\newtheorem{corollary}[equation]{Corollary}
\newtheorem{conjecture}[equation]{Conjecture}
\newtheorem{definition-lemma}[equation]{Definition-Lemma}

\theoremstyle{definition}
\newtheorem{definition}[equation]{Definition}

\newtheorem{example}[equation]{Example}

\theoremstyle{remark}
\newtheorem{remark}[equation]{Remark}

\numberwithin{equation}{section}
\numberwithin{figure}{section}

\newcommand{\bZ} {\mathbb{Z}}
\newcommand{\bQ} {\mathbb{Q}}
\newcommand{\bR} {\mathbb{R}}
\newcommand{\bC} {\mathbb{C}}

\newcommand{\bG} {\mathbb{G}}
\newcommand{\bN} {\mathbb{N}}
\newcommand{\bP} {\mathbb{P}}
\newcommand{\bF} {\mathbb{F}}

\newcommand {\cA}  {\mathcal{A}}

\newcommand {\cD}  {\mathcal{D}}
\newcommand {\cE}  {\mathcal{E}}
\newcommand {\cF}  {\mathcal{F}}

\newcommand {\cL}  {\mathcal{L}}

\newcommand {\cO}  {\mathcal{O}}
\newcommand {\cP}  {\mathcal{P}}

\newcommand {\cW} {\mathcal{W}}

\newcommand {\cY}  {\mathcal{Y}}

\renewcommand {\ker} {\operatorname{ker}}

\newcommand {\Id}  {\operatorname{Id}}

\newcommand {\Bir}  {\operatorname{Bir}}

\newcommand {\Hom}  {\operatorname{Hom}}
\newcommand {\Pic}  {\operatorname{Pic}}

\newcommand {\loc}  {\mathrm{loc}}
\newcommand {\tw}  {\mathrm{tw}}
\newcommand {\Bl}  {\mathrm{Bl}}
\newcommand {\HK}  {\mathrm{HK}}

\newcommand {\Supp}  {\operatorname{Supp}}

\newcommand {\SL}  {\operatorname{SL}}

\DeclareMathOperator{\Ext}{Ext}
\DeclareMathOperator{\Coh}{Coh}
\DeclareMathOperator{\Perf}{Perf}
\DeclareMathOperator{\PD}{PD}
\DeclareMathOperator{\Auteq}{Auteq}
\DeclareMathOperator{\Aut}{Aut}
\DeclareMathOperator{\End}{End}

\DeclareMathOperator{\Tube}{Tube}

\newcommand{\ev}{\mathrm{ev}}
\newcommand{\gr}{\mathrm{gr}}
\newcommand{\Lag}{\mathrm{Lag}}

\newcommand{\Adm}{\mathrm{Adm}}

\newcommand{\Br}{\mathrm{Br}}
\newcommand{\gen}{\mathrm{gen}}
\newcommand{\symp}{\mathrm{symp}}

\newcommand{\ch}{\mathrm{ch}}
\newcommand{\w}{\varpi}

\newcommand{\W}{\mathcal{W}}
\newcommand{\Symp}{\operatorname{Symp}}

\def\mydate{\ifcase\month \or January\or February\or March\or
April\or May\or June\or July\or August\or September\or October\or 
November\or December\fi \space\number\day,\space\number\year}

\begin{document}

\title{Symplectomorphisms of some Weinstein 4-manifolds}
\author{Paul Hacking}
\address{Department of Mathematics and Statistics, Lederle Graduate Research Tower, University of Massachusetts, Amherst, MA 01003-9305}
\email{hacking@math.umass.edu}
\author{Ailsa Keating}
\address{Department of Pure Mathematics and Mathematical Statistics, Centre for Mathematical Sciences, University of Cambridge, Wilberforce Road, Cambridge, CB3 0WB}
\email{amk50@cam.ac.uk}

\maketitle


\begin{abstract} 

Let $M$ be a Weinstein four-manifold mirror to $Y \backslash D$ for $(Y,D)$ a log Calabi--Yau surface; this is typically the Milnor fibre of a smoothing of a cusp singularity. We introduce two families of symplectomorphisms of $M$: Lagrangian translations, which we prove are mirror to tensors with line bundles; and nodal slide recombinations, which we prove are mirror to automorphisms of $(Y,D)$. The proof uses a detailed compatibility between the homological and SYZ view-points on mirror symmetry. Together with spherical twists, these symplectomorphisms are expected to generate all autoequivalences of the wrapped Fukaya category of $M$ that are compactly supported in a categorical sense. A range of applications is given.

\end{abstract}

\tableofcontents

\section{Introduction}

\subsection{Context and main theorem}

Consider a symplectic manifold $M$, possibly with boundary. Let  $ \Symp_c (M)$ be its group of compactly supported symplectomorphisms (if $\partial M \neq \emptyset$, we ask for elements of $\Symp_c (M)$ to be the identity on  a collar neighbourhood of $\partial M$). We are interested in the symplectic mapping class group $\pi_0 \Symp_c (M)$.

The best current source of  symplectic mapping class group elements is Dehn twists in Lagrangian spheres. They admit modest generalisations: fibred twists, and twists in other Lagrangians with periodic geodesic flow \cite{Seidel_graded, Perutz, WW}. In the case where $M$ is non exact, there are sporadic examples of symplectomorphisms beyond these constructions, such as global symmetries (e.g.~ some Lie group actions on closed manifolds), and  Shevchishin--Smirnov's  four-dimensional `elliptic twist',  defined in the presence of a self-intersection $(-1)$ symplectic two-torus \cite{Smirnov-Shevchishin}. However, in the case where $M$ is Weinstein, no further examples of symplectic mapping classes are known. Moreover, as the monodromy of complex Morse singularities, Dehn twists have long been regarded as particularly natural in the Weinstein setting. A long-standing folklore question is whether they generate, or generically generate, the symplectic mapping class group of a  Weinstein manifold.

In this paper, we introduce new families of symplectic mapping classes for a large collection of four-dimensional Weinstein manifolds. The symplectic mapping classes are demonstrably not generated by Dehn twists. Heuristically, they should be thought of as `natural' rather than `exotic'; in particular, they are mirror to some basic operations on coherent sheaves.

\subsubsection*{Mirror symmetry context} Here is the setting. On the $A$ side, we have a four-dimensional Weinstein domain $M$. It should be thought of as any smoothing of a cusp or simply elliptic surface singularity (see \cite{Looijenga, GHK1, Engel}). With a small number of exceptions, we get all possible total spaces of exact, 2-dimensional Lagrangian torus fibrations with non-degenerate singularities. 
In particular, $M$ can be described explicitly as a Weinstein handlebody, by starting with $D^\ast T^2$ and gluing on Weinstein two handles along the conormal Legendrian lifts of rational slope circles in $T^2$, see \cite{HK}.  
On the $B$ side, we have a log Calabi--Yau surface $(Y,D)$. Whenever $M$ is a smoothing of a cusp singularity, $D$ is the cycle of the dual cusp singularity. Within its complex deformation class, $(Y,D)$ has the complex structure which is distinguished by the fact that it induces the split mixed Hodge structure. Let $U$ denote $Y \backslash D$. We can equip $M$ with a superpotential   $w: M \to \bC$ (see \cite{Keating, HK}).  
Homological mirror symmetry then consists of a triple of compatible explicit equivalences, which we take to be fixed for the rest of the introduction:
$$
\Perf(D) \simeq D^\pi \cF(\Sigma) \qquad D(Y) \simeq D^b \cF (w) \qquad D(U) \simeq D^b \cW(M)
$$
where $D^b \cF (w)$ is the derived directed Fukaya category of $w$, $D^b \cW(M)$ the derived wrapped Fukaya category of $M$, and $D^\pi \cF(\Sigma)$ the split-closed derived Fukaya category of a smooth fibre of $w$ near infinity, see \cite[Theorem 1.1]{HK}. 

 There are rich connections to cluster theory, cf.~\cite{GHKK, STW}.

\subsubsection*{Categorical compact support} Suppose $f \in \pi_0 \Symp_c M$ is a mapping class element mirror to an autoequivalence $\check{f} \in \Auteq D(Y)$. 
 As $f$ fixes the fibres of $w$ outside a compact set, homological mirror symmetry implies that $\check{f}$ respects restriction to $\Perf D$, i.e.~$\iota^\ast_D  \circ  \check{f} = \iota^\ast_D $. By abuse of terminology, we say that an autoequivalence of $D(Y)$ satisfying this condition has `compact support on $Y \backslash D$'. (From the perspective of the geometry of $Y$, `asymptotically compactly supported' would perhaps be more accurate.)

\subsubsection*{Main theorem}

Our main results can be summarised as follows.

\begin{theorem} \label{thm:main}
Let $(Y,D)$ be an arbitrary log Calabi-Yau surface with maximal boundary and distinguished complex structure, and let $M$ be the mirror Weinstein manifold, with superpotential $w: M \to \bC$. 
\begin{enumerate}

\item (Theorem \ref{thm:mirror_autos}.) Suppose $\cL \in \Pic (Y)$ is any line bundle such that $\cL|_D = \cO_D$. Then there exists a compactly supported symplectomorphism $\sigma_{\cL}$ of $M$ such that  under homological mirror symmetry, $\otimes \cL \in \Auteq D(U)$ is taken to $\sigma_{\cL} \in \Auteq D^b \cW(M)$. We call $\sigma_{\cL}$ a Lagrangian translation.

\item (Theorem \ref{thm:nodal_slide_recombination}.) Suppose $\varphi \in \Aut (Y)$ is any biholomorphic automorphism fixing $D$ pointwise. Then there exists a compactly supported symplectomorphism $\check{\varphi}$ of $M$ such that under  homological mirror symmetry, $\varphi_\ast \in \Auteq D(U)$, is taken to $\check{\varphi} \in  \Auteq D^b \cW(M)$. (This also holds at the level of $\Auteq D(Y)$ and $  \Auteq D^b \cF (w)$.) We call $\check{\varphi}$ a nodal slide recombination. 
\end{enumerate}

\end{theorem}

The names will be self-explanatory from the constructions. The conditions on $\cL$ and $\phi$ are precisely the ones needed  so that $\otimes \cL$ and $\phi_\ast$ have compact support on $Y \backslash D$ in the sense introduced above.  
Both families on the $B$ side are understood:
\begin{itemize}
\item The  line bundles in (1) form the subgroup $Q = \langle D_1 \ldots, D_k \rangle^\perp \subset \Pic (Y)$, where the $D_i$ are the irreducible components of $D$. 
 We have $Q \simeq H_2(U) / \bZ \langle \gamma \rangle$, where $\gamma$ is the class of a linking torus for a node of $D$. 
\item
For (2), the Torelli theorem in \cite{GHK2} implies that the group of automorphisms fixing $D$ pointwise is identified with $\Adm / W$ in Gross--Hacking--Keel's notation, where $\Adm$ is the group of admissible lattice automorphisms of $\Pic (Y)$, and $W$ is the Weyl group. (See \cite[Definition 4.2]{GHK2}.) 
\end{itemize}

Geometrically, our constructions are well-behaved with respect to composition: we get a faithful group homomorphism (Theorem \ref{thm:new_symplectos} part (1)):
$$
Q \rtimes \Adm / W \hookrightarrow \pi_0 \Symp_c (M).
$$
Let $\langle \Phi \rangle \subset Q$ be the sublattice generated by classes of $(-2)$ curves in $Y \backslash D$. For any $\cL \in \Phi $, we'll see that the symplectomorphism $\sigma_\cL$ can be factorised as the product of two Dehn twists in Lagrangian spheres lifting these roots (Lemma \ref{lem:two_twists_Lag_translation}). All other symplectomorphisms we construct are new, in the following sense. Denote by $\Br \subset \pi_0 \Symp_c M$ the subgroup generated by all Dehn twists, and by $Q_\symp$, etc, the images of our groups in $\pi_0 \Symp_c (M)$. Theorem \ref{thm:new_symplectos} part (2) tells us that:
$$
( Q \rtimes \Adm / W )_\symp \cap  \Br= \langle \Phi \rangle_\symp \rtimes \{ 1\}  \subset \pi_0 \Symp_c M.
$$  
 Note that elements of $\Adm / W$ persist under blow-ups  at points of $D$, which are precisely the ones which respect the log CY property; from that perspective, non-triviality of $\Adm / W$ should be thought of as generic. 
More generally, via Weinstein handle attachments, we get symplectic mapping classes of more general Weinstein 4-manifolds. As non-generation by Dehn twists is detected at the level of the second homology lattice, one immediately gets verifiable sufficient criteria for it surviving under handle attachment.

\begin{example} 
\label{ex:favourite_ex}
(Examples \ref{ex:k=8case} and \ref{ex:k=8ctd}.)
The first examples of new symplectomorphisms are in relatively elementary manifolds. One example is the Weinstein manifold $M$ given by $\bF_1 \backslash E$, the complement of a \emph{smooth} anticanonical divisor in $\bF_1$, the blow-up of $\bP^2$ at a point. The mirror is a log Calabi-Yau $(Y,D)$ such that $D$ is a cycle of eight $(-2)$ curves. Then $Q \simeq \bZ^2$, $\Phi = \emptyset$, and $\Adm  / W \simeq \bZ$. (The latter consists of eighth powers of elements of the Mordell--Weil group of $Y$, see Section \ref{sec:Mordell-Weil}.) This means we have a family of symplectomorphisms
$$
\bZ^3 \leq  \pi_0 \Symp_c (M).
$$
This includes a $\bZ$ factor which acts trivially on the wrapped Fukaya category $D^b\cW(M)$ (but not  on the directed Fukaya category $D^b \cF(w)$), which we expect to be generated by a boundary twist. In contrast, notice that following \cite{Gromov}, $\Symp(\bF_1, \omega_{FS})$ is homotopy equivalent to $U(2)$, where $\omega_{FS}$ is the Fubini--Study form; in particular, its symplectic mapping class group is trivial. 
\end{example}

For further examples, see  Sections \ref{sec:nodal_slide_examples} and \ref{sec:new_symplectos}.

\subsection{Construction sketches} Both constructions are most natural from the perspective of an almost-toric fibration $\pi: M \to \bR^2$. This has connections with the Gross--Siebert programme, and should be thought of as an SYZ fibration on $M$. 
To describe it, pick a toric model for $(Y,D)$: possibly after blowing up corners of $D$, we can find a toric pair $(\bar{Y}, \bar{D})$ and a map $(Y,D) \to (\bar{Y}, \bar{D})$ given by iteratively blowing up smooth points on $\bar{D}$ and its strict transforms. As $Y$ has the distinguished complex structure, we (possibly iteratively) blow up a single favourite point on each $\bar{D}_i$, say $m_i$ times. Let $v_i$ be the ray in the fan for $\bar{Y}$ corresponding to $\bar{D}_i$. The associated almost-toric fibration $\pi$ has exactly $m_i$ nodal fibres with invariant direction $v_i$, with all invariant lines concurrent. (See \cite[Theorem 1.2]{HK}.) Changes of toric models give equivalent fibrations, which will be further exploited below. 

\subsubsection*{Lagrangian translations}

Fix a reference Lagrangian section $L_0$ for the almost-toric fibration $\pi: M \to \bR^2$. This should be thought of as mirror to $\cO_Y$. Going back to early days of SYZ mirror symmetry, Lagrangian sections of $\pi$, up to suitable equivalence, are expected to correspond to line bundles on $Y$. Restricting to `compactly supported' subsets on both sides, we will see in Corollary \ref{cor:mirror_autos_K-theory} that there are one-to-one correspondences
\begin{center}
\begin{tabular}{ccc}
$\Bigg\{ $ \makecell{ Lagrangian sections $L$ of $\pi$ equal to $L_0$ near $\partial M$ \\ up to compactly supported fibre-preserving Hamiltonian isotopy } $ \Bigg\}$
& $\longleftrightarrow$ & 
 $ H_2(M) / \bZ \langle \gamma^{\vee} \rangle 
$   
\\
&& $\Big\updownarrow$ \\
Line bundles in $Q $ & $\longleftrightarrow$ &  $H_2(U) / \bZ \langle \gamma \rangle $    
 \end{tabular}

\end{center}

where $\gamma$, as before, is a linking torus for a node of $D$, and $\gamma^\vee$, the ``dual torus'', is the class of a fibre of $\pi$. 
The vertical equivalence is induced by an explicit diffeomorphism between $M$ and $U$, the top one by taking $L$ to $[L] - [L_0]$, and the bottom one by taking the first Chern class and Poincar\'e duality.
To define a Lagrangian translation, fix a Lagrangian section $L$. 
 Arnol'd--Liouville coordinates give a linear structure on each smooth fibre, and adding $L-L_0$ fibrewise can be checked to be locally a symplectomorphism. We show that this can be extended over the singular fibres to get an element of $\pi_0 \Symp_c(M)$ (Proposition \ref{prop:Lag_translations}).

This is related to an early conjecture of Gross for monodromy diffeomorphisms associated to large complex structure limit points for \emph{compact} Calabi--Yau manifolds \cite[Conjecture 3.7]{Gross_specialI}. Related constructions on individual Lagrangian sections can be found in Subotic's thesis \cite[Chap.~5]{Subotic} and Hanlon--Hicks' work \cite[Section 5.2]{Hanlon-Hicks}, for fibrations without singular fibres. Finally, for some distinguished line bundles, the mirror Lagrangian translation is described as a monodromy diffeomorphism in \cite{EF} (see also Remark \ref{rmk:Engel-Friedman}).

\subsubsection*{Nodal slide recombinations}
Suppose $\varphi$ is an automorphism of $Y$ preserving $D$ pointwise. To construct the mirror $\check{\varphi}$, we use input from birational geometry. Fix any toric model for $(Y,D)$, say $f: (Y,D) \to (\bar{Y}, \bar{D})$. (We can without loss of generality ignore toric blow-ups.) It follows from the two-dimensional Sarkisov programme that the toric model $f \circ \varphi: (Y,D) \to (\bar{Y}, \bar{D})$ can be obtained by starting from $f: (Y,D) \to (\bar{Y}, \bar{D})$  and performing a sequence of elementary transformations (see Proposition \ref{prop:change_of_toric_model} and \cite{Corti-Kaloghiros, CKS}). Equivalently, $\varphi$ restricts to a birational transformation of $(\bC^*)^2$, say $\bar{\varphi}$, and the sequence of elementary transformations corresponds to a factorisation of $\bar{\varphi}$  into a composition of maps of the form $\psi \circ E \circ \psi^{-1}$, where $E: (x,y) \mapsto (x, y/(x+1))$ and $\psi \in SL_2(\bZ)$ (cf.~\cite{Blanc}). 
 Readers may recognise $E$ as the standard two-variable cluster transformation. 

On the mirror side, each elementary transformation determines an operation on $M$. In terms of the almost-toric fibration, this is a nodal slide followed by a cut transfer (\cite{Symington}, see Section \ref{sec:change-toric-model-A-side}); in terms of the Weinstein skeleton of $M$, a Lagrangian mutation (for which we will follow \cite{Pascaleff-Tonkonog}). The full sequence of transformations returns the same almost-toric fibration (or presentation of the Weinstein skeleton) as we started with, and determines a symplectomorphism of $M$. 
Independence of the choice of factorisation reduces to checking a symplectic incarnation of the `$A_2$ cluster relation', see Proposition \ref{prop:indep_of_choices}; this uses Blanc's description of the relations in the group of volume-preserving birational transformations of $\bP^2$ \cite{Blanc}.

 From the cluster perspective, the symplectomorphisms are interesting in of their own right: they are associated to loops in the cluster graph for $Y \backslash D$. In terms of actions on Lagrangian submanifolds, the idea that elementary transformations are mirror to Lagrangian mutations has been widely explored in the wake of \cite{Auroux_Gokova} -- see  \cite{Vianna, Vianna_delPezzo, STW, Pascaleff-Tonkonog}. Finally, for integral affine bases of SYZ fibrations on K3 surfaces, 
 \cite[Conjecture 7]{Kontsevich-Soibelman} considers similar operations.

\begin{remark}

Heuristically, we expect a symplectomorphism to preserve the SYZ fibration (up to Hamiltonian isotopy) if and only if the mirror autoequivalence preserves the collection of structure sheaves of points. On the other hand, by \cite[Corollary 5.23]{Huybrechts}, such autoequivalences are precisely given by $\Pic(Y) \rtimes \Aut(Y)$: this would mean that we have found all compactly supported symplectomorphisms respecting the SYZ fibration, up to equivalence in $\Auteq \cW(M)$. 

\end{remark}

\subsubsection*{Notes on non-compactly supported maps}
Our constructions naturally extend to give examples of non-compactly supported symplectomorphisms:

\begin{itemize}

\item for Lagrangian translations, we can use a Lagrangian section $L$ which does not agree with $L_0$ near infinity;
see Lemma \ref{lem:sections_M_noncpt} and Proposition \ref{prop:Lag_translations_group}. 
We expect this to be mirror to $\otimes \cL$ for a line bundle $\cL$ such that $\cL|_D \ncong \cO_D$. One possible approach is to use the monomially admissible framework as developed in  \cite{Hanlon, Hanlon-Hicks, Hanlon-Ward}.

\item for nodal slide recombinations, we can add extra automorphisms of $Y$, which fix $D$ setwise but not pointwise, 
 see Section \ref{sec:SL_2(Z)_transformations}. This corresponds to adding actions of elements of $SL_2(\bZ)$ on the base of the almost-toric fibration.  We expect to get an action by symplectomorphisms of the cluster modular group introduced in \cite{FG2}. See also \cite{KW}.

\end{itemize}

\subsection{HMS action}

The proof that the symplectomorphisms we construct have the expected mirrors requires several sections.
In particular, we establish tractable criteria for an autoequivalence of $D(Y)$ to be the identity: see  Proposition \ref{prop:detecting_id} or the proof of Theorem \ref{thm:mirror_autos}.  
As part of applying this,  we give an explicit construction of Lagrangian spheres mirror to $i_\ast \cO_C(a)$ for \emph{any} $C \subset U$ a $(-2)$ curve, and $a \in \bZ$; this may be of independent interest. 

\subsection{Spherical objects and full auto-equivalence groups}\label{sec:intro_further_applications}

\subsubsection*{Spherical objects} Some corollaries on Lagrangian spheres may be of independent interest. 
In particular, Corollary \ref{cor:infinite_spheres} shows that whenever there are infinitely many $(-2)$ curves in $Y \backslash D$,  we can construct a countably infinite family of Lagrangian spheres, none of which is contained in the subcategory of $D^b \cW(M)$ split-generated by the others.  
This implies that there is a large collection of finite type Weinstein four-manifolds with this property. 
In contrast, for Milnor fibres of hypersurface singularities, all known Lagrangian spheres are split generated by a finite collection of them (by Picard--Lefschetz theory).

In general, spherical objects in $D(Y)$ are not classified, and that problem is largely orthogonal to the focus of this article. The folklore expectation is as follows. 

\begin{conjecture}\label{conj:sphericals} 
Let $\cA$ be the collection of spherical objects 
$$
\cA := \{ i_\ast \cO_C (a) \, | \, C \text{ a }(-2) \text{ curve, } a \in \bZ \}.
$$
Let $T_{\cA}$ be the group generated by spherical twists in objects of $\cA$. Then any spherical object in $D(Y)$ is quasi-isomorphic to an element of $T_{\cA} \cdot \cA$.
\end{conjecture}

Any spherical object in $D(Y)$ has support in $Z$, the union of all $(-2)$ curves in $Y$, and the conjecture is known for all connected components of $Z$ which are either chains or cycles of $(-2)$ curves (\cite{IU, Uehara1}, see Lemma \ref{lem:type_A_sphericals}; note the extra technical assumption in the case where the cycle of $(-2)$ curves is $D$ itself). 

\subsubsection*{Autoequivalence groups}  A companion question is to understand the full group of autoequivalences of $D(Y)$. For compactly supported ones, the general expectation translates as follows.

\begin{conjecture}\label{conj:auteqs}
Cf.~\cite[Conjecture 1.2]{Uehara2}.
Let $\Auteq_c D(U)$ denote compactly supported autoequivalences of $D(U)$ (see Section \ref{sec:compact_support}). 
Then 
$$
\Auteq_c D(U) = \langle T_{\cA'}, Q \rangle \rtimes \Aut (Y, D; \text{pt})
$$
where $\cA' \subset \cA$ is given by restricting to $(-2)$ curves in $Y \backslash D$,  $Q$ is as before, and $ \Aut (Y, D; \text{pt})$ is the group of automorphisms of $Y$ fixing $D$ pointwise.

\end{conjecture}
This too is known to hold whenever all connected components of $Z \cap (Y \backslash D)$ are chains or cycles of $(-2)$ curves \cite{IU, Uehara1, Uehara2}. Examples include the cases where $D$ is semi-definite case with $k \geq 5$ (assuming $D$ contains no $(-1)$ curves). 
The difficulty is classifying autoequivalences of $D_Z(Y)$, the derived category of coherent sheaves with support on $Z$. Note that our results imply the following:

\begin{theorem}\label{thm:full_auteq_groups}
(Proposition \ref{prop:Lagrangian_spheres}  and Theorem  \ref{thm:new_symplectos})
Assume that Conjecture \ref{conj:auteqs} holds. Then any  compactly supported autoequivalence of $D^b \cW(M)$ can be represented by a compactly supported symplectomorphism of $M$. Moreover, the 
 categorical symplectic mapping class group of $M$, i.e.~the image of the (graded) symplectic mapping class group in $\Auteq D^b \cW(M)$, is generated by Lagrangians translations, nodal slide recombinations, and Dehn twists in Lagrangian spheres mirror to line bundles on $(-2)$ curves.
\end{theorem}

The theorem holds unconditionally whenever  Conjecture \ref{conj:auteqs}  is known -- for instance, for Example \ref{ex:favourite_ex}, for which  $Z = \emptyset$.

\subsection{Monodromy and moduli spaces of complex structures}\label{sec:monodromies}

We end with some further motivation. Why would one expect autoequivalences of $D(Y)$ to have mirrors realised by symplectomorphisms? For our families, we conjecture that all of these symplectomorphisms are induced by monodromy in the complex moduli space for $M$.

 For the finite list of  cases where the intersection pairing associated to $D$ is semi-definite, $M$ is the Milnor fibre of a simple elliptic singularity. It compactifies to a del Pezzo surface $X$, with $M = X \backslash E$, for  $E \subset X$ a smooth anticanonical. The moduli space of pairs $(X', E')$, where $X'$ is deformation equivalent to $X$ and $E' \subset X'$ smooth anticanonical, is well understood. Aspects of this will be revisited in Section \ref{sec:monodromy_semidef_symplectic_side}.

For now, focus on the case where $D$ is negative definite. There is now no a priori reason to expect the Stein mainfold $M$ to be a variety. To get a precise `complex moduli space' for $M$, consider the smoothing component of the singularity for which $M$ is a Milnor fibre, with the discriminant locus removed; call this ${\mathbf{S}}$. Note that one expects monodromy in a smoothing component to  always induce compactly supported symplectomorphisms, whereas for the full  complex moduli space  in the semi-definite case this won't always be true.

\begin{conjecture}\label{conj:smoothing_cpt} Assume that  $D$ is negative definite. 
There is a natural isomorphism:
$$
{\mathbf{S}}  \simeq \Big(  ( Q_{\mathbb{R}} + i \mathcal{C} ) \backslash \bigcup_{\alpha \in \Phi, k \in \bZ } \{ x \, | \, \langle x, \alpha \rangle = k \} \Big) / (Q \rtimes \mathrm{Adm})
$$
where $Q = \langle D_1, \ldots, D_k \rangle^\perp \subset \Pic Y$, for $D_1, \ldots, D_k$ the irreducible components of $D$; $\mathcal{C}$ is the interior of $\overline{K(Y_{\gen})} \cap Q \subset \Pic Y \otimes \bR = H^2_{dR}(Y; \bR)$, where  $K(Y_{\gen})$ is the generic Kaehler cone (i.e.~the Kaehler cone for a generic $Y'$ in the complex deformation space of $Y$); $\Phi$ is set of all roots in $Q$; and $\Adm$ is the subgroup of admissible lattice automorphisms of $\Pic Y$, i.e.~the ones preserving  the $[D_i]$ and $\mathcal{C}$. 
\end{conjecture}

\begin{remark} For the purpose of computing $\pi_1(\mathbf{S})$, we could replace $\mathcal{C}$ by the larger cone given by one connected component of the positive square cone; heuristically, this is because we know that the strata of the hyperplane arrangements in these open sets are in bijection.  
\end{remark}

In the cases $k \leq 5$, Conjecture \ref{conj:smoothing_cpt} is a theorem of Looijenga \cite{Looijenga}. 

The intersection pairing on $\Pic Y$ has hyperbolic signature, $\mathcal{C}$ is contained in the positive square cone, and $\Adm$ acts by  isometries,  which implies that $Q \rtimes \Adm$ acts properly discontinuously on its domain. This means the quotient is well-behaved. 

We expect a representation of $\pi_1(\mathbf{S})$ in $\pi_0 \Symp_c(M)$, and, on the mirror side, in $\Auteq_cD(U)$. Loosely,  $Q \rtimes \Adm$ gives the action at the level of homology, with $Q$ corresponding 
to  Lagrangian translations ($\leftrightarrow \otimes \cO(\cL)$), and $\Adm$ to the actions of nodal slide recombinations ($\leftrightarrow \Aut(Y,D; \textrm{pt})$) and Dehn twists ($\leftrightarrow$ Weyl group); and loops in the hyperplane complement should give Torelli elements. We expect an elementary loop around a single hyperplane to correspond to the square of  a Dehn twist in a known Lagrangian sphere; on the $B$ side, using the previous notation, these are $T_{\cA} \cdot \cA$.

The action of $Q \rtimes \Adm$ lifts to $\Hom(K(U), \bC)$, where $K(U) = \bZ \oplus \Pic(U)$ is the $K$-theory of $U$. We can consider the space 
$$
\mathbf{S}' := \Big( \Hom (K(U), \bC)^\circ \backslash \bigcup_{\hat{\alpha} \in \hat{\Phi}} \{ {\hat{\alpha}}^\perp = 0 \}  \Big)/ (Q \rtimes \Adm) $$
where $\hat{\Phi}$ is the lift of $\Phi \subset Q$ to $K(U)^\ast$, and $\phantom{.}^\circ$ denotes the restriction to the $\bC^\ast$ cover of ${\mathbf{S}}$, the moduli space of Conjecture \ref{conj:smoothing_cpt}.
Informally, $\mathbf{S}'$ is the moduli space of complex structures on $M$ equipped with a choice of holomorphic volume form $\Omega$. (This could be made precise by compactifying  the cusp to an Inoue surface, which fixes $\Omega$ up to a scalar, cf.~\cite{Looijenga}.) 
Monodromy about $\bC^\ast$ should correspond to the square of the shift.

\begin{conjecture}
The universal cover $\widetilde{\mathbf{S}'}$ of $\mathbf{S}'$ is a connected component of the space of Bridgeland stability conditions on $D^b \Coh(U)$. The natural map to the intermediate covering space which lies inside $Hom (K(U), \bC)$, before quotienting by $Q \rtimes \Adm$, 
 is given by mapping to the central charge. 
\end{conjecture}

We expect that there are two connected components of the space of Bridgeland stability conditions, interchanged by the shift functor, mapping to $Q_\mathbb{R} \pm i\mathcal{C}$ via the central charge. 

\begin{remark} \label{rmk:Engel-Friedman}
The reader may be interested in comparing with work of Engel--Friedman in \cite[Section 5.1]{EF}. They construct monodromy diffeomorphisms of the Milnor fibre $M$ of a smoothing of a cusp singularity associated to a negative definite log CY $(Y,D)$ together with a choice of line bundle $\cL \in \mathcal{C}$, describing in particular the action on homology.
We expect that $M$ is mirror to $(Y,D)$ and these monodromy diffeomorphisms are the Lagrangian translations mirror to $(\cdot) \otimes \cL$, associated to loops bounding a disc passing through the origin in the deformation space of the cusp (rather than an arbitrary closed loop). 
\end{remark}

\subsection*{Organisation of the paper}
Section \ref{sec:autoequivalences} gathers background material on autoequivalences of the derived categories on the $B$ side. 
Section \ref{sec:hms_background} contains the input we will use from mirror symmetry, including some refinements of results in \cite{HK}. 
Section \ref{sec:translations_and_tensors} sets up Lagrangian translations, including a proof of Theorem \ref{thm:main} part (1) at the $K$-theoretic level.
Section \ref{sec:spheres} constructs Lagrangian spheres mirror to line bundles on $(-2)$ curves; this also allows us to complete the proof of  Theorem \ref{thm:main} part (1) (Theorem \ref{thm:mirror_autos}). 
Section \ref{sec:nodal_slide_recombinations} constructs nodal slide recombinations and proves Theorem \ref{thm:main} part (2), and the independence of choices mentioned above. Relations between symplectomorphisms from different families  are established in Section \ref{sec:relations}. 
Finally, Section \ref{sec:applications} contains a range of applications: first, the packaging of the new symplectomorphisms as Theorem \ref{sec:new_symplectos} and discussion of examples in Section \ref{sec:new_symplectos}; applications of our constructions to questions about spherical objects in Section \ref{sec:applications_to_sphericals}; and it ends with extended discussion of the semi-definite case -- further constructions from monodromy (Section \ref{sec:monodromy_semidef_symplectic_side}) and connections with work of Collins--Jacob--Lin on hyperkaehler structures (Section \ref{sec:semi-def_hyperkaehler}).

\subsection*{Notation and conventions}
While they will be introduced as we go along, we record standing notational conventions here for the readers' convenience.

\emph{B-side.}
$(Y,D)$ will denote a log Calabi-Yau pair, assumed to have maximal boundary unless otherwise specified; we use $U$ to denote $Y \backslash D$, $k$ for the number of irreducible components $D_1, \ldots, D_k$ of $D$, and $(Y_e,D)$ for the surface in the deformation class of $(Y,D)$ with the distinguished complex structure, i.e.~such that the mixed Hodge structure on $Y_e \backslash D$ is split (see exposition in \cite[Section 2.2]{HK}). We say that $D$ is negative definite, semi-definite or indefinite whenever the intersection form associated to the $D_i$ is. (Abusing terminology, we use `semi-definite' to mean strictly semi-definite.)
When $D$ is semi-definite, $Y_e$ is rational elliptic, and we denote by $\w: Y_e \to \bP^1$ the minimal rational elliptic fibration ($\w^{-1}(\infty) = D$).

\emph{A-side.} The mirror to $(Y_e, D)$ is a Weinstein manifold $M$, which is the total space of a Lefschetz fibration $w: M \to \bC$.  $M$ is the total space of  (the equivalence class of) an almost-toric fibration $\pi: M \to \bR^2$.  In the semi-definite case, $M = X \backslash E$, where $X$ is a del Pezzo surface of degree $k$ and $E \subset X$ a smooth anti-canonical divisor (elliptic curve).

\subsection*{Acknowledgements}
We are grateful to 
Roger Casals for correspondence about Remark \ref{rmk:contact_pentagon}; 
Alessio Corti and Wendelin Lutz for discussions related to \cite{Corti-Kaloghiros, CKS}; Abigail Ward for feedback on earlier versions of this article, including discussions around Section \ref{sec:Lag_sections}; and
Ivan Smith for explanations of \cite{Sheridan-Smith} and for detailed comments on an earlier draft.

PH was partially supported by NSF grants DMS-1901970 and DMS-2200875. AK was partially supported by an award from the Isaac Newton Trust, EPSRC Fellowship  EP/W001780/1, and the ERC starting grant SingSymp.

\textbf{Open Access.} For the purpose of open access, the authors have applied a Creative Commons Attribution (CC:BY) licence to any Author Accepted Manuscript version arising from this submission.
\textbf{UKRI data access statement.} 
There is no dataset associated with this paper.

\section{Autoequivalences of $D^b \Coh (Y)$ and $D^b \Coh (U)$}\label{sec:autoequivalences}

\subsection{Weyl group and classes of spherical objects}

In the wake of \cite{Seidel-Thomas}, a well-studied class of autoequivalences of $D(Y)$ are spherical twists. 
In particular, given a $(-2)$ curve $C \subset Y$, the sheaves $i_\ast \cO_C(a)$, $a \in \bZ$, are spherical objects. Let $Z$ be the union of all $(-2)$ curves in $Y$. (These are either components of $D$ or disjoint from it.) Any spherical object in $D(Y)$ has support on $Z$, using the arguments in the proof of \cite[Claim 6.3]{Uehara2}. 

\begin{lemma}\label{lem:type_A_sphericals}
Suppose that $\tilde{Z} \subset Z$ is either a type $A$ chain of $(-2)$ curves or a cycle of $(-2)$ curves. In the latter case, we further assume that either $\tilde{Z} \neq D$ or $Y=Y_e$. Then the spherical objects with support on $\tilde{Z}$  are the images of
$$\{ i_\ast \cO_C(a) \, | \, a \in \bZ, C \subset \tilde{Z} \text{ a } (-2)\text{ curve}\}$$
 under spherical twists in themselves.
\end{lemma}

\begin{proof}
When $\tilde{Z}$ is a type $A$ chain of $(-2)$ curves, this is \cite[Proposition 1.6 and Lemma 4.1]{IU}. When $\tilde{Z}$ is a cycle of $(-2)$ curves, we want to use \cite[Proposition 4.4]{Uehara1}. This is stated for the case where $\tilde{Z}$ ($Z_0$ in the notation in \cite{Uehara1}) is a fibre of a relatively minimal rational elliptic fibration on $Y$. (Note this would imply that $D$ is semi-definite, cf.~start of Section \ref{sec:neg_indef_auteq}.)
However, inspecting the proof of \cite[Proposition 4.4]{Uehara1}, one can check that relative minimality is not required for that statement: it's enough to check that there exists an elliptic fibration $f: Y \to \bP^1$ such that $\tilde{Z}$ is a fibre of $f$. Say $\tilde{Z}$ has components $C_1, \ldots, C_l$. 

First, by the adjunction formula, we have that
$\omega_{\tilde{Z}}=\omega_Y \otimes \cO_Y(\tilde{Z}) |_{\tilde{Z}}$. Also,   the dualizing sheaf $\omega_{\tilde{Z}}$ is trivial, and $\omega_Y|_{\tilde{Z}} = K_Y|_{\tilde{Z}}=-D|_{\tilde{Z}} =0$. (Here we use $Y=Y_e$ in the case where $\tilde{Z} = D$.) Thus the normal bundle $\cO_Y(\tilde{Z})|_{\tilde{Z}}$ of $\tilde{Z}$ in $Y$ is trivial. 
It then follows from the exact sequence
$$
0 \to \cO_Y \to \cO_Y(\tilde{Z}) \to  \cO_Y(\tilde{Z})|_{\tilde{Z}} \to  0
$$
that $\dim H^0(\cO_Y(\tilde{Z}))=2$ (note $H^1(\cO_Y)=0$ because $Y$ is rational). Thus the linear system $|\tilde{Z}|$ gives a rational map to $\bP^1$. 

Next, we claim that $|\tilde{Z}|$ is basepoint free. Write $\tilde{Z}=M+F $ for the moving and fixed parts of the divisor $\tilde{Z}$.
Note that $M \neq 0$ since $\dim H^0(\cO_Y(\tilde{Z}))=2$. 
The lattice $\langle C_1, \ldots, C_k \rangle $ is negative semidefinite with kernel $\langle \tilde{Z} \rangle$. 
Thus if $F \neq 0$,
 then $M^2<0$, which  
contradicts $M$ moving. 
So $F=0$ and $\tilde{Z}=M$ is moving. Now $\tilde{Z}^2=0$ implies that $|\tilde{Z}|$ is basepoint free. Thus the linear system $|\tilde{Z}|$ gives a morphism to $\bP^1$, with one fiber equal to $\tilde{Z}$.
(This is essentially the same as the proof of the analogous result for K3 surfaces.)
\end{proof}

\begin{remark}
In general, there could be infinitely many $(-2)$ curves, and also other finite configurations -- e.g.~whenever $k=3$, a nodal configuration with Y-shaped dual graph  (mirror to the Milnor fibre of a hypersurface cusp).

When $D$ is indefinite, the only possible configurations of $(-2)$ curves are ADE configurations (and finitely many of them): there's a birational morphism which contracts all the internal $(-2)$ curves to ADE singularities and is an isomorphism elsewhere \cite[Lemma 6.9]{GHK1}. 
\end{remark}

We can completely classify classes of sphericals in $K(D(Y))$. First recall that the set of roots $\Phi$ consists of classes in $\Pic(Y)$  obtained by parallel transport from an internal $(-2)$ curve in a deformation equivalent pair $(Y',D')$. (Alternatively, consider the generic ample cone $\cA_\gen \subset \Pic Y \otimes \bR$; the roots are the $(-2)$ classes $\delta$ such that $\langle \delta \rangle^\perp \cap \cA_\gen \neq \emptyset $, see \cite{Friedman}.)
The Weyl group, denoted by $W$, is the group generated by reflections $s_{\alpha}(x)=x+\langle x,\alpha \rangle\alpha$ for $\alpha \in \Phi$, acting on $H^2(Y; \bZ)  = \Pic(Y)$. The set of \emph{simple roots} $\Delta$ consists of classes of $(-2)$ curves $C \subset Y_e \setminus D$. By  \cite{GHK2}, $W$ is generated by the reflections in elements of $\Delta$, and $\Phi = W \cdot \Delta$.
As the spherical twist in $i_\ast \cO_C(a)$ lifts the action of $s_{[C]}$, any root is the first Chern class of a spherical object in $D(Y)$. The converse also holds:

\begin{lemma}\label{lem:roots_and_sphericals}
Suppose $E \in D(Y)$ is a spherical object with support disjoint from $D$. Then $c_1(E) \in H^2(Y; \bZ) \simeq \Pic(Y)$ is a root.
\end{lemma}

\begin{proof}
Let $E$ be a spherical object in $D(Y)$ with support disjoint from $D$. Consider framed deformations of $E$ with framing along $D$. 
These are controlled by $\Ext^i(E,E(-D))$, $i=0,1,2$. 
(Note that, since the support of $E$ is disjoint from $D$, the framed deformations are identified with ordinary deformations, and $\Ext^i(E,E(-D))=\Ext^i(E,E)$; 
however, this is not true for deformations of the associated determinant line bundle considered below.) 

Consider the trace map from $\Ext^i(E,E(-D)) $ to $H^i(\cO_Y(-D))$. 
This corresponds to the map from framed deformations of $E$ to framed deformations of the line bundle $\det E$. 
The map $$\Ext^2(E,E(-D)) \rightarrow H^2(\cO_Y(-D))$$ is an isomorphism, 
that is, the obstruction spaces for deformations of $E$ and framed deformations of $\det E$ are identified. 
Indeed, recalling that $-D=K_Y$, the map is Serre dual to the map $H^0(\cO_Y) \rightarrow \Hom(E,E)$, which is an isomorphism since $E$ is spherical. 
(This is the usual argument for smoothness of moduli spaces of stable sheaves on Calabi--Yau surfaces due to Mukai, adapted to the log Calabi--Yau / framed case.) 

This implies that if $(\cY,\cD)/(0 \in S)$ is a deformation of $(Y,D)$ over an analytic germ $(0 \in S)$, 
then $E$ extends to an object $\cE$ in $D^b(\Coh \cY)$ iff there is a framed deformation of $\det E$ over $S$. 
Equivalently, writing $\cL_s$ for the line bundle on $\cY_s$ obtained from $\det E$ by parallel transport (noting that $c_1 \colon \Pic \cY_s  \rightarrow H^2(\cY_s,\bZ)$ is an isomorphism since $\cY_s$ is rational), we have $\cL_s |_{\cD_s} \simeq \cO_{\cD_s}$ for all $s \in S$. 
In the terminology of \cite{GHK2}, after choosing a trivialization of the local system over $S$ with fibers $H^2(\cY_s,\bZ)$, 
we require that the period point $\phi_{\cY_s} \colon H^2(Y,\bZ) \rightarrow \bC^*$ satisfies $\phi_{\cY_s}(c_1(E))=1$ for all $s \in S$. 
(Moreover, in this case the extension $\cE$ is uniquely determined since $\Ext^1(E,E)=0$.) 
Write $\alpha=c_1(E) \in H^2(Y,\bZ)$. Note that $\alpha^2=-2$ by the Riemann-Roch formula; in particular, $\alpha \in H^2(Y,\bZ)$ is primitive.

By the local Torelli theorem for log Calabi--Yau surfaces \cite{Looijenga}, 
there is a small deformation $(\cY_s,\cD_s)$ of $(Y,D)$ such that $\ker \phi_{\cY_s} = \bZ \cdot \alpha$. 
Thus $E$ deforms to a spherical object $\cE_s$ on $\cY_s$. 
(A deformation of a spherical object is necessarily spherical by upper semicontinuity of coherent cohomology and the topological nature of the Euler
characteristic.) 
Recall that any spherical object in $D(Y)$ has support a union of $(-2)$ curves.
Moreover, if $C \subset Y \setminus D$ is an internal curve,  then $\phi_{Y}([C])=1$. 
Finally, the support of $\cE_s$ is disjoint from the boundary $\cD_s$ of $\cY_s$ for all $s$ (since this is an open condition and it holds for $s=0$). 
It follows that $\cY_s \setminus \cD_s$ contains a unique $(-2)$ curve $C$ with class $\pm \alpha= \pm c_1(E)$. Thus $c_1(E) \in H^2(Y,\bZ)$ is a root.\end{proof}

\begin{remark}
\label{rmk:classes_sphericals}
The $K$-theory $K(\Coh Y)$ is generated by the image of $\Coh(Y)$ under the Chern character. 
Explicitly, 
$$K(\Coh Y)=\{(r,c_1,\ch_2) \in \bZ \oplus H^2(Y,\bZ) \oplus \frac{1}{2}\bZ \ | \ ch_2 \equiv \frac{1}{2}c_1^2 \bmod \bZ\}.$$
(As $Y$ is rational, the algebraic and topological $K$-theories for $Y$ agree.)
The classes of the spherical objects in $K(\Coh Y_e)$ are given by $(0,c_1,\ch_2)$ where $c_1 \in \Phi$ is a root and $\ch_2 \in \bZ$.
Above we only considered the first Chern class $c_1$, but the same argument shows this stronger result. (Note that the spherical twist acts by the same formula on $K(\Coh Y)$ and $\ch(\cO_C(n))=\ch(\cO_C)+(0,0,n)$ using the basis above.)
\end{remark}

\subsection{Automorphisms of $Y$} We collect relevant results from \cite{GHK2, Friedman}.

\subsubsection{Generalities and Torelli theorem}
We consider the following subgroups of the group of biholomorphic automorphisms $\Aut(Y)$: let $\Aut(Y,D)$ be the subgroup preserving $D$, $\Aut(Y,D; \text{cpt})$, the subgroup preserving each component of $D$; and $\Aut(Y,D; \text{pt})$, the subgroup fixing $D$ pointwise. ($\Aut (Y,D; \text{cpt})$ clearly has (at worst) finite index in $\Aut(Y,D)$.)

If $D$ is negative definite, $\Aut(Y) = \Aut (Y,D)$; in the negative semi-definite case, $\Aut(Y,D)$ has finite index in $\Aut(Y)$.  In either of these cases, provided $D$ contains no $(-1)$ curves, any automorphism of $U = Y \backslash D$ compactifies to an element of $\Aut(Y)$: this is an straightforward exercise in the semi-definite cases, and follows from uniqueness of the minimal resolution of the corresponding cusp singularities for the negative definite ones. 
(Note that this is emphatically not true in the indefinite case, see e.g.~\cite{Cantat-Loray}.)

Let 
$\Adm$ denote the subgroup of automorphisms of the lattice $\Pic(Y)$ preserving the classes $[D_i]$ and the generic ample cone in $\Pic (Y)$. (See \cite[Definition 1.7]{GHK2};  this is called $\Gamma(Y,D)$ in \cite{Friedman}.) $\Adm$ is the monodromy group of pairs $(Y,D)$ in the sense of \cite[Theorem 5.15]{GHK2}. We have  $W  \unlhd \Adm$.

\begin{theorem}\label{thm:GHK_Torelli} \cite[Theorem 5.1]{GHK2}. Assume $Y$ has the distinguished complex structure within its deformation class. 
There is an exact sequence
$$
1 \to N \to \Aut(Y,D; \text{cpt}) \to \Adm/W \to 1
$$
where $N = \Hom (N', \bG_m)$, and $N'$ the cokernel of the evaluation map $\ev: \Pic(Y) \to \bZ^k$ given by taking intersection numbers with the $D_i$. Equivalently, $N' = \pi_1 (Y \backslash D)$, by considering the long exact sequence of the pair $(Y, U)$.
\end{theorem}

$N$ is the subgroup of $\Aut(Y,D; \text{cpt})$ which acts trivially on $H^2(Y; \bZ)$. It is finite whenever a toric model for $Y$ involves interior blow ups on components corresponding to (at least) two linearly independent toric rays; in particular, it is finite whenever $D$ is negative semi-definite or negative definite. 
The action of $N$ on $(Y,D)$ can be described explicitly as follows: choose a toric model $(Y,D) \to (\bar{Y},\bar{D})$. Let  $T \simeq (\bC^*)^2$ be the big torus acting on  $\bar{Y}$. Then $N$ is the subgroup of $T$ fixing the points we blow up to obtain $(Y,D)$, and the action of $N$ on $(\bar{Y},\bar{D})$ lifts to an action of $N$ on $(Y,D)$. 

\begin{lemma} 
\label{lem:AutYDpt} Assume that $Y = Y_e$. 
The exact sequence of Theorem \ref{thm:GHK_Torelli} splits, and we have
$$\Aut(Y,D; \text{pt}) = \ker \{ \Aut(Y,D; \text{cpt})  \to \Aut D \} \simeq \Adm / W $$
\end{lemma}

\begin{proof} 
From the discussion above, we get $N \subset T \subset \Aut^0 (D) ( \simeq (\bC^\ast)^k)$. Now consider the restriction map $\Aut(Y,D; \text{cpt}) \to \Aut^0 (D)$. We claim that its image lies in $N$. This uses the fact that $Y = Y_e$: for any $L \in \Pic (Y)$, $L|_D = \cO_D (\sum (L \cdot D_i) p_i)$, for some distinguished points $p_i \in D_i$. Now consider an arbitrary $\varphi \in  \Aut(Y,D; \text{cpt}) $. It follows from the above that $ \varphi|_D^\ast (L|_D) \simeq \varphi^\ast (L) |_D  \simeq  L|_D  $. On the other hand, as this holds for an arbitrary $L$, we see that $ \varphi|_D \in N$ by \cite[Proposition 2.6]{GHK2}, as required. Thus restriction gives a splitting $ \Aut(Y,D; \text{cpt})  \to N$, and we're done.
\end{proof}

\subsubsection{Negative semi-definite case}\label{sec:Mordell-Weil} 

 Assume that $D$ is negative semi-definite and does not contain a $(-1)$ curve, i.e.~$D$ is a cycle of $(-2)$ curves. Complex deformation classes of such pairs $(Y,D)$ are completely understood: 

\begin{proposition}\cite[Propositions 9.15 and 9.16]{Friedman} Suppose $(Y,D)$ is a log CY surface, and that $D$ is a cycle of $k$ $(-2)$ curves. Then $k \le 9$, there is a unique deformation type for $k \neq 8$, and two deformation types for $k=8$ distinguished by $\pi_1(Y\backslash D)$.
\end{proposition}

\begin{remark}
$Y \backslash D$ is mirror to $X \backslash E$ where $X$ is a del Pezzo surface of degree $K_X^2=k$ and $E$ is a smooth anticanonical divisor \cite{AKO_delPezzo, HK}.
\end{remark}

The various flavours of automorphism groups are also known explicitly. First, $N$ is known, see \cite[Proposition 9.16 (iii)]{Friedman}:
$$
\pi_1 (Y \backslash D) =
\begin{cases}
 \bZ/2 \mbox{ or } 0& 
\text{ when }k=8\\
\bZ/3 & 
\text{ when } k=9 \\
0 &  \text{otherwise} \\
\end{cases}
$$
Second, by \cite[Corollary 9.20]{Friedman},
$\Adm = W$ except when $k=7$, or $k=8$ and $\pi_1(Y \backslash D)=0$. In both of these cases, $\Adm / W \simeq \bZ$, see  \cite[Example 5.6]{GHK2}, \cite[Example 5.3]{GHK2} and \cite[Lemma 9.18 (iii)]{Friedman}.

Thus if $Y=Y_e$, in both of these cases,  we get that $\Aut(Y,D; \text{pt}) \simeq \bZ$. 
We can describe these automorphisms explicitly using the Mordell-Weil group of $Y$. As $Y= Y_e$, it is the total space of a minimal rational elliptic fibration $\varpi: Y \to \bP^1$, with $D = \varpi^{-1} (\infty)$ cycle of $(-2)$ curves. (The existence of such an elliptic fibration is codimension one in the complex moduli space: it exists whenever the period map $\phi$ satisfies $\phi ([D])=1$, which is certainly true for $Y_e$. For background on $\phi$, see e.g.~\cite[Section 2.2]{HK}.)  The fibration $\varpi$ is unique, and admits a holomorphic section. 
 Recall that the Mordell-Weil group  $MW(Y, \bP^1)$ is the group of holomorphic sections of $\varpi$, with identity element the reference section $s$. This acts by translation on each smooth fibre. As the fibration is relatively minimal, this extends to an automorphism of $Y$, cf.~\cite[Proposition III.8.4]{BHPV}.
From \cite[Theorem 1.3]{Shioda}, we have 
$$MW(Y, \bP^1) \simeq \langle F \rangle^\perp / K$$
where $F$ is the class of a smooth fibre of $\varpi$, $\langle F \rangle^\perp$ is taken inside $ \Pic(Y) = H^2(Y; \bZ) $ with respect to the standard pairing, and $K$ is the subgroup of $H^2(Y; \bZ)$ generated by all the irreducible components of fibres. 
One can calculate that $MW(Y, \bP^1) \cong \bZ$ when  $k=7$, or $k=8$ and $\pi_1(Y\backslash D) = 0$, in which cases $\Aut(Y,D; \text{pt}) =  k\cdot MW(Y, \bP^1)$. In other cases it's a finite group. 

\subsection{Autoequivalences of $D^b \Coh(Y)$} 

\subsubsection{$D$ negative definite or indefinite} \label{sec:neg_indef_auteq}

\begin{lemma}\cite{Uehara2}
Assume that $D$ is negative definite or indefinite. Then any autoequivalence of $D(Y)$ has Fourier--Mukai kernel with  two-dimensional support. 
\end{lemma}

\begin{proof}
We claim that these cases both fall under (iii) of \cite[Theorem 5.3]{Uehara2}: namely that $K_Y \neq 0$, and $Y$ does not admit a minimal elliptic fibration. The second part needs checking.  Suppose that $Y$ admits a minimal elliptic fibration. Since $Y$ is rational, there is at most one multiple fiber, say of multiplicity $m \ge 1$, by the Kodaira canonical bundle formula.
Then $D=-K_Y = \frac{1}{m}F$, where $F$ is the class of a fibre. Now $D \cdot F =0$ so $D$ is contained in a fiber, and $D^2=0$ so $D$ is equal to a fiber, and thus $D$ is semi-definite, a contradiction.

\end{proof}

As before, let $Z$ be the union of all $(-2)$ curves in $Y$, and let $D_Z(Y)$ be the full triangulated subcategory of $D(Y)$ consisting of objects supported on $Z$. By \cite[Proposition 6.1]{Uehara2}, there is a well-defined group homomorphism
$$
\Auteq D(Y) \to \Auteq D_Z(Y)
$$
induced by restriction. Let $\Auteq^{\dagger} D_Z(Y)$ denote the image of this map. (This is $\Auteq^{\dagger}_{\text{K-equiv}} D_Z(Y)$ in \cite{Uehara2}, though in the case at hand all autoequivalences are of so-called $K$-equivalent-type.)

\begin{theorem}\cite[Theorems 6.6 and 6.8]{Uehara2} \label{thm:neg_def_auteq}
There is a short exact sequence
\begin{equation}\label{eq:Auteq_SES}
1 \to \Pic_Z(Y) \rtimes \Aut_Z(Y) \to \Auteq D(Y) \to \Auteq^{\dagger} D_Z(Y) \to 1
\end{equation}
where $\Pic_Z(Y) = \{ \cL \in \Pic(Y) \, | \, \cL|_Z \simeq \cO_Z \} $, acting by tensor product, and $\Aut_Z(Y) = \{ \varphi \in \Aut (Y) \, | \, \varphi|_Z = \text{id}_Z \}$, acting by pushforward.  

If $Z$ is a disjoint union of cycles and chains of $(-2)$ curves, and $Y=Y_e$ if $D$ is a cycle of $(-2)$ curves, then
$$
\Auteq D(Y) = \langle \Br_Z(Y) , \Pic(Y) \rangle \rtimes \Aut(Y) \times \bZ[1]
$$ 
by Lemma \ref{lem:type_A_sphericals},
where 
$
\Br_Z = \langle T_{\cO_C(a)}\, | \, C \subset Z \text{ a } (-2) \text{ curve}, a \in \bZ \rangle 
$,
and $\bZ[1]$  denotes shifts.
\end{theorem}

\subsubsection{$D$ negative semi-definite}  
\label{sec:semi_def_D}
Assume that there is a  rational elliptic fibration  $\varpi: Y \to \bP^1$ as above (e.g.~if $Y=Y_e$), and let  $F$ denote any fibre of $\varpi$. Every element of $\Auteq D(Y)$ induces an automorphism of $(K(Y), \chi)$, where $\chi$ is the Mukai pairing, and, via restriction, an automorphism of $(K(F), \chi)$: see the proof of \cite[Theorem 3.11]{Uehara1} and lemmas leading up to it. 
(For $F$ we are referring to topological $K$-theory.)
 $K(F)$ is identified with $H^0(F, \bZ) \oplus H^2(F, \bZ)$ by taking rank and degree, and, as an element of $GL_2(\bZ)$, an automorphism preserves $\chi$ precisely when it has determinant one. Thus we get a map $\Theta: \Auteq D(Y) \to SL_2(\bZ)$. 

\begin{theorem}\cite[Theorem 4.1]{Uehara1}\label{thm:auto_elliptic} 
Assume that all the reducible fibres of $\varpi$ are of type $I_m$ (i.e.~cycles of $(-2)$ curves). Then there is a short exact sequence
\begin{equation}
1 \to \langle \Br_Z, \otimes \cO_Y(E) \, | \, E \cdot F = 0 \rangle \rtimes \Aut Y \times \bZ[2]
\to \Auteq D(Y)
\xrightarrow{\Theta} 
SL_2(\bZ) \to 1 \label{eq:auteq_elliptic}
\end{equation}
where $\Br_Z = \langle T_{\cO_C (a)} \, | \, C \subset Y \text{ a }(-2) \text{ curve}, a \in \bZ \rangle$, and $\bZ[2]$ is even degree shifts. 
\end{theorem}
Note that all $(-2)$ curves are contained in reducible fibres, and so Lemma \ref{lem:type_A_sphericals} applies. When $\varpi$ has singular fibres of other types, a similar short exact sequence is conjectured in \cite{Uehara1}, and known up to the same ambiguity as in the negative definite case: `extra' autoequivalences with Fourier--Mukai kernel with support on $Z$ have not been ruled out.

The degree one shift acts as $-1$ on $K$ theory, and so is not in the kernel of the map to $SL_2 (\bZ)$.
The sequence \ref{eq:auteq_elliptic} doesn't split; this will be clear from the perspective of moduli space monodromies, see Section \ref{sec:monodromies}. Explicit lifts of $SL_2(\bZ)$ elements can be given as follows.  Use the basis $([\cO_F], [\cO_{pt}])$ for $K(F) \simeq \bZ^2$, and let $t= \begin{pmatrix} 1 & 0  \\ 1 & 1 \end{pmatrix}$ and $s = \begin{pmatrix} 0 & 1  \\ -1 & 0 \end{pmatrix}$ be the usual generators for $SL_2(\bZ)$. Let $E_0$ be a holomorphic section of $\varpi$. Then $\otimes \cO_Y  (E_0) \in \Auteq D(Y)$ satisfies $\Theta (\otimes \cO_Y  (E_0)) = t$. Let $\Delta \subset Y \times_{\bP^1} Y$ be the relative diagonal, $\pi_i: Y \times_{\bP^1} Y \to Y$ be projection to the $i$th factor, and let $\cP$ be the relative Poincar\'e bundle associated to $E_0$:
$$
\cP = \cO_{Y \times_{\bP^1} Y} (\Delta - \pi^\ast_1 E_0 - \pi^\ast_2 E_0).
$$
Let  $\Phi^\cP$ be the fibrewise Fourier-Mukai transform with kernel $\cP$; then $\Theta (\Phi^\cP) = s$, see e.g.~\cite{Bridgeland-elliptic}.

\subsubsection{Detecting the identity} 

\begin{proposition}\label{prop:detecting_id}
Assume that $(Y,D)$ is arbitrary negative definite or indefinite; or semi-definite admitting a rational elliptic fibration (e.g.~$Y=Y_e$). 
 Assume that $\psi \in \Auteq D(Y)$ induces the identity on $K(Y)$, and that for each $(-2)$ curve $C \subset Y$, $\psi$ fixes $i_\ast \cO_C(-1)$ and $i_\ast \cO_C$ as objects of $D(Y)$. Then $\psi = \varphi_\ast$, where $\varphi \in N \leq \Aut(Y,D; \text{cpt})$, up to an even shift if $Z$ is empty.
In particular, $\psi$ must be the identity autoequivalence if in addition  $\iota^\ast  \circ \psi = \iota^\ast$, where $\iota$ is the inclusion $D \to Y$, and $\iota^\ast$ the derived pullback $D(Y) \to \Perf(D)$. 
\end{proposition}

\begin{proof}

Let's first assume that $D$ is either negative definite or indefinite.
Let $C$ be an arbitrary $(-2)$ curve.
 The structure sheaves  $i_\ast \cO_p$ of points $p \in C$ are in one-to-one correspondence with cones on morphisms from  $i_\ast \cO_C(-1)$ to $i_\ast \cO_C$, so the second assumption implies that $\psi$ permutes these structure sheaves. 
On the other hand,  the proof of \cite[Theorem 6.6]{Uehara2} implies for any $q \in Y \backslash Z$, $i_\ast \cO_q$ must be mapped to $i_\ast \cO_{q'}$, some $q' \in Y \backslash Z$. 
From \cite[Corollary 5.23]{Huybrechts}, it follows that $\psi$ is in $\Pic(Y) \rtimes \Aut(Y)$. Now use the assumption on the $K$-theoretic action: first, considering the image of $[\cO_Y] \in K(Y)$, we see that $\psi = \varphi_\ast$, some $\varphi \in \Aut(Y)$; second, the Torelli theorem (Theorem \ref{thm:GHK_Torelli}) then implies that $\varphi \in N$.

When $D$ is negative semi-definite, as before, $\psi$ must permute structure sheaves of points of $Z$.
By the assumption on $K$-theory, $\psi \in \ker \Theta$.
We want to show that $\psi$ is in $\Pic Y \rtimes \Aut Y$.
After composing with an automorphism of $\bP^1$, $\psi$ restricts to an autoequivalence of $Y \backslash Z$ over $V=\bP^1 \backslash \varpi(Z)$ (note here $Z$ is the union of the reducible fibers). Now since $\psi \in \ker(\Theta)$ for any smooth fibre $F \subset Y \backslash Z$ the induced autoequivalence of $D(F)$ is a composition of an automorphism, tensor by a line bundle, and even shift by the description of autoequivalences of $D(F)$ for $F$ an elliptic curve; in particular for any $x \in Y \backslash Z$, $\psi(\cO_x)=\cO_y[2m]$ for some $y \in Y \backslash Z$ and $m \in \bZ$. Finally $m=0$ if $Z \neq \emptyset$ (since $\psi(\cO_x)$ is a sheaf for $x \in Z$ and openness of this property).
 The remainder of the argument is then similar to the  first case. 

The criterion for $\psi = \Id$ is immediate from the characterisation of $N$ in Lemma \ref{lem:AutYDpt}. 
\end{proof}

We record relations between different categories of known autoequivalences, for later reference:

\begin{lemma}\label{lem:autoequivalence_relations} 
Assume that $\varphi \in \Aut(Y)$, $\cL \in \Pic (Y)$, and  $\mathcal{S}$ is a spherical object in $D(Y)$. Let $T_\mathcal{S}$ denote the spherical twist in $\mathcal{S}$. Then
\begin{enumerate}
\item $\varphi^\ast \circ T_{ \mathcal{S}} = T_{\varphi^\ast \mathcal{S}} \circ \varphi^\ast  $
\item $ (\underline{\phantom{...}} \otimes \cL) \circ T_\mathcal{S} =   T_{\mathcal{S} \otimes \cL}  \circ (\underline{\phantom{...}} \otimes \cL) $
\item $\varphi^\ast \circ (\underline{\phantom{...}} \otimes   \cL) = (\underline{\phantom{...}} \otimes \varphi^\ast  \cL )  \circ \varphi^\ast $
\item $T_{\cO_C (a-1)} \circ T_{\cO_C (a)} = \underline{\phantom{...}}\otimes \cO_Y (C)$, where $C$ is a $(-2)$ curve and $a \in \bZ$. 
\end{enumerate}
\end{lemma}

\begin{proof}
These are all immediate apart from the final one, for which the calculation is in \cite[Lemma 4.15 (i) (2)]{IU}. 
\end{proof}
Note also that $\Br_Z \cap \Aut (Y) = \Pic (Y) \cap \Aut (Y) = \{ \Id \} \in \Auteq (Y)$, see \cite[Lemma 4.14]{IU}.

\subsection{`Compactly supported' autoequivalences of $D^b \Coh(U)$} \label{sec:compact_support}

Recall that $D(U)$ is equivalent to the derived wrapped Fukaya category of the mirror, $D^b \cW(M)$. 
 On that side, any compactly supported symplectomorphism of $M$ induces an automorphism of $D^b \cF^{\to} (w)$ (the mirror to $D(Y)$),  and the identity on the fibre $\Sigma$ of $w$, whose Fukaya category $\cF(\Sigma)$ is mirror to $\Perf (D)$. This motivates the following definition:

\begin{definition}\label{def:compact_support}
We say that  $\psi \in \Auteq D(U)$ has compact support if there exists $\widetilde{\psi} \in \Auteq D(Y)$ such that
  $\widetilde{\psi}$ restricts to $\psi$ on $D(U)$, i.e.~we have a commutative diagram:
$$
\xymatrix{
D(Y) \ar[r]^{\widetilde{\psi}} \ar[d]^{i^\ast} & D(Y) \ar[d]_{i^\ast} \\
D(U) \ar[r]^{\psi} & D(U) \\
}
$$
and moreover the following diagram commutes: 
$$
\xymatrix{
D(Y) \ar[rr]^{\widetilde{\psi}} \ar[rd]^{i^\ast} && D(Y) \ar[ld]_{i^\ast} \\
& \Perf(D) & 
}
$$

Let $\Auteq_c D(U)$ denote the subgroup of compactly supported autoequivalences. 
\end{definition}

\begin{definition}
Let $\gamma \subset H_2 (U, \bZ)$ be the class of of a ``crossing torus'' at a node $p$ of $D$: identifying the germ $(p \in D \subset Y)$  with $(0 \in (z_1z_2=0) \subset \bC^2_{z_1,z_2})$, $\gamma$ is the class of the $2$-torus  $(|z_1|=|z_2|=\epsilon)$ for $0< \epsilon \ll 1$.  

\end{definition}

A choice of sign for $\gamma$ is the same as a choice of generator of $H_1(D,\bZ)$. 
Note that $\gamma$ generates the kernel of the map $i_* \colon H_2(U,\bZ) \rightarrow H_2(Y,\bZ)$, where $i \colon U \rightarrow Y$ denotes the inclusion of $U$ in $Y$.

Let $Q =  \langle D_1, \ldots, D_k \rangle^\perp \subset \Pic(Y)$. 
 Assuming $Y = Y_e$, we have that 
$$
Q = \{ \mathcal{L} \in \Pic(Y) \, | \, \mathcal{L}|_D \simeq \cO_D \}  \simeq  \text{Im} (H_2(U) \to H_2(Y)) \simeq H_2(U) / \bZ \cdot \gamma. 
$$
Let $\bar{Q}$ be the image of $Q$ in $\Pic(U)$. This is isomorphic to $Q$ unless $D$ is semi-definite, in which case we have a short exact sequence
$$
0 \to \langle D \rangle \to Q \to \bar{Q} \to 0.
$$

Choosing (for instance) a section $s$ of $\w: Y_e \to \bP^1$ determines a splitting of this sequence: 
\begin{equation}\label{eq:Qbar_split}
Q \simeq \langle D \rangle \oplus (  \langle D_1, \ldots, D_k \rangle^\perp \subset  \langle D, s \rangle^\perp ).
\end{equation}

\begin{conjecture}\label{conj:compact_generators} $\Auteq_c D(U)$ is generated by $\bar{Q}$, $\Aut(Y,D; \text{pt})$ and $\Br_{Z'}$, where $Z' = Z \cap U$. 
\end{conjecture}

Whenever $Z$ consists of disjoint chains and cycles of $(-2)$ curves, this follows from Theorems \ref{thm:neg_def_auteq} and \ref{thm:auto_elliptic}. In general, the conjecture is known up to rulling out `extra' autoequivalences with support on $Z$ (i.e.~restricting trivially to $Y \backslash Z$).

\section{Mirror symmetry for log Calabi--Yau surfaces: background}\label{sec:hms_background}

\subsection{Homological mirror symmetry}\label{sec:hms_background1}

Recall the following. 

\begin{theorem}\label{thm:hms}\cite[Theorem 1.1]{HK}
Suppose $(Y,D)$ is a log Calabi--Yau surface with maximal boundary, and distinguished complex structure. Then there exists a four-dimensional Weinstein domain $M$ and a Lefschetz fibration $w: M \to \bC$, with fibre $\Sigma$, such that:
\begin{enumerate}
\item $\Sigma$ is a $k$--punctured elliptic curve, where $k$ is the number of irreducible components of $D$; there is a quasi-equivalence $D^\pi \cF(\Sigma) \simeq \Perf(D)$, due to Lekili--Polishchuk \cite{Lekili-Polishchuk}, where $\cF(\Sigma)$ is the Fukaya category of $\Sigma$, with objects compact Lagrangian branes. (These are graded categories; in particular, there is a natural choice of line field on $\Sigma$.)

\item $D^b \cF^{\to} (w) \simeq D(Y)$, where $\cF^{\to} (w) $ is the directed Fukaya category of $w$;

\item $D^b \W(M) \simeq D(Y \backslash D)$, where $\W(M)$ is the wrapped Fukaya category of $M$.

\end{enumerate}
\end{theorem}

\begin{remark}
When $D$ is semi-definite, i.e.~a cycle of $k$ $(-2)$ curves, $M$ is the Milnor fibre of a rational elliptic singularity, and we have $M = X \backslash E$, where $X$ is a degree $k$ del Pezzo surface, and $E \subset X$ an anti-canonical divisor given by a smooth elliptic curve. \cite{AKO_delPezzo} proves HMS for these spaces with the $A$ and $B$ sides reversed. 
\end{remark}

Any log CY pair $(Y,D)$ can be described by the data of a toric model \cite[Proposition 1.3]{GHK1}. We follow \cite[Definition 2.2]{HK}. A
 choice of toric model is given by
$$
(\bar{Y}, \bar{D}) \leftarrow (\widetilde{Y}, \widetilde{D}) \to (Y,D)
$$
where $(\bar{Y}, \bar{D})$ is a toric pair, the map $(\bar{Y}, \bar{D}) \leftarrow (\widetilde{Y}, \widetilde{D})$ is given by interior blow ups, and $(\widetilde{Y}, \widetilde{D}) \to (Y,D)$ is given by blowing down a sequence of components of $ \widetilde{D}$ (in other words, one goes from $(Y,D)$ to $(\widetilde{Y}, \widetilde{D})$ by performing a sequence of `corner'  blow-ups).
Varying the blow-up locus for interior blow-ups deforms the complex structure. For the distinguished one, we blow up a single favourite point on each component of $\bar{D}$, which is determined by the torus action. (See \cite[Section 2.2]{HK}. Fixing a torus action, one can take the point $-1 \in \bar{D}_i \backslash \cup_{j \neq i} \bar{D}_j \simeq \bC^\ast$.) When $(\widetilde{Y},D) = (\widetilde{Y}_e, D)$, $(\widetilde{Y}, \widetilde{D}) \to (\bar{Y}, \bar{D})$ is entirely determined by numerical data:
\begin{itemize}
\item the self-intersection numbers of the $\bar{D}_i$, say $n_1, \ldots, n_k$;
\item the number of interior blow-ups of each component, say $m_1, \ldots, m_k$.
\end{itemize}
Note also that there's a natural isomorphism $\widetilde{Y}_e \backslash \widetilde{D} \simeq Y_e \backslash D$; they have the same mirror $M$, with different choices of superpotential for the two different compactifications (these both give Lefschetz fibrations on $M$, which are related by a sequence of stabilisations mirror to the sequence of corner blow-ups, see  \cite[Proposition 3.11]{HK}).

Assume $\widetilde{Y}= \widetilde{Y}_e$. Given the data of a pair $f: (\widetilde{Y}, \widetilde{D}) \to (\bar{Y}, \bar{D})$, we have the following:
\begin{itemize}
\item[(1)] a full exceptional collection for $D(\widetilde{Y})$, namely
\begin{multline*}
\cO_{\Gamma_{km_k}}(\Gamma_{km_k}), \cdots, \cO_{\Gamma_{k1}}(\Gamma_{k1}), \cdots, \cO_{\Gamma_{1m_1}}(\Gamma_{1m_1}),\cdots, \cO_{\Gamma_{11}}(\Gamma_{11}), \cO,
f^\ast \cO(\bar{D}_1), \\ \cdots, f^\ast\cO(\bar{D}_1+ \cdots + \bar{D}_{k-1})
\end{multline*}
where $\Gamma_{ij}$ is the pullback of the $j$th exceptional curve over $\bar{D}_i$, for $i=1, \ldots, k$, $j=1, \ldots, m_i$.

\item[(2)] a Lefschetz fibration $w: M \to \bC$ with central fibre a $k$ punctured elliptic curve, and distinguished collection of vanishing cycles
$$
\{ W_{ij} \}_{i=1, \ldots, k, j = 1, \ldots, m_i}, V_0, \ldots, V_{k-1}
$$
which correspond under Theorem \ref{thm:hms} to the full exceptional collection above. 
We use the notation $\vartheta_{ij}$ for the Lefschetz thimble corresponding to $W_{ij}$. 

\item[(3)] a Weinstein deformation equivalence between $M$ and the total space of an explicit almost-toric fibration, below. 
\end{itemize}

\begin{theorem}\label{thm:ATstructure}\cite[Theorem 1.2]{HK}
Let $(Y,D)$ be a log Calabi--Yau surface with maximal boundary and  distinguished complex structure. 
Fix a toric model for $(Y,D)$, with notation as before. 
Let $v_i$ be the primitive vector for the ray associated to $\bar{D}_i$ in the fan of $\bar{Y}$. 
 Then $M$, the mirror space in Theorem \ref{thm:hms}, 
is Weinstein deformation equivalent to the total space of an almost-toric fibration with base a two-dimensional integral affine space; smooth fibres Lagrangian two-tori; and  a nodal fibre for each of the interior blow-ups on $\bar{D}_i$, with invariant line in direction $v_i$. All invariant lines are concurrent. (Note this implies they are cyclically ordered by the indices of the $v_i$.)
\end{theorem}
Denote by $\gamma^\vee$ the class of a fibre of the almost-toric fibration.

An explicit identification is constructed in \cite[Section 5.2]{HK} for the exact torus mirror to the $(\bC^*)^2$ chart `inherited' from $\bar{Y} \backslash \bar{D}$, and then \cite[Section 6.1--2]{HK} for the nodal fibres. 
We often refer to these as `Symington almost-toric fibrations', owing to \cite{Symington}. 
One should think of them as SYZ fibrations. The integral affine bases  already appear explicitly in \cite{GHK1}. 
As a Weinstein handle-body, $M$ is given by attaching two-handles to $D^\ast T^2$; there is one two-handle for each interior blow-up of a component $\bar{D_i}$ of $\bar{D}$. The attaching Legendrian in $S(D^\ast T^2)$ is given as follows: letting $N$ denote the toric lattice, take the line $S^1 v_i \hookrightarrow T^2 = N \otimes S^1$ (given by quotienting $\bR v_i \hookrightarrow T^2 = N \otimes \bR$), and take its conormal Legendrian lift associated to the orientation induced by $v_i$. We denote this conormal lift by $S^1_{v_i^\perp}$.

\begin{remark}

 Given a toric model for $(Y,D)$, we can also describe an almost-toric fibration with total space $U$ by starting with the toric fibration $\bar{Y} \to \bR^2$, and adding nodal edge cuts for each of the interior blow ups on $\bar{D}$, following \cite[Section 5.4]{Symington}.
Topologically, the fibrations on $U$ and $M$ as readily dual to each other: in particular, there's an obvious correspondence between nodal fibres, and with respect to the standard bases, the monodromy about one is the transpose of the monodromy about the other, with dual invariant lines. 
(The reader may recall that as well as being mirror, $U$ and $M$ are diffeomorphic -- see e.g.~\cite[Remark 1.4]{HK}.)

\end{remark}

\subsection{Elementary transformations as almost-toric moves}

In \cite[Section 3]{HK}, we prove that as a Lefschetz fibration, $w: M \to \bC$ is independent of choices. 
In the present paper, we will later need a strengthening of \cite[Section 6.1.3]{HK}: a suitably fine compatibility between changes of toric models for $(Y,D)$, mutations of the associated mirror Lefschetz fibration $w: M \to \bC$, and changes to the almost-toric fibration on the mirror $M$ (itself also associated to a given choice of toric model for $(Y,D)$). We establish this in this subsection.

\subsubsection{Change of toric model: B side}
\begin{proposition}\label{prop:change_of_toric_model}(\cite[Proposition 3.27]{HK})
Given a pair $(Y,D)$, any two choices of toric models can be related  by a sequence of moves of the following two types:
\begin{itemize}
\item toric blow ups \cite[Definition 3.23]{HK}. These are given by performing a toric blow up on $(\bar{Y},\bar{D})$, and the corresponding blow-up on $(\widetilde{Y}, \widetilde{D})$. 
\item elementary transformations \cite[Definition 3.24]{HK}. Suppose that there are two opposite toric rays in the fan for $(\bar{Y},\bar{D})$, say $v_i = -v_j$, and that $m_i > 0$. The elementary transformation replaces $n_i$ with $n_i - 1 $ and $n_j$ with $n_j + 1 $, and similarly for $m_i$ and $m_j$. The other $n_l$ and $m_l$ are unchanged, and $(Y,D)$ itself is unchanged. The toric pair changes, say to $(\bar{Y}^\natural, \bar{D}^\natural)$. 
\end{itemize}

\end{proposition}
A toy model for an elementary transformation is between $\bP^1 \times \bP^1$ (blown up at a point) and the Hirzebruch surface $\bF_1$ (blown up at a point).

This follows from the two-dimensional Sarkisov programme, using the exposition in \cite{CKS}. It is also closely related to J.~Blanc's work on birational automorphisms of $(\bC^*)^2$ \cite{Blanc}.

\subsubsection{Change of toric model: A side}\label{sec:change-toric-model-A-side}

Toric blow-ups correspond to stabilisations of the Lefschetz fibration, with a natural compatibility of the respective full exceptional collections, worked out in \cite[Section 3.4.5]{HK}. They don't change the almost-toric fibration.

Suppose that $\widetilde{Y} = \widetilde{Y}_e$, and that  we have two different toric models for $(\widetilde{Y}, \widetilde{D})$ related by an elementary transformation. The two toric models give two different full exceptional collections for $D(\widetilde{Y})$. We showed that these are related by a sequence of mutations (i.e.~Hurwitz moves), and that the two mirror distinguished collections of vanishing cycles are related by the same sequence of mutations \cite[Section 3.4.8]{HK}. 

Assume for convenience that we have $m_j > 1$ and $v_j = -v_k$. (The mutations needed to allow the cyclic reordering of labels are described in \cite[Section 3.4.1]{HK}.) On the symplectic side, the sequence of Hurwitz moves is encoded in Figure \ref{fig:elementary_transformation}, and goes as follows. 
First, mutate $W_{j m_j}$ to the end of the list of vanishing cycles. (Note that by \cite[Proposition 3.15]{HK} on the total monodromy of the fibration, this leaves the vanishing cycle itself unchanged.) 
Then mutate $V_{k-1}, \ldots, V_j$ over it to get $V_{k-1}^\natural, \ldots, V_j^\natural$. Then mutate $W_{j m_j}$ over $V_{j-1}, \ldots, V_0$, which gives the vanishing cycle $W_k$. (We have that $V_l^\natural = V_l$ for $l < j$. This description also ignores trivial Hurwitz moves of meridians over each other.) 

\begin{figure}[htb]
\begin{center}
\includegraphics[scale=0.38]{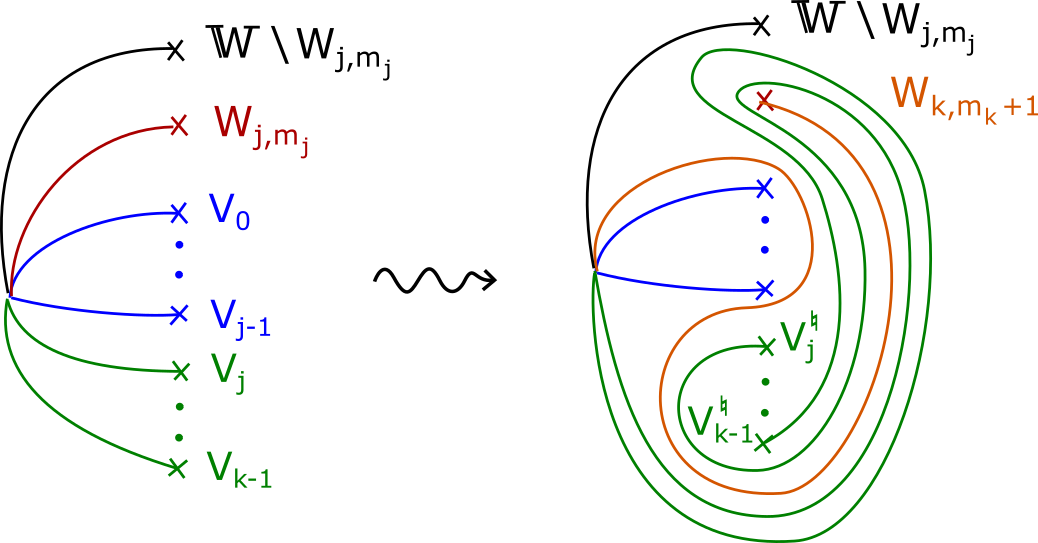}
\caption{Base of the Lefschetz fibration $w: M \to \bC$, with vanishing paths for the initial distinguished basis (left), and for the one after the elementary transformation (right). $\mathbb{W}$ denotes the collection of all the $W_{il}$.}
\label{fig:elementary_transformation}
\end{center}
\end{figure}

We will use the following well-known moves on almost-toric fibrations: 

\begin{itemize}

\item nodal slide: this deforms the almost-toric fibration by sliding a node along its invariant line (without passing any other nodes that might be on that line), see \cite[Section 6.1]{Symington}. Formally, we have a one-parameter family of exact symplectic manifolds $\{ (M_t, \omega_t) \}_{t \in [0,1]}$, which agree outside a compact set, and are symplectically isotopic by Moser.

\item cut transfers: for a given node, instead of describing the almost-toric fibration using one invariant half-line, switch to using the other invariant half-line \cite[Definition 2.1]{Vianna}. This should be thought of as changing the description of an almost-toric fibration rather than the fibration itself; it corresponds to a piecewise linear transformation of the base. 

\end{itemize}

A nodal slide implicitly uses a choice of reference Lagrangian section in a neighbourhood of the nodal fibre. Up to fibrewise Hamiltonian isotopy, there is  is a unique one locally, and so this choice is usually suppressed. We will sometimes later need a global choice of Lagrangian section.
Note that if we definine an almost-toric fibration by starting with $T^\ast \bR^2 / \bZ^2 \to \bR^2$ and introducing Symington cuts, we get a prefered zero-section, inherited from the zero section in $T^\ast \bR^2$. 
 Unless otherwise specified, we will work with these; they are taken to each other under nodal slides and cut transfers.

Following \cite[Section 6.1.3]{HK}, elementary transformations are `mirror' to nodal slides combined with cut transfers. We want to  make this precise. First, notice that the combinatorics matches up.

\begin{lemma}\label{lem:slide_combinatorics} Assume $\widetilde{Y} = \widetilde{Y}_e$.
Suppose $(\widetilde{Y}, \widetilde{D}) \to (\bar{Y}, \bar{D})$ and  $(\widetilde{Y}, \widetilde{D}) \to (\bar{Y}^\natural, \bar{D}^\natural)$ are related by an elementary transformation as above: assume that  $v_j = -v_k$ in the fan for $\bar{Y}$, and that $m_j > 0$, so that we're replacing $n_j$ with $n_j^\natural = n_j - 1 $ and $n_k$ with $n_k^\natural = n_k + 1 $, and similarly for the $m_l$. Start with the mirror almost-toric fibration associated to the first toric model. The one for the second toric model can be obtained by sliding the `final' node (i.e.~the one closest to the central fibre) with invariant direction $v_j$ past the central fibre, and then transfering the cut to the $v_k = -v_j$ side. 
\end{lemma}

\begin{proof}
Recall that all elementary transformations are obtained by toric blow ups on one of the `model' ones: the Hirzebruch surfaces $\bF_a$ and $\bF_{a-1}$ each blown up once for $a \geq 1$ (by convention $\bF_0 = \bP^1 \times \bP^1$)  \cite[Proposition 3.37]{HK}. This means that it's enough to check these cases, which can be readily done by hand. 
\end{proof}

The two almost-toric fibrations determine two different Weinstein handbody descriptions for $M$:

\begin{itemize}

\item
For the original fibration, start with $D^\ast T$, where $T$ is the original (exact) central fibre, i.e.~the intersection point of the invariant lines; and, for each nodal fibre with invariant direction $v_i$, glue a Weinstein two-handle $D^\ast D^2$ with attaching Legendrian the half-conormal to $S^1_{v_i^\perp} \subset T = \bR^2 / \bZ^2$ with coorientation given by $v_i$. Call these $D^\ast D_{i l}$, $i=1, \ldots, k$, $l=1, \ldots, m_i$. (For $m_i > 1$, we get an $A_{m_i-1}$ chain of Lagrangian spheres.)

\item  For the new fibration, start with $D^\ast T'$, where $T'$ is the new central fibre, and glue Weinstein two-handles similarly to above, say $D^\ast D^\natural_{i l}$, $i=1, \ldots, k$, $l=1, \ldots, m_i^\natural$. 

\end{itemize}

These are related by a Lagrangian mutation of $T$ over $D_{j m_j}$. Recall Lagrangian mutations are surgery-type operations modelled on passing from a Clifford to a Chekanov torus; they change the Lagrangian skeleton of the Weinstein domain but not the space itself. These constructions go back to \cite{Auroux_Gokova}, see \cite[Section 4.4]{Pascaleff-Tonkonog} for a careful general description. (They were first systematically used in \cite{Vianna}; we are in the same setting as \cite{STW}.) 
In particular, in our case, we immediately have that $T'$ is the mutation of $T$ over $D_{j m_j}$, and under this mutation $D_{j m_j}$ becomes $D^\natural_{k m_k^\natural}$ (this reflects the fact that $v_j  =- v_k$), and $D_{i l}$ just becomes $D^\natural_{i l}$ otherwise.

We're now ready to complete our discussion: 
\begin{lemma}
Suppose $(\widetilde{Y}, \widetilde{D}) \to (\bar{Y}, \bar{D})$ and  $(\widetilde{Y}, \widetilde{D}) \to (\bar{Y}^\natural, \bar{D}^\natural)$ are related by an elementary transformation as above. Then the two mirror operations are compatible: under the explicit identification in Theorem \ref{thm:ATstructure} between $M$ as the total space of the Lefschetz fibration associated to the first toric model, resp.~second one, and as the total space of the almost-toric fibration associated to the first toric model, resp.~second one, the sequence of Hurwitz moves of Figure \ref{fig:elementary_transformation} induces the Lagrangian mutation described above.
\end{lemma}

\begin{proof}

$T$ is given by iterated Polterovich surgery on the dual collection  $V^\ast_{k-1}, \ldots, V^\ast_1, V^\ast_0 (= V_0)$ \cite[Theorem 5.5]{HK}. Moreover, as all elementary transformations are pulled back from model ones between $\bF_{a-1}$ and $\bF_a$, inspecting the proof of \cite[Theorem 5.5]{HK}, we see that the second of the two base cases considering therin always applies for us. Further, we see that the iterated Polterovich surgeries can be factored into gluing $V^\ast_{k-1}, \ldots, V^\ast_j$ and $V^\ast_{j-1}, \ldots, V^\ast_0$, with each giving a cylinder, glued by the product of $W_j \sqcup W_k$ with a small interval. See Figure \ref{fig:torus_for_mutation_1}. In particular, from this perspective we readily see the handle $D^\ast D_{j m_j}$, which corresponds to $W_{j m_j}$ under \cite[Proposition 6.3]{HK}.

\begin{figure}[htb]
\begin{center}
\includegraphics[scale=0.37]{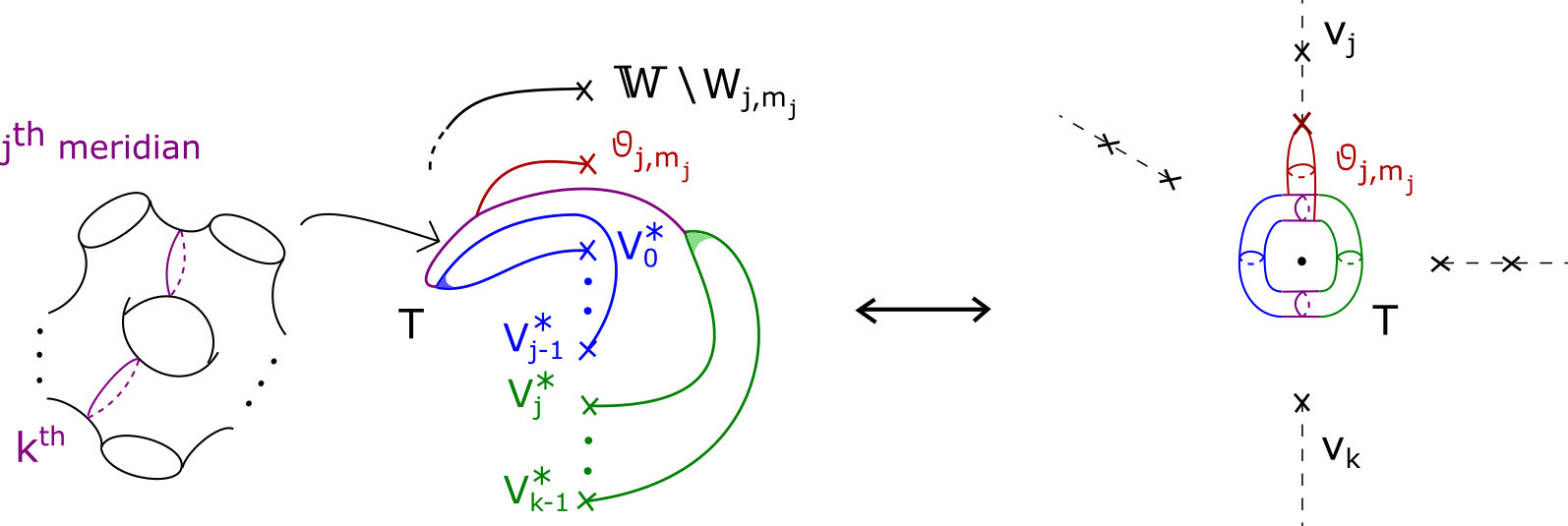}
\caption{Visualising the Weinstein equivalence between $M$ and the total space of the almost-toric fibration: before the mutation.}
\label{fig:torus_for_mutation_1}
\end{center}
\end{figure}

Now notice that in the Lefschetz fibration picture, the thimble $\vartheta_{j m_j}$ (corresponding to $W_{j m_j}$) can be Hamiltonian isotoped relative to $T$ to be viewed as the vanishing thimble above the vanishing path for $W_{k, m_k+1}= W_{k, m_k^\natural}$, which appears after the Hurwitz moves for the elementary transformation. This isotopy respects the Lefschetz fibration, but not the almost-toric fibration. This is illustrated in Figure \ref{fig:torus_for_mutation_2}. 

\begin{figure}[htb]
\begin{center}
\includegraphics[scale=0.37]{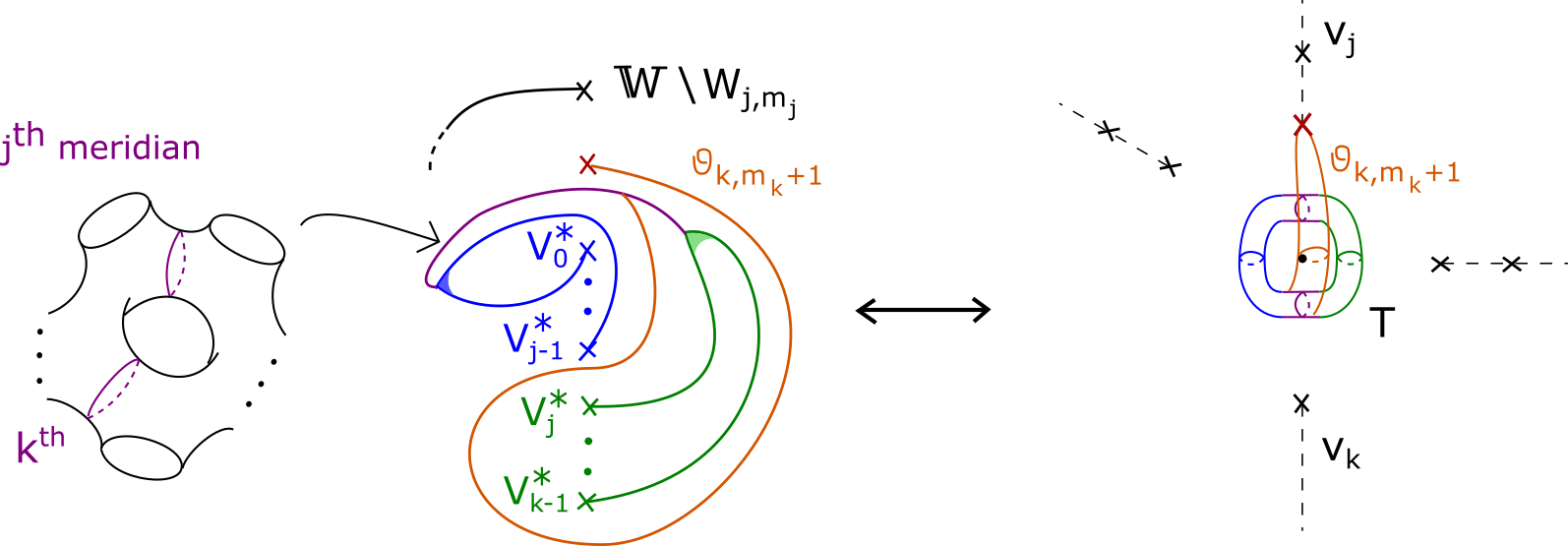}
\caption{Visualising the Weinstein equivalence between $M$ and the total space of the almost-toric fibration: preparing for the mutation.}
\label{fig:torus_for_mutation_2}
\end{center}
\end{figure}

The torus $T'$ is given, on the Lefschetz side, by iterated Polterovich surgery on the dual collection  $(V_{k-1}^\natural)^\ast, \ldots, (V_1^\natural)^\ast, (V_0^\natural)^\ast$, which can be factored as with the description for $T$. On the other hand, there is a standard model for a Lagrangian mutation in a Lefschetz fibration: the local model for the mutation between a Clifford and Checkanov torus. This is carefully described in \cite[Section 4]{Pascaleff-Tonkonog}. By considering that model together with the diagrams for $T$, $T'$ and $\vartheta_{j m_j} = D_{j m_j}$, we now recognise that $T'$ is given by mutating $T$ over $\vartheta_{j m_j}$. See Figure \ref{fig:torus_for_mutation_3}. (We've implicitly used the fact that $W_j \cap W_k = \emptyset$ to mutate both of the purple annuli on $T$ in the figure, given by thickenings of the $j$th and $k$th meridians, over $\vartheta_{j m_j}$ -- the second one doesn't change.)
\begin{figure}[htb]
\begin{center}
\includegraphics[scale=0.35]{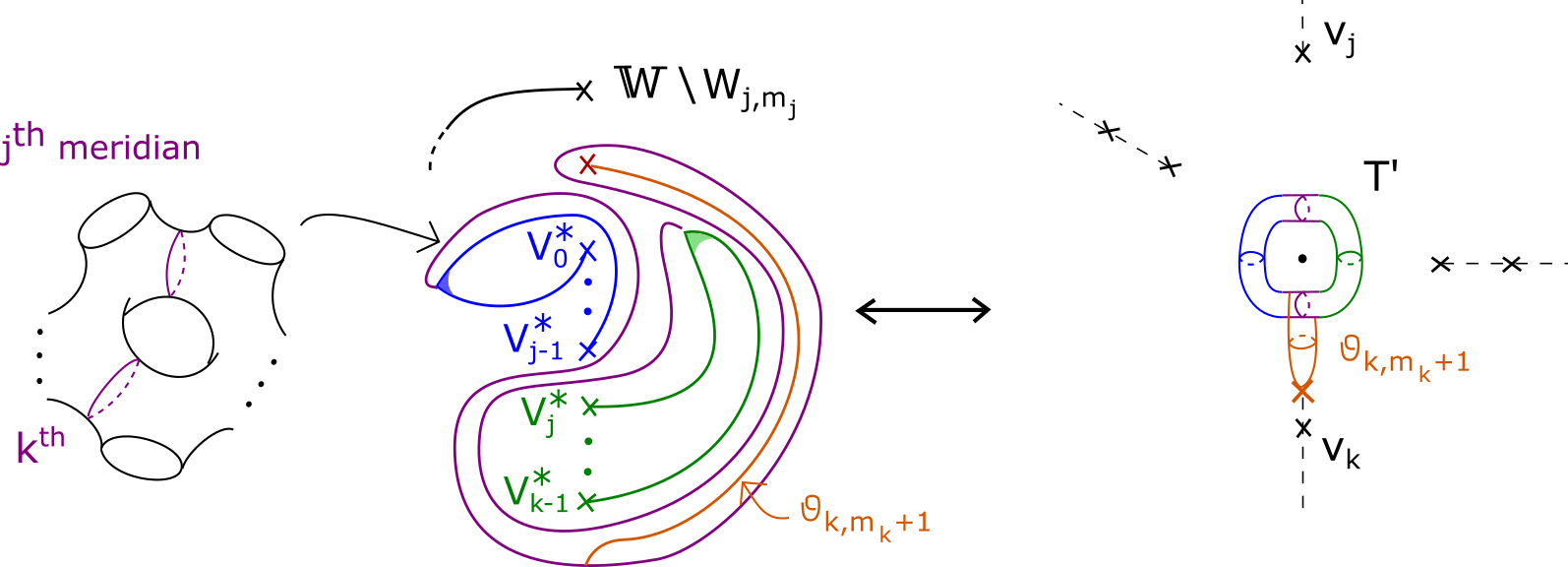}
\caption{Visualising the Weinstein equivalence between $M$ and the total space of the almost-toric fibration: after mutation. We've isotoped the thimble $\vartheta_{k,m_k +1}$ relative to $T'$ a little for legibility.}
\label{fig:torus_for_mutation_3}
\end{center}
\end{figure}
\end{proof}

\section{Lagrangian translations and tensors by line bundles}\label{sec:translations_and_tensors}

Suppose $M^4$ is an exact symplectic manifold, and the total space of an almost-toric fibration $\pi: M \to B$, as defined in \cite{Symington}, where $B$ is an intergral affine manifold homeomorphic to a disc. We will assume throughout that $\pi$ only has nodal, i.e.~focus-focus, singularities.  We also assume that there is a unique nodal singularity in each critical fibre (by sliding along the invariant direction for the singularities, this is true w.l.o.g.). Finally, assume throughout that $\pi$ is equipped with a reference Lagrangian section, say $L_0$. (Recall that this is the case when $M$ is the mirror to a log CY surface with one of the explicit almost-toric fibrations.)

\subsection{Spaces of Lagrangian sections}\label{sec:Lag_sections}
 
Assume $M$ and $\pi$ are as above.
By the Arnol'd--Liouville theorem on action-angle coordinates \cite{Arnold}, any regular fibre $F$ of $\pi$ has an open neighbourhood that is fibre-preserving symplectomorphic to $(V \times T^2, \omega_0) \to V$, where $V \subset \bR^2$. (Here $\omega_0$ denotes the standard symplectic structure $\sum dp_i \wedge dq_i$ on $\bR^2 \times T^2$, where the $p_i$ are the standard coordinates on $\bR^2$ and the $q_i$ the ones on $T^2$, normalised to have period $2\pi$.) We'll sometimes refer to patches $V \times T^2$ as above `Arnol'd--Liouville charts'. By Duistermaat \cite{Duistermaat}, these can be upgraded to global action-angle coordinates. We briefly recall the features we will use. Consider the map $\mathfrak{p}_{L_0}$ given by:
\begin{align*}
\mathfrak{p}_{L_0}: T^\ast B & \longrightarrow M \\
(q,p) & \longmapsto H_{h_{p} \cdot \pi} (L_0 (q))
\end{align*}
where $L_0$ is viewed as a map $B \to M$, $h_p$ is the germ of a smooth function $B(q) \subset B \to \bR$ representing $p \in T^\ast_q B$, and $H_f$ denotes the time-one Hamiltonian flow of a smooth function $f$. By construction, $\mathfrak{p}_{L_0}$ intertwines the projection to $B$. If $q \in B$ is a smooth value of $\pi$, then $\mathfrak{p}_{L_0}: T^\ast_q B \to \pi^{-1} (q)$ is surjective; if $q \in B$ is a singular value, $\mathfrak{p}_{L_0}: T^\ast_q B \to \pi^{-1} (q)$ only misses the critical point. (This uses our assumption that $\pi$ only has nodal singularities with a single node in each critical fibre.) The period lattice $\Lambda^\ast_q \subset T^\ast_q B$ is defined to be $\mathfrak{p}_{L_0}^{-1}(L_0 (q))$. This has full rank of $q$ is a smooth value of $\pi$, and rank one if $q$ is a critical value. Let $B^0$ be the complement of the critical values of $\pi$ in $B$. Then $\mathfrak{p}_{L_0}$ induces a symplectomorphism $T^\ast B_0 / \Lambda^\ast \simeq \pi^{-1}(B_0)$. For nodal fibres, we get local symplectomorphisms away from critical points of $\pi$.

We want to describe certain spaces of Lagrangians sections of $\pi$. Let's start by working away from the critical fibres. As before, let $B^0$ be the complement of the critical values in $B$. Let $M^0 = \pi^{-1} (B^0)$ be the complement of the critical fibres, and let $\pi_0 = \pi|_{M^0}$.  Let $\gamma^\vee$ denote a fibre of $\pi$. 
By the Leray spectral sequence, there is a short exact sequence 
\begin{equation}\label{eq:leray}
0 \to H^1(R^1 (\pi_0)_\ast \bZ) \to H^2(M^0) \xrightarrow{\int_{\gamma^\vee}} \bZ \to 0
\end{equation}
which automatically splits.
On the other hand, by \cite[Proposition 6.69]{Clay2}, $H^1(R^1 (\pi_0)_\ast \bZ)$ classifies Lagrangian sections of $\pi_0$ up to a fibre-preserving symplectic isotopy, i.e.~the addition  in $T^\ast B^0$ of the graph of a closed one-form on $B^0$. 
In particular, this implies that we have a one-to-one correspondence:
\begin{center}
\begin{tabular}{ccc}
$\Bigg\{ $ \makecell{ Lagrangian sections $L$ of $\pi_0$ up to \\  fibre-preserving symplectic isotopy } $ \Bigg\}$
& $\longleftrightarrow$ & \Big\{ $\text{ker}(\smallfrown \gamma^\vee: H_2(M^0, \partial M_0) \to \bZ ) \Big\} $ \\
$L$ & $\mapsto$ & $[L] - [L_0]$
\end{tabular}
\end{center}

Let $N_i$ be a small open neighbourhood of the $i$th singular fibre, and $\partial_i = \partial N_i$ be its boundary; by construction, this is the mapping torus of the transformation $\begin{pmatrix} 1 & 1 \\ 0 & 1 \end{pmatrix}$ acting on $\gamma$, for some choice of basis for $H_1(\gamma; \bZ)$. In particular, $H_1 (\partial_i; \bZ) \simeq \bZ \oplus \bZ$, where the first factor is generated by a loop about the critical point in the base, and the second factor is the class of a generator of $H_1(\gamma; \bZ)$ which intersects the vanishing cycle for this singularity at a single point. (Note that while the class in $H_1(\gamma; \bZ)$ isn't well-defined, the one in $H_1 (\partial_i; \bZ)$ is.) Projecting to the second factor gives a map $\epsilon_i: H_1(\partial_i; \bZ) \to \bZ$. 

The Mayer-Vietoris exact sequence and Poincar\'e--Lefschetz duality imply that
\begin{equation}\label{eq:MV}
H_2 (M, \partial M) = \text{ker} \Big(\sum \epsilon_i: H_2 (M^0, \partial M^0) \to \bigoplus_i \bZ \Big). 
\end{equation}

We're now ready to fill the critical fibres back in.
\begin{lemma}\label{lem:sections_M_noncpt}
There is a one-to-one correspondence:
\begin{center}
\begin{tabular}{ccc}
$\Bigg\{ $ \makecell{ Lagrangian sections $L$ of $\pi$ up to \\  fibre-preserving Hamiltonian isotopy } $ \Bigg\}$
& $\longleftrightarrow$ & \Big\{ $\text{ker}(\smallfrown \gamma^\vee: H_2(M, \partial M) \to \bZ \Big\} $ \\
$L$ & $\mapsto$ & $[L] - [L_0]$
\end{tabular}
\end{center}

\end{lemma}
As all closed forms on $B$ are exact, fibre-preserving symplectic isotopies are necessarily Hamiltonian.

\begin{remark}
The Leray spectral  sequence still implies that there is a short exact sequence
\begin{equation}\label{eq:leray}
0 \to H^1(R^1 (\pi)_\ast \bZ) \to H^2(M) \xrightarrow{\int_{\gamma^\vee}} \bZ \to 0.  
\end{equation}
In particular, this characterisation of Lagrangian sections naturally generalises \cite[Proposition 6.69]{Clay2}. 
\end{remark}

\begin{proof} of Lemma \ref{lem:sections_M_noncpt}. 
If $L$ is a section of $\pi$, then certainly $[L]-[L_0] \in \text{ker}(\smallfrown \gamma^\vee: H_2(M, \partial M) \to \bZ )$, and two isotopic sections give the same class. Thus the map is well-defined.

Surjectivity: by equation \ref{eq:MV}, $$  \text{ker}(\smallfrown \gamma^\vee: H_2(M, \partial M) \to \bZ )  \simeq  \text{ker}\big( (\smallfrown \gamma^\vee, \sum \epsilon_i): H_2(M^0, \partial M^0) \to \bZ \oplus_i \bZ \big). $$
Fix an arbitrary class on the right-hand side. There exists a Lagrangian section $L$ of $\pi_0$ such that $[L] - [L_0]$ represents that class. It's enough to show that, possibly after a fibre-preserving symplectic isotopy, we can extend $L$ to a section of $\pi$. 

Let $B^0_i \subset B^0$ be an open annulus about the $i$th critical value. From the discussion above, Lagrangian sections of the fibration restricted to $B^0_i$, up to fibre-preserving symplectic isotopy, are classified by $\bZ = \text{ker}(\smallfrown \gamma^\vee: H_2(M^0_i, \partial M^0_i) \to \bZ )$; moreover, $\epsilon_i$ restricts to an isomorphism $\bZ \to \bZ$. It follows that there exist closed forms $\zeta_i \in \Omega^1(B^0_i)$ such that $L$ and $L_0$ differ by the graph of $\zeta_i$ over $B^0_i$. Now pick a closed form $\zeta \in \Omega^1(B^0)$ extending the $\zeta_i$. Let $L' = L - \Gamma(\zeta)$, where $\Gamma(\zeta)$ is the graph of $\zeta$. By construction, $L'$ is symplectic isotopic to $L$ over $B^0$ via a fibre-preserving map, and it agrees with $L_0$ on each of the $B^0_i$. Now extend $L'$ by $L_0$ over each of the punctures.

Injectivity: suppose that $L$ and $L'$ are sections of $\pi$ such that $[L]-[L_0]$ and $[L']-[L_0]$ agree in $\text{ker}(\smallfrown \gamma^\vee: H_2(M, \partial M) \to \bZ ) $. Then their restrictions to $M^0$ differ by the graph of a \emph{closed} one-form on $B^0$, say $\zeta$. It's enough to show that $\zeta$ must also be exact. Let $S^1_i \subset B^0_i$ be a waist curve of the annulus. Let $V_i: S^1 \times [0,1] \to T^\ast B^0$ be the annulus given by $(\theta, t) \to t \zeta (\theta)$, for $\theta  \in S^1_i$, and let $\bar{V_i}$ be its image in $T^\ast B^0 / \Lambda^\ast = M$.
 Now notice that
$$
\int_{S^1_i} \zeta = \int_{V_i} \omega = \int_{\bar{V_i}} \omega = \int_{L_0|B_i} \omega  - \int_{L|B_i} \omega = 0
$$
where $B_i \subset B$ is a disc with boundary $S^1_i$, and we are using Stokes' theorem together with the fact that $\bar{V}_i$, $L_0|B_i$ and $L|B_i$ glue to give a null-homologous two-cycle in $M$.
\end{proof}

\begin{remark}\label{rmk:Hanlon1}
To what extent can we control the behaviour of a section near $\infty$? One can check that for a choice of Liouville form on $M$ suitably adapted to the almost-toric fibration, any Lagrangian section of $\pi$ can be arranged to be invariant under the Liouville flow outside a compact set, after a fibre-preserving Hamiltonian isotopy. We do not pursue this further here: we expect that a comprehensive treatment of the behaviour of Lagrangian sections near infinity will follow using a notion of monomial admissibility, introduced by Hanlon \cite{Hanlon}.
\end{remark}

We haven't yet used exactness of $M$. However, this is required to prove the next classification. 

\begin{proposition}\label{prop:sections_M_cpt}
Suppose $M$ is an exact symplectic manifold as above, and $L_0$ a reference Lagrangian section of $\pi$. 
Then we have a one-to-one correspondence:
\begin{center}
\begin{tabular}{ccc}
$\Bigg\{ $ \makecell{ Lag.~sections $L$ of $\pi$ equal to $L_0$ near $\partial M$  up to \\ compactly supported fibre-preserving Ham.~isotopy } $ \Bigg\}$
& $\longleftrightarrow$ & 
 \Big\{ 
$H_2(M) / ( \bZ \cdot \gamma^\vee ) $
\Big\}
\\
$L$ & $\mapsto$ & $\big[[L] - [L_0] \big]$
\end{tabular}
\end{center}
where the homology groups are taken with integer coefficients.
\end{proposition}

\begin{proof}
The map is clearly well-defined: as $L_0$ and $L_1$ agree near $\partial M$, $[L_0]-[L_1]$ gives a class in $H_2(M)$, and a fortiori in $H_2(M) / (\bZ \cdot \gamma^\vee)$. 
We will use the following algebro-topological features of the set-up. Let $A$ be the total monodromy of the fibration. Then $\partial M$ is a 2-torus bundle over $S^1$ with monodromy $A$. We have, first, that $H_1(\partial M) \simeq \bZ \oplus \bZ^2 / (A-I)$; the first factor is given by the base $S^1$, and the second by the image of $H_1(\gamma^\vee, \bZ)$; composing the boundary map with projection to the first factor  precisely gives $\smallfrown \gamma^\vee: H_2(M, \partial M) \to \bZ$. 
Secondly, we have that $H_2(\partial M) \simeq H^1(\partial M)$ consists of a $\bZ$ factor generated by $\gamma^\vee$, and, when $\det (A-I)=0$, either one or two extra $\bZ$ factors generated by tori fibred over the circle $\partial B$. 


Injectivity: suppose $L$ and $L_0$ are Lagrangian sections which agree outside a compact set, and such that $\big[ [L]-[L_0] \big] = 0 \in H_2(M) / (\bZ \cdot \gamma^\vee) $.
By the long exact sequence of the pair for $(M, \partial M)$, we have that 
$[L]-[L_0] = 0 \in H_2(M, \partial M)$. 
By Lemma \ref{lem:sections_M_noncpt}, there certainly exists a fibrewise Hamiltonian isotopy taking one to the other: in other words, $L$ and $L_0$ differ by the graph of an exact one-form on $B$, say $df$. By assumption, we have $df \in \Lambda^\ast$ outside a compact set. If $1$ is not an eigenvalue of the monodromy matrix $A$, we immediately get that $f$ must be constant outside a compact set, and the Hamiltonian isotopy can be taken to have compact support. In general, the possibilities for $f$ outside a compact set are (up to constants) indexed by $H^1(\partial M) / \PD[\gamma^\vee]$. As we have assumed that $\big[ [L]-[L_0] \big] = 0 \in H_2(M, \partial M) / (\bZ \cdot \gamma^\vee) $, we again get that $f$ must be constant outside of a compact set. (Note that this would not have been the case if we had merely assumed that $[L] - [L_0]$ vanished in $ H_2(M) / H_2(\partial M) \simeq \ker ( H_2(M, \partial M) \to H_1(\partial M))$.)

Surjectivity: suppose we are given a class $l$ in $H_2 (M) / (\bZ \cdot \gamma^\vee)$. Let $l'$ be its image in $H_2(M, \partial M)$.  Whenever $A$ does not have eigenvalue $1$, the map $H_2 (M) / (\bZ \cdot \gamma^\vee) \to H_2(M, \partial M)$ is injective. Whenever $A$ does have eigenvalue $1$,  by similar considerations to those at the end of the injectivity argument, we see for that a given class $l' \in H_2(M, \partial M)$, if we can find a section $L$ agreeing with $L_0$ near $\partial M$ and such that $[L] - [L_0] \in H_2(M) / (\bZ \cdot \gamma^\vee)$ is a lift of $l'$, then we can find Lagrangian sections giving representatives for all other lifts of $l'$. This is given by adding to $L$ an exact one-form $df$, for some function $f$ which near $\partial M$ is  non-constant  and such that $df \in \Lambda^\ast$. 

Given $l'$, by Lemma  \ref{lem:sections_M_noncpt}, there is a Lagrangian section $L$ of $\pi$ such that $[L]-[L_0] = l'$. Mapping to $H_2(\partial M)$ and projecting, we get $[L]-[L_0] \mapsto 0 \in \bZ^2 / (A-I)$. Now proceed as before: first, notice that again applying the result for fibrations without singularities \cite[Proposition 6.69]{Clay2}, the map to $\bZ^2 / (A-I)$ classifies Lagrangian sections of the restriction of the fibration to a neighbourhood of $\partial B$, up to addition of the graphs of closed one-forms. Thus $L$ and $L_0$ differ by the graph of such a one-form, say $\zeta$, near $\partial B$. Then proceed similarly to the end of the proof of Lemma \ref{lem:sections_M_noncpt}: the obstruction to $\zeta$ being exact is the symplectic area of an annulus which can be capped off in $M$ by two Lagrangian discs to give a closed cycle. In general, this cycle will no longer be null-homologous. However, by exactness of $\omega$, the obstruction still vanishes. Thus after a fibre-preserving Hamiltonian isotopy (indeed, supported near $\partial M$), we can assume that $L$ agrees with $L_0$ near $\partial M$. 
\end{proof}

We note the following consequence of the arguments in our proof.
\begin{corollary}\label{cor:isotoping_Lag_sections}
Suppose $L_0, L_1$ are Lagrangian sections on $\pi$ which are smoothly isotopic as sections near $\partial B$. Then there is a fibre-preserving Hamiltonian isotopy of $M$, supported near $\partial M$, such that the image of $L_1$ agrees with $L_0$ on a tubular neighbourhood of $\partial M$. 
\end{corollary}

For completeness, we record the following.

\begin{lemma}
Suppose $M$ is mirror to a log CY2 surface $(Y,D)$ as in Theorem \ref{thm:hms}. Then the monodromy matrix $A$ has one as an eigenvalue  if and only if $H^1(M,\bZ) \neq 0$ or $(D_i \cdot D_j)$ is strictly negative semi-definite.
\end{lemma}
\begin{proof}

Let $(Y,D)$ be a Looijenga pair. Let $N$ be a tubular neighborhood  of $D$ in $Y$ and $U'=Y \setminus N$, an oriented smooth 4-manifold with boundary $\partial U'$. This is diffeomorphic to $M$, and $\partial U' \simeq \partial M$ is a $2$-torus bundle over $S^1$ with monodromy matrix $A \in \SL(2,\bZ)$.

By the Leray spectral sequence for the fibration $f \colon \partial U' \rightarrow S^1$ we have an exact sequence
$$0 \rightarrow H^1(S^1,\bZ) \rightarrow H^1(\partial U',\bZ) \rightarrow H^0(R^1f_*\bZ) \rightarrow 0$$
and
$$H^0(R^1f_*\bZ) \simeq \bZ^2/(A-I)\bZ^2$$

It follows that the monodromy matrix $A$ has an eigenvector with eigenvalue $1$ if and only if $H^1(\partial U')$ has rank $>1$.
We have the exact sequence of cohomology for the pair $(U',\partial U')$:
$$\cdots \rightarrow H^1(U') \rightarrow H^1(\partial U') \rightarrow H^2(U', \partial U') \rightarrow H^2(U') \rightarrow \cdots$$
The last arrow is $H_2(U') \rightarrow H^2(U')$ by Poincar\'e duality, with kernel the radical of the intersection form on $H_2(U')$. Recall that we have the exact sequence
$$0 \rightarrow \bZ \rightarrow H_2(U') \rightarrow Q \rightarrow 0,$$
where the $\bZ$ factor is generated by the fiber of $\partial U' \rightarrow S^1$, and 
$$Q=\langle D_1,\ldots,D_n\rangle^{\perp} \subset H_2(Y,\bZ).$$
The lattice $Q$ is non-degenerate unless $(D_i \cdot D_j)$ is strictly negative semidefinite by the Hodge index theorem.  
So $H^1(\partial U')$ has rank $>1$ if and only if either $H^1(U') \neq 0$ or $(D_i \cdot D_j)$ is negative semi definite.
\end{proof}

\begin{remark}
Note that $H^1(U') \neq 0$ if and only if there is a toric model for $Y$ where all the non-toric blowups occur on the same boundary divisor. This is because $\pi_1(U')=N/\langle v_1,\ldots,v_k\rangle$ where $N=\pi_1(T)$ is the fundamental group of the algebraic torus $T=\bar{Y} \setminus \bar{D}$ and the $v_i$ are the primitive integral generators of the rays corresponding to boundary divisors of $\bar{Y}$ containing the centers of the non-toric blowups. (Note also that blowups on the boundary divisors corresponding to $v$ and $-v$ can be transferred to the boundary divisor corresponding to $v$ by elementary transformations.)
\end{remark}

\subsection{Lagrangian translations}\label{sec:Lag_translations}

\begin{proposition}\label{prop:Lag_translations}
Suppose $M$ and $\pi$ are as above, and $L_0, L$ are Lagrangian sections of $\pi$. Let $M^s$ be the smooth locus of $\pi$, and let $\mathfrak{p} = \mathfrak{p}_{L_0}: T^\ast B \to M^s \subset M$ give global action-angle coordinates, chosen so that $\mathfrak{p}$ takes the zero-section to $L_0$.  Consider the map
\begin{eqnarray*}
\sigma_L: M^s & \to & M^s  \\
x & \mapsto & \mathfrak{p}( \widetilde{x} + \widetilde{L|_{\pi(x)}} )
\end{eqnarray*}
where $\widetilde{y}$ denotes any preimage of $y \in M$ under $p$. (In words, we are using the linear structure on each fibre to add $L - L_0$.) Then 
\begin{enumerate}
\item $\sigma_L$ is well defined;

\item $\sigma_L$ extends to a smooth map $M \to M$, which we also denote by $\sigma_L$;

\item $\sigma_L$ is a symplectomorphism of $M$;

\item whenever $L_0$ and $L$ agree near $\partial M$, $\sigma_L$ has compact support. 
\end{enumerate}

\end{proposition}

We will refer to $\sigma_L$ as a \emph{Lagrangian translation}.

\begin{proof}
The fact that $\sigma_L$ is well-defined on $M^s$ is immediate. Locally, $L$ is $\mathfrak{p}(\Gamma(df))$, where $f$ is a smooth real-valued function on a small disc in $B$; now notice that $\sigma_L$ is locally given by the time-one Hamiltonian flow of $f \circ \pi$. This description makes it clear that $\sigma_L$ extends over the critical points of $\pi$, and that it is a symplectomorphism. The final point is immediate.
\end{proof}

\begin{remark}\label{rmk:Hanlon2}
Assume that $L_0$ and $L$ don't agree near $\partial M$. Following on Remark \ref{rmk:Hanlon1}, one could choose the Liouville form on $M$, and representatives for $L_0$ and $L$ , so as to get a well-defined action of $\sigma_L$ on the wrapped Fukaya category $\cW(M)$. However, as before, we expect the monomially admissible framework currently being developed in \cite{Hanlon-Ward} to be best suited for this. 
\end{remark}

 Let $\pi_0 \Symp_c(M)$ denote the symplectic mapping class group of $M$. Fix a universal cover $\widetilde{\Lag}$ of the Lagrangian Grassmanian over $M$. Let $\Symp_c^\gr(M)$ be the group of compactly supported, graded symplectomorphisms \cite[Section 2.b]{Seidel_graded}. Here we assume both that the underlying diffeomorphism has compact support, and that the bundle map is the identity on $\widetilde{\Lag}$ outside a compact set. (This is a distinguished subgroup of the group of graded symplectomorphisms for which we only require the support of the underlying diffeomorphism to be compact.) 
By \cite[Remark 2.5]{Seidel_graded}, there is an exact sequence
\begin{equation}\label{eq:graded_symplecto_lift}
1 \to \pi_0 \Symp_c^\gr(M) \to \pi_0 \Symp_c(M) \to H^1(M, \partial M; \bZ)
\end{equation}
where the final map is not a group homomorphism, but measures the difference between the original universal covering $\widetilde{\Lag}$ and its pullback under an element of $\pi_0 \Symp_c(M)$. (Note that in the cases we will consider, $M$ is homotopy equivalent to a CW 2-complex, and in particular $H_3(M; \bZ) = 0$, and so $H^1(M, \partial M; \bZ) = 0$ by Lefschetz duality.)

We record the following properties of Lagrangian translations, the proofs of which are immediate.

\begin{proposition} \label{prop:Lag_translations_group}

We get the following:

\begin{itemize}
\item
Lagrangian translations by arbitrary Lagrangian sections give a subgroup of $\pi_0 \Symp (M)$ isomorphic to $\text{ker}(\smallfrown \gamma^\vee: H_2(M, \partial M) \to \bZ) $.

\item
Lagrangian translations by Lagrangian sections which agree with $L_0$ near infinity give a subgroup of $\pi_0 \Symp_c (M)$ isomorphic to $H_2(M) /(\bZ \cdot \gamma^\vee)$. Moreover, this lifts to a subgroup of $\pi_0 \Symp_c^\gr(M)$.

\end{itemize}

The action of $\sigma_L$ on $H_2(M, \partial M)$  is in either case determined by the fact that $\sigma_L$ acts on Lagrangian sections by adding $[L]-[L_0]$, and fixes $\gamma^\vee$ and hence the $\bZ = \text{Im} (\smallfrown \gamma^\vee)$. (Dually, we could use the action on $H^2(M)$, see equation \ref{eq:leray}.) In particular, with respect to the decomposition $(\bZ, \text{ker}(\smallfrown \gamma^\vee: H_2(M, \partial M) \to \bZ)) $, $\sigma_L$ is given by $\begin{pmatrix}  1& 0 \\ [L]-[L_0] & \text{Id} \end{pmatrix} $.
\end{proposition}

We will also want compatibility with nodal slides, which follows from standard Moser-type arguments:

\begin{lemma}
\label{lem:nodal_slide_Lag_section}
Suppose that we perform a nodal slide and cut transfer using $L_0$ as the reference section. Formally, we get 
almost-toric fibrations $\pi^t: M^t \to \bR^2$, with an isotopy $h^t: M^0 \to M^t$ preserving the reference Lagrangian sections $L_0^t$. 
Suppose $\sigma^t_L \in \pi_0 \Symp_c M^t$ is the Lagrangian translation associated to $[L -L^t_0] \in H_2(M^t) / (\bZ \cdot \gamma^\vee)$; then $$\sigma^t_{h^t_\ast [L]} = h^t \circ \sigma^0_{ L} \circ (h^t)^{-1} \in \pi_0 \Symp_c M^t.$$
Similarly for the non-compactly supported case.
\end{lemma}

\subsection{Mirror symmetry for Lagrangian translations: $K$-theory} \label{sec:MS_translations_Ktheory}

We want to show that Lagrangian translations are mirror to tensors with line bundles. We start by establishing this at the $K$-theoretic level.

\begin{lemma}
Line bundles on $U$ are in one-to-one correspondence, via the first Chern class, with $\text{ker} \Big( \int_\gamma: H^2(U) \to \bZ \Big) \simeq \text{ker} \Big(\smallfrown \gamma: H_2(U, \partial U) \to \bZ \Big)$. 
\end{lemma}

\begin{proof}
The long exact sequence for the pair $(Y,D)$ and excision give
$$
0 \to \bZ^k \to H^2(Y) \to H^2(U) \xrightarrow{\int_\gamma} \bZ \to 0
$$
which in turn induces
$$
0 \to\Pic(U) \xrightarrow{c_1} H^2(U) \xrightarrow{\int_\gamma} \bZ \to 0. 
$$
\end{proof}

In the semi-definite case, let $F \subset \partial U$ be the class of a torus given by smoothing all the nodes of $D$. 
Note that 
$$
H_2(\partial U) =
\begin{cases}
\bZ \langle \gamma \rangle & \text{in the negative definite case} \\
\bZ  \langle \gamma \rangle \oplus \bZ  \langle F \rangle & \text{in the semi-definite case}
\end{cases}
$$

Now assume $Y=Y_e$. 
As before, let $Q = \langle D_1, \ldots, D_k \rangle^\perp \subset \Pic(Y)$, and let $\bar{Q}$ be the image of $Q$ in $\Pic(U)$.

\begin{lemma}\label{lem:Qbar}
The first Chern class gives an isomorphism 
 $$\bar{Q} \simeq \text{ker} \big( H^2(U) \to H^2(\partial U) \big) \simeq H_2(U) / H_2 (\partial U). $$
\end{lemma}

\begin{proof}  Consider the following commutative diagram

$$
\xymatrix{
&& 0 \ar[d] & 0 \ar[d]  &  \\
0 \ar[r]  &Q \cap \langle D_1, \ldots, D_k \rangle \ar[r] \ar[d] & Q \ar[r] \ar[d] & \bar{Q} \ar[d] \ar[r] & 0 \ar[d]\\
0 \ar[r] & \langle D_1, \ldots, D_k\rangle\ar[r] \ar[d]_{\text{Id}} & H^2(Y) \ar[r] \ar[d]_{(\cdot D_i)_{i=1, \ldots, k}} & H^2(U) \ar[r]^{\int_\gamma} \ar[d] & \bZ \ar[d]_{\text{Id}} \\
0 \ar[r] & H_2(D) \ar[r] & H^2(D) \simeq H_2(\nu D, \partial U) \simeq \bZ^k \ar[r] & H_1 ( \partial U) \ar[r] & H_1(D) \simeq \bZ
}
$$
where $\nu D$ denotes a small neighbourhood of $D \subset Y$ (this deformation-retracts onto $D$). 
All rows are exact;  the bottom row is part of the long exact sequence for the pair $(\nu D, \partial \nu D) = (\nu D, \partial U)$. Also, all the columns apart from the one involving $\bar{Q}$ are automatically exact; exactness of that final column, and our claim, then follow. 
\end{proof}

\begin{lemma}
$K(U) = K_0 (D^b \Coh (U))$ is isomorphic to $\bZ \oplus \Pic(U)$ via $(\text{rk}, c_1)$. 
\end{lemma}

\begin{proof}
There is an exact sequence
$$
K(D) \to K(Y) \to K(U) \to 0
$$
using \cite[Exercise 2.10.6 (b)]{Hartshorne}. We also have an isomorphism $K(Y) \to \text{Im}(\text{ch}(Y)) \subset H^\text{even}(Y; \bQ)$, where $\text{ch}$ denotes the Chern character; as $Y$ is a surface, this map is given by $(rk, c_1, \text{ch}_2)$.  (This follows e.g.~from the existence of our exceptional collection for $D(Y)$.)

For each $D_i$, $K(D_i) \simeq \bZ^2$, generated by, for instance, $\cO_{D_i}, \cO_{p_i}$, where $p_i \in D_i$ can be taken to be any interior point of $D_i \subset D$. Moreover, there is a surjection $\bigoplus_i K(D_i) \twoheadrightarrow K(D)$. Applying the Chern character, $\cO_{D_i} \mapsto (0, [D_i], -1/2[D_i]^2)$; and all of the $\cO_{p_i}$ have the same image, of the form $(0,0, 1)$. The claim is then clear.
\end{proof}

We note the following immediate consequence:
\begin{lemma}
Let  $\cL \in \Pic(U)$. The automorphism $ \otimes \cL: D^b \Coh(U) \to D^b \Coh(U)$ induces an automorphism $K(U) \to K(U)$, given by $\begin{pmatrix}  1 & 0 \\ c_1(\cL) & \text{Id}     \end{pmatrix}$ on $\bZ \oplus \Pic(U)$. 
\end{lemma}

Let's still assume that $Y=Y_e$, and let $M$ be the mirror to $Y \backslash D$, with $\pi: M \to \bR^2$ an almost-toric fibration from Theorem \ref{thm:ATstructure}; note that it only has  focus-focus singularities.
Let  $\gamma^\vee$ denote a smooth fibre of $\pi$, and on the mirror side, let $\gamma$ denote the class of a crossing torus at a node of $D$.

\begin{lemma}\label{lem:iso_Ktheories}
Let $\cW(M)$ denote the wrapped Fukaya category of $M$. Then there is an isomorphism $K_0 (D^b \cW(M)) \simeq H_2(M, \partial M)$, which takes a Lagrangian brane to $\pm$ its class in relative homology, where the sign depends on the grading of the brane.

Moreover, we have a commutative diagram of isomorphisms:
$$
\xymatrix{
K(U) \ar[r]^{\text{HMS}} \ar[d]^{(c_1, \text{rk})} & K_0(D^b \cW(M)) \ar[d] \\
\text{ker} \Big( \int_{\gamma}: H^2(U) \to \bZ \Big) \oplus \bZ \ar[r]_-{\iota_\ast} & \text{ker} \big( \smallfrown \gamma^\vee: H_2(M, \partial M) \to \bZ \big) \oplus \bZ \simeq H_2(M, \partial M) 
}
$$
where the first line is induced by the explicit homological mirror symmetry equivalence between $D^b \Coh U$ and $D^b \cW(M)$  in Theorem \ref{thm:hms}, and  $\iota_\ast$ is induced by a prefered diffeomorphism $\iota: U \to M$, which takes $\gamma$ to $\gamma^\vee$, together with Poincar\'e--Lefschetz duality. 

\end{lemma}

\begin{proof}
This follows from assembling previous results.  From the description of $M$ in Section \ref{sec:hms_background1}, we see that Lagrangian co-cores generate $H_2(M, \partial M)$; as they also generate the wrapped Fukaya category \cite{CDRGG}, the natural map $K_0 (D^b \cW(M)) \to H_2 (M, \partial M)$ is clearly surjective. 

On the other hand, by using the proof of Theorem \ref{thm:hms} (3) together with descriptions of $M$ given by attaching handles to $D^\ast T^2$ along Legendrian (also recalled in Section \ref{sec:hms_background1}), we get a diffeomorphism $\iota: U \to M$ such that the diagram in our statement above commutes. (This is checked inductively, starting with a toric pair and performing interior blow-ups / handle attachments. See also \cite[Remark 1.4]{HK}.)  As the two horizontal maps and the left-hand vertical one are isomorphisms, the right-hand vertical one must be too.
\end{proof}

The following corollary is then immediate.

\begin{corollary}\label{cor:mirror_autos_K-theory}
Suppose $(Y,D) = (Y_e, D)$ is a log Calabi--Yau surface with distinguished complex structure, with a choice of toric model, and $w: M \to \bC$ its mirror. Set $U = Y \backslash D$, and let $\pi: M \to B$ be the almost-toric fibration on $M$ associated to that model, with reference Lagrangian section $L_0$. 

\begin{enumerate}
\item
 The prefered diffeomorphism $\iota: U \to M$ above induces identifications:

\begin{center}
\begin{tabular}{ccc}
$\Bigg\{ $ \makecell{ Lagrangian sections $L$ of $\pi$ up to \\  fibre-preserving Hamiltonian isotopy } $ \Bigg\}$
& $\longleftrightarrow$ & \Big\{ $\text{ker}(\smallfrown \gamma^\vee: H_2(M, \partial M) \to \bZ \Big\} $ \\
&& $\Big\updownarrow$ \\
$\Pic(U)$ & $\longleftrightarrow$ &  \Big\{ $\text{ker}\Big(\int_{\gamma}: H^2(U) \to \bZ  \big)\Big\} $
 \end{tabular}
\vspace{10pt}

\begin{tabular}{ccc}
$\Bigg\{ $ \makecell{ Lagrangian sections $L$ of $\pi$ equal to $L_0$ near $\partial M$ \\ up to fibre-preserving Hamiltonian isotopy } $ \Bigg\}$
& $\longleftrightarrow$ & 
 $ H_2(M) / (\bZ \cdot \gamma^\vee ) $   
\\
&& $\Big\updownarrow$ \\
Line bundles in $Q $ & $\longleftrightarrow$ &  $H_2(U) / (\bZ \cdot \gamma )$    
 \end{tabular}

\end{center}

\item Suppose that a section $L$ corresponds to a line bundle $\cL$ in $Q$. As an element of $\pi_0 \Symp^\gr_c(M)$, $\sigma_L$ acts on $\cW(M)$, and the induced action on $K_0 D^b \cW(M)$ naturally agrees with the homological action $[\sigma_L]_\ast$ on $H_2(M, \partial M; \bZ)$. Moreover, under the identification given by homological mirror symmetry (as in Lemma \ref{lem:iso_Ktheories}), the action of $\otimes \cL$ on $K(U)$ agrees with the action of $\sigma_L$ on $K_0 D^b \cW(M)$. 

\end{enumerate}

\end{corollary}

\begin{remark}
In (2), the action on $\cW(M)$ only depends on  the image of $\cL$ in $\bar{Q}$, or equivalently the image of $L$ in $H_2(M) / H_2(\partial M) \simeq \ker (H_2(M, \partial M) \to H_1(\partial M))$. 
\end{remark}

Note that even in the non-compact case we can always make sense of the $K$-theoretic action up to a shift (depending on the choice of lift of $\sigma_L$ to a graded Lagrangian), and that the shift can be chosen to get it to agree with the action of $\otimes \cL$.

\section{Lagrangian spheres mirror to $(-2)$ curves}\label{sec:spheres}
\subsection{Mirrors to line bundles on $(-2)$ curves: construction}

\begin{proposition}\label{prop:Corti-Filip-Petracci}
Let $(Y,D)$ be a log CY surface, and let $C \subset Y \backslash D$ be a $(-2)$ curve. Then there exists a toric model for $(Y,D)$ that contracts $C$: in other words, $C$ is the strict transform of the exceptional divisor for an interior blow-up.  
\end{proposition}

\begin{proof}
This follows from the minimal model program for surfaces, and was noted by  Corti--Filip--Petracci as \cite[Lemma 3.7]{CFP}. 
\end{proof} 

This can be combined with results from Section \ref{sec:hms_background1} to give the following.

\begin{proposition}\label{prop:Lagrangian_spheres}
Assume that $(Y,D)$ is a log CY surface with $Y=Y_e$, and $M$ the mirror to $U = Y \backslash D$. Consider the explicit homological mirror symmetry equivalence between $D^b \Coh U$ and $D^b \cW (M)$ given by Theorem \ref{thm:hms} (3). 
Suppose $C \subset U$ is a $(-2)$ curve. Then

\begin{enumerate}
\item 
 There exists an embedded Lagrangian sphere $s_C \subset M$ such that $s_C$ is mirror to $i_\ast \cO_C(-1)$. This can be explicitly described as a matching cycle for the Lefschetz fibration $w: M \to \bC$, and, for a suitable choice of almost-toric fibration on $M$, $s_C$ lies above a segment joining two nodal critical points with the same invariant direction. 

\item For any $E \subset Y$ such that $E \cdot C = 1$, let  $\sigma_E$ be the Lagrangian translation associated to $\otimes \cO(E)$. Then for any $a \in \bZ$, $\sigma_E^{a} s_C$ is mirror to $i_\ast \cO_C(a-1)$.

\end{enumerate}
\end{proposition}

\begin{proof}
First, notice that the independence of choice of toric model of the mirror, as established in \cite[Section 3.4]{HK}, means that we are allowed to work with a different toric model for each $(-2)$ curve.

Assume now that  our $(-2)$ curve $C \subset U$ arises as part of a toric model $f: (\tilde{Y}, \tilde{D}) \to (\bar{Y}, \bar{D})$, where $(\bar{Y}, \bar{D})$ is toric, and $(\tilde{Y}, \tilde{D})$ is obtained from $(Y,D)$ by a sequence of corner blow-ups (so $U = Y \backslash D \simeq \tilde{Y} \backslash \tilde{D}$). Consider the usual exceptional collection on $\tilde{Y}$, namely
\begin{multline}
\cO_{\Gamma_{km_k}}(\Gamma_{km_k}), \cdots, \cO_{\Gamma_{k1}}(\Gamma_{k1}), \cdots, \cO_{\Gamma_{1m_1}}(\Gamma_{1m_1}),\cdots, \cO_{\Gamma_{11}}(\Gamma_{11}), \cO,
f^\ast \cO(\bar{D}_1), \\ \cdots, f^\ast \cO (\bar{D}_1+ \cdots + \bar{D}_{k-1})
\end{multline}
where $\Gamma_{ij}$ is the pullback of the $j$th exceptional curve over $\bar{D}_i$, for $i=1, \ldots, k$, $j=1, \ldots, m$. 

Let's first prove the claim about $i_\ast \cO_C(-1)$ in (1).  
Assume that $C$ is the strict transform of the $j$th exceptional curve over $\bar{D}_i$, for some fixed $i,j$. Then $$i_\ast \cO_C (-1) \simeq \{ e: \, \cO_{\Gamma_{i(j+1)}}(\Gamma_{i(j+1)}) \to \cO_{\Gamma_{ij}}(\Gamma_{ij}) \}$$ where $e$ is a non-trivial constant in $\Hom (\cO_{\Gamma_{i(j+1)}}(\Gamma_{i(j+1)}), \cO_{\Gamma_{ij}}(\Gamma_{ij})) \simeq \bC$. 

On the mirror side, $ \cO_{\Gamma_{i(j+1)}}(\Gamma_{i(j+1)})$  and $\cO_{\Gamma_{ij}}(\Gamma_{ij}) $ correspond to consecutive thimbles $\vartheta_{i,j+1}$ and $\vartheta_{i,j}$ with the same vanishing cycle, namely the $i$th meridian (see e.g.~\cite[Section 3.3.2]{HK}). The morphism $e$ corresponds to a non-trivial degree zero morphism in $H^1(S^1) \simeq HF(\vartheta_{i,j+1}, \vartheta_{i,j})$, where we are working in the directed Fukaya category of the superpotential. The cone over this element is the matching cycle given by gluing $\vartheta_{i,j+1}$ and $\vartheta_{i,j}$ together, by \cite[Proposition 18.21]{Seidel_book}. Let this be the Lagrangian sphere $s_C$. 

The assertion about matching cycles for the Lefschetz fibration is satisfied by construction. For the one about almost-toric fibrations, consider the almost-toric fibration associated to the same choice of toric model. The fibration is described in Theorem \ref{thm:ATstructure}; inspecting the construction of the symplectomorphism between it and $M$ in \cite[Section 6]{HK}, we see that $s_C$ is a Lagrangian sphere lying above the $j$th segment between the critical values with invariant direction the ray for $\bar{D}_i$.  (See \cite[Section 7.1]{Symington} for a more general discussion of such Lagrangians.)

We now want to prove the claim about $i_\ast \cO_C(a-1)$ in (2). Set $f =f_2 \circ f_1$, where $f_1$ is given by starting $(\bar{Y}, \bar{D})$ and blowing up the $i$th component of the divisor $j$ times, to get the log CY pair $(Y^\dagger, D^\dagger)$, say, and $f_2$ is given by making all the remaining blow-ups. 
Now notice there is an isomorphism 
  $$ i_\ast \cO_C(a-1) \simeq \{ e: \, \cO_{\Gamma_{i(j+1)}}(\Gamma_{i(j+1)}) \to \cO_{\Gamma_{ij}}(\Gamma_{ij}) \otimes f_2^\ast (\cO(D^\dagger))^{\otimes a}) \}.$$ 

On the other hand, $\cO_{\Gamma_{ij}}(\Gamma_{ij}) \otimes f_2^\ast (\cO(D^\dagger))$
is given by mutating $\cO_{\Gamma_{ij}}(\Gamma_{ij})$
 over 
 $f^\ast \cO(\bar{D}_1)$,
$\ldots$, 
$f^\ast \cO(\bar{D}_1+ \ldots, \bar{D}_{k-1})$: this follows from applying \cite[Corollary 2.10]{Bridgeland-Stern} to the exceptional collection for $(Y^\dagger, D^\dagger)$, and pulling everything back under $f_2$. 
 (We can ignore the $\Gamma_{i'j'}$ for $i' \neq i$.) 
This is the same as mutating it over 
the dual exceptional collection $f^\ast (\cO(\bar{D}_1+ \ldots, \bar{D}_{k-1})^\ast), \ldots,f^\ast (\cO(\bar{D}_1)^\ast), \cO$. Let $\varsigma_{k-1}, \ldots, \varsigma_0$ be the mirror thimbles. These can be glued together by iterative Polterovich surgery to get  
 the exact Lagrangian torus in the copy of $(\bC^\ast)^2$ mirror to $\bar{Y} \backslash \bar{D}$ (see \cite[Theorem 5.5]{HK}). Let  $\vartheta_{i,j}$ be as above.

Now inspect the proof of \cite[Proposition 6.3]{HK}: notice that the vanishing cycle $W_i = \partial (\vartheta_{i,j})$ intersects the vanishing cycle $\partial \varsigma_{k-1}$ transversally at a single point; moreover, the Polterovich surgery $\partial (\vartheta_{i,j}) \# \partial \varsigma_{k-1}$ intersects $\partial \varsigma_{k-2}$ transversally at a single point, and the same remains true as one iterates. 
On the other hand, in the case where two thimbles intersect transversally at a single point, mutating one over the other is the same as performing Polterovich surgery (c.f.~\cite{Seidel_knotted} and \cite[Section 18]{Seidel_book}); this also implies that the result of this iterative surgery has boundary a vanishing cycle which is still a copy of the $i$th meridian, by applying the total monodromy of the part of the fibration that is mirror to $(\bar{Y}, \bar{D})$, see \cite[Proposition 3.15]{HK}. 
It follows that $\cO_{\Gamma_{ij}}(\Gamma_{ij}) \otimes f^\ast (\cO(D^\dagger))$ is mirror to a sphere given by performing an iterated Polterovich surgery on $\vartheta_{i,j}, \varsigma_{k-1}, \ldots, \varsigma_0$, and then capping off the boundary of the resulting Lagrangian with $\vartheta_{i,j+1}$; topologically, this is clearly a Lagrangian sphere; call it $s_C'$. 
To relate $s_C'$ to $\sigma_E (s_C)$, recall that we also know that iterated Polterovich surgery on the $\varsigma_l$ gives an exact Lagrangian torus $T$ which can be taken to be a fibre of the Symington almost-toric fibration; and $\vartheta_{i,j}$ and $\vartheta_{i,j+1}$ are locally given by half-conormals to $S^1_{{v_i}^\perp}$ in $D^\ast_\epsilon T^2$, where $v_i \in \bR^2$ is the ray associated to $\bar{D}_i$ in the fan for $(\bar{Y}, \bar{D})$, and $S^1_{{v_i}^\perp}$ the image of $\bR{v_i}^\perp$ in $T^2 = \bR^2 / \bZ^2$ (see Section \ref{sec:hms_background1}).
In particular, we see that $s_C$ can be arranged to cleanly intersect $T^2$ along $S^1_{{v_i}^\perp}$, and that $s'_C$ is the result of an $S^1$ Lagrangian surgery between them (with suitable sign). Now notice that $\sigma_E (s_C)$ can be described in precisely the same way.

For $\sigma_E^{-1} (s_C)$, one can proceed similarly, taking the Polterovich surgery (with the other ordering choice) of $\vartheta_{i,j}$ with $\varsigma_0$, then $\varsigma_1$, etc, which will correspond to twisting $\cO_{\Gamma_{ij}}(\Gamma_{ij})$ with $\cO(-D)$, and taking the $S^1$ Lagrangian surgery between $s_C$ and $T^2$ with the opposite sign to before. More generally, notice that one can iterate this process by using disjoint parallel copies of $T^2$ in $D^\ast_\epsilon (T^2)$. (At most one of them can be exact, but this doesn't matter after surgery to get a Lagrangian sphere.)
\end{proof}

\begin{remark} \label{rmk:mirror_sphere_uniqueness}
While we won't need this in this paper, we expect that up to Hamiltonian isotopy, the Lagrangian sphere $s_C$ constructed above is independent of the choice of toric model contracting $C$. This would follow e.g.~from a strengthening of Proposition \ref{prop:change_of_toric_model}: we expect that the sequence of moves between any two toric models contracting $C$ can be chosen so that $C$ is contracted at each step. (For these purposes we consider any power of a fixed elementary transformation to be a single step.) Using the discussion of mirror moves in Section \ref{sec:change-toric-model-A-side}, we would then get that under the corresponding sequence of nodal slides and cuts, $s_C$ would remain a matching cycle between two nodal critical points at all time. 
\end{remark}

\subsection{Full mirror symmetry for Lagrangian translations}

\begin{theorem}\label{thm:mirror_autos}
Suppose $(Y,D) = (Y_e, D)$ is a log Calabi-Yau surface with distinguished complex structure, and $w: M \to \bC$ its mirror. Set $U = Y \backslash D$, let $\pi: M \to B$ be an almost-toric fibration on $M$, and $L_0$ a reference Lagrangian section of $\pi$.
 Assume that $L$ is a Lagrangian section of $\pi$ which agrees with $L_0$ near $\partial M$, and $\cL \in Q \subset \Pic Y$ the corresponding line bundle, as in Corollary \ref{cor:mirror_autos_K-theory}. Then under the homological mirror symmetry equivalence 
$D(U) \simeq D^b \cW(M)$ of Theorem \ref{thm:hms} (3), $\otimes \cL$ corresponds to $\sigma_L$. 

\end{theorem}

\begin{proof}
As it is compactly supported, $\sigma_L \in \pi_0 \Symp_c^\gr(M)$  gives an autoequivalence of $D^b \cF^\to (w)$ (using e.g.~the characterisation of $D^b \cF^\to (w)$ in  \cite{GPS_covariant}: we can define the category using exact Lagrangians with prescribed behaviour in a neighbourhood of the boundary but no conditions inside of this). 
Moreover, this restricts to the identity on $D^\pi \cF(\Sigma)$, where $\Sigma$ is a fibre of $w$ near infinity. 
Under mirror symmetry, we get an autoequivalence of $D(Y)$ which restricts to the identity on $\Perf D$, say $\check{\sigma}_L$. Consider $\check{\sigma}_L^{-1} \circ (\underline{\phantom{...}} \otimes \cL) \in \Auteq D(Y)$. 
 First, Corollary \ref{cor:mirror_autos_K-theory} implies that it acts as the identity on $K$-theory of $D(U)$; also, Proposition \ref{prop:Lagrangian_spheres}, together with Lemma \ref{lem:nodal_slide_Lag_section}, tells us that  it acts as the identity on $\cO_{C}(a)$ for any $(-2)$ curve $C$ in $U$ and any $a \in \bZ$. The proof of Proposition \ref{prop:detecting_id}, together with the classification result for chains of $(-2)$ curves (Lemma \ref{lem:type_A_sphericals}), then constrains  $\check{\sigma}_L^{-1} \circ (\underline{\phantom{...}} \otimes \cL)$ as an element of $\Auteq D(Y)$: it could only be a product of spherical twists in objects $\cO_{C'}(a)$ for $C' \subset D$ a $(-2)$ curve. In particular, it restricts to the identity  in $\Auteq D(U)$. 
\end{proof}

\begin{remark} If $D$ is indefinite or negative definite, we have $Q = \bar{Q}$, and  it follows from the proof above that $\otimes \cL$ corresponds to $\sigma_L$ under the equivalence $D(Y) \simeq D^b \cF^{\to}(w)$. In the negative semi-definite case, there is an overall ambiguity of $\otimes \cO(D)$, i.e.~the action of the Serre functor. The mirror to this is well-understood, going back to \cite{Kontsevich_ENS}: it's action of the total monodromy of the Lefschetz fibration on the directed Fukaya category, cf.~\cite{Seidel_A_infty_II} and follow-up papers. In the toric case, this was also carefully studied studied by Hanlon \cite{Hanlon}.
\end{remark}

\begin{remark}\label{rmk:Hanlon3}
In the non-compactly supported case, \emph{if} we knew that $\sigma_L$ acted both on $\cW(M)$ \emph{and} $D^b \cF^\to (w)$, then we would be able to apply Proposition \ref{prop:detecting_id} to show that $\sigma_L$ is mirror to tensoring with $\cL$. (This would make use of case (2) of the proposition, by considering the action of $\sigma_L$ on the central torus.) This crucially relies on having an action on $D^b \cF^\to (w)$, as the Prop.~\ref{prop:detecting_id} applies to autoequivalences of $D(Y)$. In order to do this, we need a condition on the behaviour of $L$ near infinity that is more rigid that simply being Liouville adapted (reflecting the difference between $\Pic Y$ and $\Pic U$); we expect monomial admissibility to be well suited for this.
\end{remark}

\section{Nodal slide recombinations}\label{sec:nodal_slide_recombinations}

Assume throughout this section that $(Y,D)=(Y_e,D)$. 

\subsection{Construction and mirror theorem}

Fix a toric model $(\bar{Y}, \bar{D}) \gets (\widetilde{Y}, \widetilde{D}) \to  (Y,D)$.  
Suppose $\varphi  \in \Aut(Y,D; \text{pt})$ is an automorphism of $Y$ fixing $D$ pointwise; this induces an automorphism of $\widetilde{Y}$ fixing $\widetilde{D}$ pointwise, which we also denote $\varphi $. Let $f: (\widetilde{Y}, \widetilde{D}) \to (\bar{Y}, \bar{D})$ denote the blow-down map. Now consider the map $f \circ \varphi  : (\widetilde{Y}, \widetilde{D}) \to (\bar{Y}, \bar{D})$. This also gives a toric model for $(\widetilde{Y}, \widetilde{D})$ (and indeed, $(Y,D)$). 
By Proposition \ref{prop:change_of_toric_model}, one can go between these via blow-ups and elementary transformations, determined by a factorisation of $\varphi $ into elementary \emph{birational} transformations of $(\widetilde{Y}, \widetilde{D})$. After performing toric blow ups to $(\widetilde{Y}, \widetilde{D})$, we can assume wlog that we're only making elementary transformations, say $\cE_1, \ldots, \cE_r$. Notice that combinatorially, on the mirror side, we `get back' to the Lefschetz fibration, and almost toric fibration, that we started with: the $n_i$ and $m_i$ are the same at the start and at the end.

\begin{theorem}\label{thm:nodal_slide_recombination}
Fix $\varphi \in \Aut(Y,D; \text{pt})$, and a factorisation of it into elementary transformations, say $\cE_1, \ldots, \cE_r$, as above. Then there exists a compactly supported symplectomorphism $\check{\varphi} \in \pi_0 \Symp^\gr_c (M)$ such that:

\begin{enumerate}
\item With respect to the Lefschetz fibration $w: M \to \bC$, $\check{\varphi }$  is induced by a compactly supported symplectomorphism of the base, relative to the singular values,  determined by the full sequence of Hurwitz moves for $\cE_1, \ldots, \cE_r$;

\item With respect to the almost-toric fibration on $M$ associated to our choice of toric model, $\check{\varphi }$ is induced by the sequence of nodal slides and cut transfers determined by $\cE_1, \ldots, \cE_r$ (using the standard reference Lagrangian section $L_0$, itself fixed). 

\item Under the homological mirror symmetry identification (Theorem \ref{thm:hms} (2) and (3)), the action of $\varphi_\ast$ on the category $D (Y)$, respectively $D(U)$, is mirror to the action of $\check{\varphi }$ on the category $D^b \cF^\to (w)$, respectively $D^b \cW(M)$. 
\end{enumerate}
\end{theorem}

\begin{proof}
Let's start with (1). Consider the explicit Lefschetz fibration associated to our choice of toric model. Using the same notation as before, it has a distinguished collection of vanishing cycles $\{ W_{i j} \}_{ i=1, \ldots, k, j=1, \ldots, m_i}$, $V_0$, \ldots, $V_{k-1}$, with vanishing paths, say, $\{ \gamma_{i j} \}, \varrho_0, \ldots, \varrho_{k-1}$. Now perform the Hurwitz moves for $\cE_1, \ldots, \cE_r$. (This follows Figure \ref{fig:elementary_transformation}.)  At then end, by our observation that the $n_i$ and $m_i$ are the same as at the start, the distinguished collection that we get has, again, ordered vanishing cycles $\{ W_{i j} \}_{ i=1, \ldots, k, j=1, \ldots, m_i}$, $V_0$, \ldots, $V_k$. The vanishing \emph{paths}, however, will have changed (indeed, the ordering of the critical points that they induce should in general also have changed). Call the new paths $\{ \gamma'_{i j} \}, \varrho'_0, \ldots, \varrho'_{k-1}$.

Up to compactly supported isotopy relative to the critical points, there exists a unique compactly supported orientation-preserving diffeomorphism of the base taking the $ \gamma_{i j}$ to $ \gamma'_{i j}$, the $\varrho_l$ to the $\varrho_l'$, and preserving a straight half-line going to $-\infty$ out of the central fibre. (Before requiring the latter, the map is only defined by to the ambiguity of rotating by $2k\pi$ in a large radius annulus.) Call this map $f_{\check{\varphi}} $. 
Similarly to the construction in \cite[Proposition 2.5]{Keating_monotone}, this induces a compactly supported symplectomorphism $\check{\varphi}$ of the total space $M$, uniquely determined up to Hamiltonian isotopy, such that 
\begin{itemize}
\item on a large compact set $K$ with $f(K) \supset \Supp f_{\check{\varphi}}$, it intertwines $w$, and lifts $f_{\check{\varphi}}$;
\item in a neighbourhood of the vertical boundary of the Lefschetz fibration above $f(K_1)$, it interpolates back between $f_{\check{\varphi}}$ and the identity (recall that after forgetting the marked points, $f_{\check{\varphi}}$ is isotopic to the identity)
\item it's the identity on fibres above points outside of $f(K)$. 
\end{itemize}
This completes (1). (The reader may also be interested in the set-up at the end of \cite[Section 3]{Torricelli} in a related setting.)

From the definition of $\check{\varphi}$, it maps the Lagrangian branes supported on the thimbles above $\{ \gamma_{i j} \}, \varrho_0, \ldots, \varrho_{k-1}$ to the ones above $\{ \gamma'_{i j} \}, \varrho'_0, \ldots, \varrho'_{k-1}$. 
On the other hand, under the HMS theorem, the former are mirror to 
\begin{multline*}
\cO_{\Gamma_{km_k}}(\Gamma_{km_k}), \cdots, \cO_{\Gamma_{k1}}(\Gamma_{k1}), \cdots, \cO_{\Gamma_{1m_1}}(\Gamma_{1m_1}),\cdots, \cO_{\Gamma_{11}}(\Gamma_{11}), \cO,
f^\ast \cO(\bar{D}_1), \\ \cdots, f^\ast \cO (\bar{D}_1+ \cdots + \bar{D}_{k-1})
\end{multline*}
where $f: (\widetilde{Y}, \widetilde{D}) \to  (Y,D)$ is part of the toric model datum. On the other hand, using the fact that the elementary transformations $\cE_1, \ldots \cE_r$ are factorising $\varphi$, the branes supported on the thimbles above $\{ \gamma'_{i j} \}, \varrho'_0, \ldots, \varrho'_{k-1}$ are mirror to 
\begin{multline*}
\cO_{\varphi(\Gamma_{km_k})}(\varphi(\Gamma_{km_k})), \cdots, \cO_{\varphi(\Gamma_{k1})}(\varphi(\Gamma_{k1})), \cdots, \cO_{\varphi(\Gamma_{1m_1})}(\varphi(\Gamma_{1m_1})),\cdots, \cO_{\varphi(\Gamma_{11})}(\varphi(\Gamma_{11})), \\ \cO,
f^\ast \cO(\bar{D}_1),  \cdots, f^\ast \cO(\bar{D}_1+ \cdots + \bar{D}_{k-1})
\end{multline*}
where we are using the fact that the $\bar{D}_i$ are fixed by $\varphi$. 
This completes (3).

For (2), consider the sequence of nodal slides and cut transfers induced by $\cE_1, \ldots, \cE_r$. We claim that as the $n_i$ and $m_i$ at the end are the same as the ones we started with, the sequence of almost-toric moves naturally induces a compactly supported symplectomorphism of $M$. Formally, this can be defined as follows. Set $M = M_0$. The first nodal slide gives a one-parameter family of symplectic manifolds $M_t$, $t \in [0,1]$, naturally identified outside a compact set, with a smoothly varying family of symplectomorphisms $\check{\varphi}_t: M_0 \to M_t$, all compactly supported (this makes sense because of the natural identification), and such that $\check{\varphi}_0 = \Id$. The first cut transfer simply takes $M_1$ to itself (no deformations). The second nodal slide extends our families to $M_t$ and $\check{\varphi}_t$ with $t \in [1,2]$. Iterating, we get smoothly varying families $M_t$ and $\check{\varphi}_t$ for $t \in [0,r]$ (wlog we reparametrise near integer points of $[0,r]$ to get smoothness everywhere). Now, after all $r$ steps, the $n_i$ and $m_i$ are the same as we started with; thus $M_r$ is naturally identified, as a symplectic manifold, with $M = M_0$; and $\check{\varphi}_r$ is a compactly supported symplectomorphism of $M$. 

We claim that under our identification of $M$ as the total space of the Lefschetz fibration $w$ and as the total space of the (initial) almost-toric fibration, $\check{\varphi}$ and $\check{\varphi}_r$ agree up to compactly supported Hamiltonian isotopy. This readily follows from the two descriptions together with the discussion in Section \ref{sec:change-toric-model-A-side}
of the compatibility of the Hurwitz moves on the Lefschetz side and the nodal slides on the almost-toric side.
 \end{proof}

\begin{remark} 
\label{rmk:N_change_local_systems}
 Recall the short exact sequence from Theorem \ref{thm:GHK_Torelli}:
$$
1 \to N = \Hom (\pi_1(Y_e \backslash D), \bC^\ast) \to \Aut(Y_e, D; \text{cpt}) \to \Adm / W \to 1.
$$
We've just seen that for any element of $\Adm / W$, we can construct a mirror (graded) symplectomorphism $\check{\varphi}$. The autoequivalence of the Fukaya category mirror to an element of $N$ is clear: it is given by the change of  local systems encoded by the corresponding element of $\Hom (\pi_1(M), \bC^\ast)$. 
\end{remark}

\begin{remark}\label{rmk:Lefschetz_isotopy}
On the Lefschetz side, we could construct the $M_t$ and $\check{\varphi}_t$, $t \in [0,r]$, by deforming the base of the Lefschetz fibration by dragging around critical values in the base in the way prescribed by the Hurwitz moves. 
See Figure \ref{fig:nodal_isotopy_Lefschetz} for a single nodal slide.
\begin{figure}[htb]
\begin{center}
\includegraphics[scale=0.38]{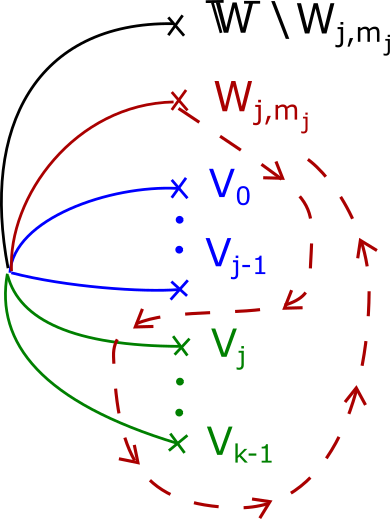}
\caption{Deforming the base of the Lefschetz fibration to get the isotopy from $M$ to $M_1$.}
\label{fig:nodal_isotopy_Lefschetz}
\end{center}
\end{figure}
\end{remark}

\subsection{Independence of choices} 

We want to prove that the map $\check{\varphi}$ constructed in Theorem \ref{thm:nodal_slide_recombination} is independent of choices. In order to do this, we start by understanding the mirrors to $SL_2(\bZ) \subset \Aut(\bC^\ast)^2 \subset \Bir ((\bC^\ast)^2)$. 

\subsubsection{Mirrors to $SL_2(\bZ)$ transformations} \label{sec:SL_2(Z)_transformations}

The $SL_2(\bZ)$ action on $(\bC^\ast)^2$ can be described as follows. The $B$ side is standard: for a given arbitrary toric pair  $(\bar{Y}, \bar{D})$, any element $g$ of $SL_2(\bZ)$  gives a birational transformation of $\bar{Y}$. After sufficiently many toric blow ups, say to $(\bar{Y}', \bar{D}')$, $g$ lifts to a biholomorphism from $\bar{Y}'$  to, say, $\bar{Y}''$. 
If $(\widetilde{Y}', \widetilde{D}')$ is given by interior blow-ups on $(\bar{Y}', \bar{D}')$ in $m_1, \ldots, m_k$ points, then $g$ induces an automorphism from $(\widetilde{Y}', \widetilde{D}')$ to $(\widetilde{Y}'', \widetilde{D}'')$, given by blowing up $(\bar{Y}'', \bar{D}'')$ according to the cyclic permutation of the $m_i$ determined by $g$. (If $g$ preserves the fan for $(\bar{Y}', \bar{D}')$ and the $m_i$ are cyclically symmetric, we get an automorphism of  $(\widetilde{Y}', \widetilde{D}')$ to itself.)

On the A side, let $M$ be the mirror symplectic manifold to  $(\widetilde{Y}', \widetilde{D}')$, and let  $M'$ be the mirror to  $(\widetilde{Y}'', \widetilde{D}'')$. 
By assumption, we can set up the almost-toric fibrations on $M$ and $M'$ so that one base is mapped to the other by $g$, with critical points and invariant lines mapped to each other. This naturally lifts to a symplectomorphism from $M$ to $M'$, say $\check{g}$.

By using suitable Liouville forms, one could arrange for $\check{g}$ to induce a map from $\cW(M)$ to $\cW(M')$. If we knew that there was a compatible map of directed Fukaya categories, Proposition \ref{prop:detecting_id} would apply to show that $\check{g}$ corresponds to $g$ under homological mirror symmetry. (Note that this would not require a fine-level description of the action on those categories!) This should follow from \cite{Hanlon-Ward}, who consider a version of the directed Fukaya category defined in terms of sections of the almost-toric fibration.
 In the current paper, we 
only use this for conjugates of compactly supported maps by $\check{g}$, which as usual follows from Proposition \ref{prop:detecting_id} together with Propositions \ref{prop:Corti-Filip-Petracci}  and \ref{prop:Lagrangian_spheres}. 

\begin{lemma} Let $g$ and $\check{g}$ be as above. Assume that $\check{h} \in \pi_0 \Symp^\gr_c M$ is a symplectomorphism of $M$ mirror to an autoequivalence $h \in \Auteq D(Y)$. Then $\check{g}^{-1} \circ \check{h}\circ \check{g} \in \pi_0 \Symp_c M$ is mirror to $g^{-1} \circ h \circ g \in  \Auteq D(Y)$. 
\end{lemma}

\begin{remark}\label{rmk:SL_2(Z)_alternative}
 Alternatively, to get a well-defined action on directed categories, we could set up a mirror  $\tilde{g}$ to $g$  in terms of the Lefschetz fibration by combining the almost-toric / Lefschetz identification with \cite[Proposition 3.19]{HK}. The map $\tilde{g}$ would be a fibred symplectomorphism constructed by combining a rotation of the fibre near infinity with an automorphism of the base of $w$ (similarly to the proof of Theorem \ref{thm:nodal_slide_recombination}). While this is less geometrically appealing,  such a map $\tilde{w}$ will by construction automatically act on $D^b \cF^{\to} (w)$, and be mirror to $g \in \Auteq D(Y)$ under our HMS correspondence. While conditions nears infinity are different for $\tilde{g}$ and $\check{g}$ (heuristically, they are suited for different versions of the directed Fukaya category), they will agree on arbitrarily large compact sets. (In particular, they will be interchangeable for the purposes of conjugating a compactly supported map.)
\end{remark}

\subsubsection{Proof of independence of choices}

\begin{proposition}\label{prop:indep_of_choices}
The (equivalence class of) symplectomorphism $\check{\varphi} \in \pi_0 \Symp^\gr_c (M)  $ constructed in Theorem \ref{thm:nodal_slide_recombination} is independent of choices.
\end{proposition}

In order to prove this, it is enough to show that two different factorisations of $\varphi$ into elementary transformations induce mirror symplectomorphisms which agree up to a compactly supported Hamiltonian isotopy. (This also takes care of the choice of toric model.)
This means that we need to check that relations in the group of birational automorphisms of $(\bC^\ast)^2$ induce isotopic symplectomorphisms. We start with the most interesting one.

\subsubsection*{The `$A_2$ cluster relation'}
There is a well-known order five relation in the group of birational transformations of $(\bC^\ast)^2$, which experts will recognise as corresponding to the fact that there are exactly five cluster charts for the $A_2$ quiver. To realise this relation (minimally) in terms of automorphisms of a log CY pair, take the model $(\widetilde{Y}, \widetilde{D}) \to (\bar{Y}, \bar{D})$ where $ (\bar{Y}, \bar{D})$ is a corner blow-up of $\bP^1 \times \bP^1$, and $(\widetilde{Y}, \widetilde{D})$ is given by one interior blow up on each of the self-intersection zero components. (This is a compactification of  the $A_2$ cluster variety.) The mirror almost-toric fibration is given in the left-hand side of Figure \ref{fig:relation_initialATF}; label the boundary components as in the figure.
 (Both the $A$ and $B$ side spaces, and the sequences of mutations, are carefully described in  \cite{Cheung-Vianna}.)

\begin{figure}[htb]
\begin{center}
\includegraphics[scale=0.5]{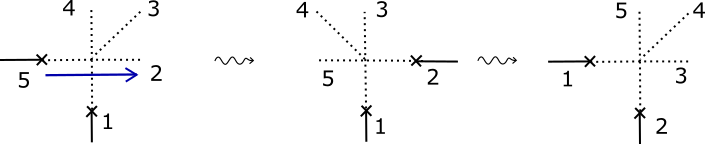}
\caption{Almost toric fibration for the `$A_2$ quiver relation': initial configuration, and mirrors to $E$ and then $g$. The cuts are in black, and dotted lines correspond to fan vectors (including ones with no interior blow-ups).
}
\label{fig:relation_initialATF}
\end{center}
\end{figure}

Let $E$ be the elementary transformation at $\bar{D}_5$, and let $g$ be the $SL_2(\bZ)$ birational transformation which takes $\bar{D}_i$ to $\bar{D}_{i-1}$. Let $P = g  E$; this defines an automorphism of $\widetilde{Y}$ (rather than a mere birational transformation), permuting the components of $D$; we have $P^5 = \Id$. 
Equivalently, setting $E_a = g^{-a} E g^{a}$, we have $E_4 \ldots E_0 = \Id$ (note $g^{-4} = g$). On the symplectic side, we have the following.

\begin{proposition}\label{prop:symplectic_P5=I_relation}
Let $(\widetilde{Y}, \widetilde{D}) \to (\bar{Y}, \bar{D})$ be as above, and let $M$ be its mirror. Let $\check{\varphi}$ be the compactly supported symplectomorphism of $M$ associated to the factorisation of $\Id \in \Aut (\widetilde{Y})$ into $E_4 \ldots E_0$; then $\check{\varphi}$ is isotopic to the identity through a compactly supported Hamiltonian. 
\end{proposition}
Symplectic incarnations of the five $A_2$ cluster charts have been extensively studied, going back to \cite{EHK} -- see e.g.~\cite{Pan, STWZ, Treumann-Zaslow, GSW, STW, Cheung-Vianna}. Nevertheless, Proposition \ref{prop:symplectic_P5=I_relation} doesn't readily follow from the existing literature, as the focus has instead largely been on proving that symplectic objects `mirror' to the five charts, such as Lagrangian fillings of the right-hand trefoil, are distinct. (Note that by Theorem \ref{thm:nodal_slide_recombination} we already know that $\check{\varphi}$ acts as the identity on the Fukaya category.) We instead give a self-contained proof of Proposition \ref{prop:symplectic_P5=I_relation}, and comment further on the relation with the contact geometry literature in Remark \ref{rmk:contact_pentagon}. 

\begin{proof} of  Proposition \ref{prop:symplectic_P5=I_relation}.
The isotopy between $\check{\varphi}$ and the identity is hard to see directly from the perspective of the almost-toric fibration. We work instead with the Lefschetz fibration $w: M \to \bC$. 

As decribed in Theorem \ref{thm:nodal_slide_recombination}, $\check{\varphi}$ is a fibred symplectomorphism with respect to $w$; as before, let $f_{\check{\varphi}}$ be the automorphism of the base. We calculate this by hand. The key is to use, first, a more symmetric description of $w$, given in Figure \ref{fig:relation_initialLefschetz}. 

\begin{figure}[htb]
\begin{center}
\includegraphics[scale=0.38]{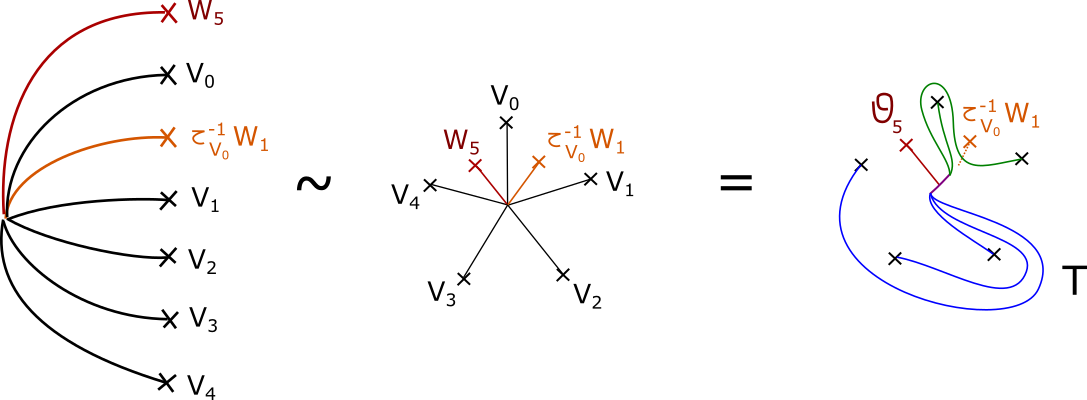}
\caption{Lefschetz fibration $w: M \to \bC$: standard (left), cyclically symmetric viewpoint (center), and in relation to the first mutation (same colour-coding as Figure \ref{fig:torus_for_mutation_1}).}
\label{fig:relation_initialLefschetz}
\end{center}
\end{figure}

We can now iteratively apply the isotopy corresponding to a single elementary transformation (in Figure \ref{fig:nodal_isotopy_Lefschetz}), combined with \cite[Proposition 3.19]{HK} for cycling the labels on components of $\bar{D}$, to get a description of $f_{\check{\varphi}}$; this is carried out in Figure \ref{fig:relation_baseisotopy}.

\begin{figure}[htb]
\begin{center}
\includegraphics[scale=0.5]{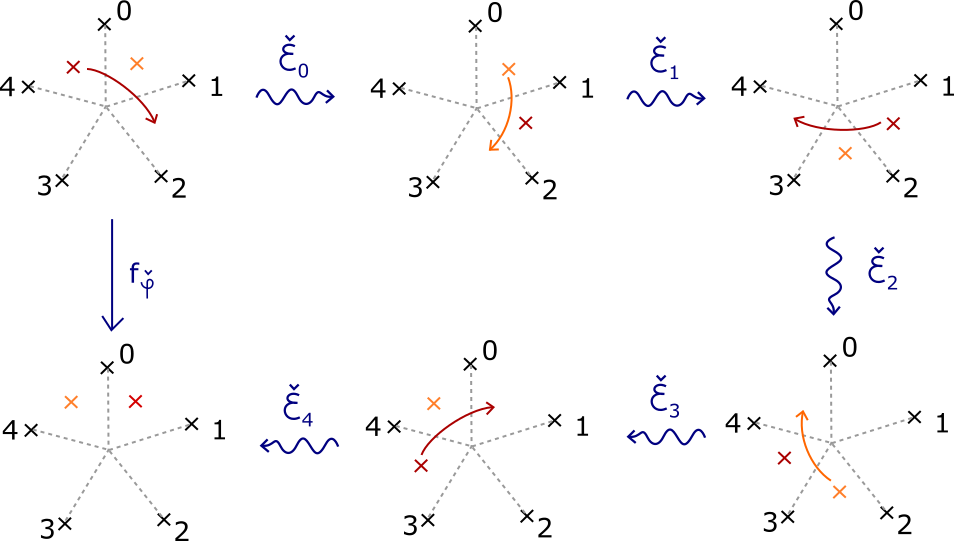}
\caption{Calculating $f_{\check{\varphi}}$ by applying the isotopies corresponding to $\cE_0, \ldots, \cE_4$. We're using the same cyclically symmetric configuration as in Figure \ref{fig:relation_initialLefschetz}: $T$ is given by gluing the blue and green thimbles along the purple arc.}
\label{fig:relation_baseisotopy}
\end{center}
\end{figure}

\begin{figure}[htb]
\begin{center}
\includegraphics[scale=0.55]{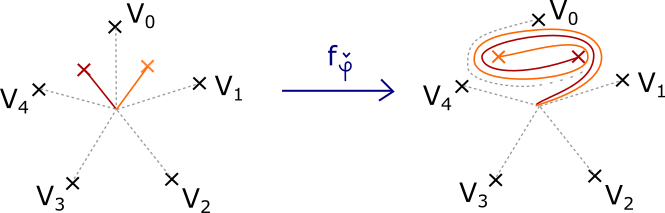}
\caption{Upshot: $f_{\check{\varphi}}$ is given by five half-twists between two vanishing cycles}
\label{fig:relation_baseoverall}
\end{center}
\end{figure}

The upshot is given in Figure \ref{fig:relation_baseoverall}: $f_{\check{\varphi}}$ is given by five half-twists between the critical points for $W_5$ and $\tau^{-1}_{V_0} W_1$. On the other hand, one can immediately check that these two vanishing cycles are in a local 
$A_2$ configuration; in particular, we can deform $w$ so that these two critical points merge (along their vanishing paths) to give a single $A_2$ critical point, at which point the half-twists readily unravel to give an isotopy to the identity. (Alternatively, as the sub-Lefshetz fibration on which $\check{\varphi}$ has support is just a copy of the `standard' $A_2$ fibration, with total space $\bC^2$, one could  quote \cite{Gromov}.)
\end{proof}

\begin{remark}\label{rmk:contact_pentagon} 
Intuitively, the key to Proposition \ref{prop:symplectic_P5=I_relation} is to show that $\check{\varphi}(T)$ is Hamiltonian isotopic to $T$. This can be checked using tools from contact geometry, which we now sketch. In fact, a slightly stronger statement is true: we'll see that the statement holds relative to a disc on $T$ (fixed pointwise under the five mutations and then the Hamiltonian isotopy), or, equivalently, we can work with a once-punctured torus. First, recall that the max-tb right-handed Legendrian trefoil knot $K \subset S^3 = \partial (B^4, \omega_0)$ has at least five exact Lagrangian fillings, all once-punctured tori (\cite[Example 8.4]{EHK}, see also \cite{Pan, Treumann-Zaslow}). On the other hand, the Weinstein domain given by attaching a Weinstein 2-handle along $K$ is precisely $M$ -- see e.g.~\cite[Section 4.1]{Casals-Murphy}, or the later conceptual argument in \cite[Corollary 1.2]{Casals}. One can then proceed, for instance, using the framework of weaves introduced in \cite{Casals-Zaslow}: first, the five Lagrangian fillings can be described in terms of 2-weaves, in this case trivalent graphs, see \cite[Section 7.1.3]{Casals-Zaslow}; the fillings are given by the 2-weaves dual to the five triangulations of the pentagon. Second, Lagrangian mutations can be understood in that language: they are given by Whitehead moves on the trivalent graph, dual to flips in the triangulation, see \cite[Theorem 4.21]{Casals-Zaslow}. Applying this, we get  that starting with any of our five Lagrangian fillings and performing the sequence of five mutations gives a filling which is Hamiltonian isotopic, relative to the boundary $K$, to the filling we started with. 
\end{remark}

We can now conclude:

\begin{proof} of Proposition \ref{prop:indep_of_choices}. 
Relations for the group of birational transformations of $(\bC^\ast)^2$ were described by Blanc in \cite[Theorem 2]{Blanc}. In our case, they consist of, first, relations in $SL_2(\bZ)$. These are readily satisfied on the symplectic side: the conjugate by $g \in SL_2(\bZ)$ of an elementary transformation with respect to ray $v_j$  is the elementary transformation with respect to ray $g \cdot v_j$, and the mirror statement is automatic. Second, Blanc's list gives us two relations involving elementary transformations which need checking on the symplectic side. The first one (`$PCP = I$' in his notation) amounts to checking that mutating and then mutating back gives the identity, which is automatic geometrically; and the second relation (`$P^5 = 1$' in his notation) was covered by Proposition \ref{prop:symplectic_P5=I_relation}. 
\end{proof}

\subsection{Examples} \label{sec:nodal_slide_examples}

We calculate the first few explicit examples of nodal slide recombinations.

\begin{example}\label{ex:k=6case} \emph{$B_2$ cluster variety relation.} Consider the mirror $M$ to the pair $(Y,D)$ given in Figure \ref{fig:nodal_slide_example_1}. (It has $k=6$, with three interior blow-ups; experts will recognise a standard compactification of the $B_2$ cluster variety, see \cite{Cheung-Vianna}.) $M$ has a non-compactly supported symplectomorphism $\rho$ described in the figure; this consists of a double nodal slide and cut transfer, followed by a single one, followed by an $SL_2(\bZ)$ map. By construction, $\rho^3$ has compact support; as $\Adm / W = \{ 1 \}$ in this case, $\rho^3$ is Hamiltonian isotopic to the identity by Theorem \ref{thm:nodal_slide_recombination}. This gives a symplectic mirror to what is sometimes known as the `$B_2$ cluster variety relation'. 
 \begin{figure}[htb]
\begin{center}
\includegraphics[scale=0.6]{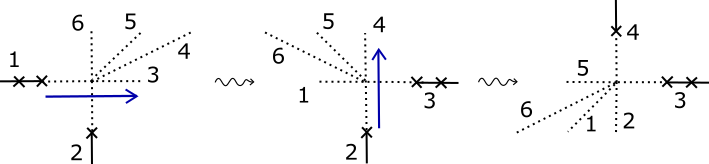}
\caption{Nodal slides for Example \ref{ex:k=6case}; the $SL_2$ action is given by taking the $i$th ray on the right-hand side to the $(i-2)$th, with indices mod 6.}
\label{fig:nodal_slide_example_1}
\end{center}
\end{figure}

\end{example}

Assume that $(Y,D)$ is negative semi-definite, so that $Y$ is a rational elliptic surface, and $M = X \backslash E$ the complement of a smooth anticanonical divisor in a del Pezzo surface. From Section \ref{sec:Mordell-Weil}, we know that $\Auteq(Y,D; \text{pt})$ is infinite in two cases, and, in both, equal to $k \cdot MW(Y, \bP^1)$, where $MW(Y, \bP^1) \simeq \bZ$ is the Mordell--Weil group of $Y$. For both of these cases, we can factorise a generator for $MW(Y)$ and write down the mirror sequence of nodal slides, together with the $SL_2(\bZ)$ action corresponding to the permutation of the components of $D$. 

\begin{example}\label{ex:k=7case}\emph{Mirror to Mordell--Weil for $\Bl(\bP^1 \times \bP^1) \backslash E$.}
Assume that $k=7$; $X$ is the blow-up of $\bP^1 \times \bP^1$ in a point. A generator $\varphi$ for $MW(Y, \bP^1)$ gives the symplectomorphism $\check{\varphi}$ of $M$ described by the sequence of nodal slides in Figure \ref{fig:nodal_slide_example_3}, postcomposed by the $SL_2(\bZ)$ action encoded by the labeling of the invariant rays. The map $\check{\varphi}^7$ has compact support, and, by Theorem \ref{thm:nodal_slide_recombination}, is mirror to $\varphi^7$. 
\begin{figure}[htb]
\begin{center}
\includegraphics[scale=0.55]{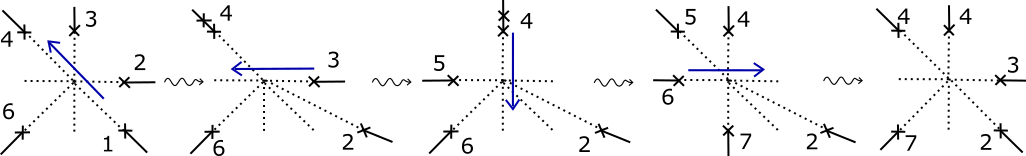}
\caption{Nodal slides for Example \ref{ex:k=7case}; the $SL_2$ action is determined by taking to $i$th ray in the right-hand side diagram to the $(i+1)$th.}
\label{fig:nodal_slide_example_2}
\end{center}
\end{figure}
\end{example}

\begin{example}\label{ex:k=8case}\emph{Mirror to Mordell--Weil for $\bF_1 \backslash E$.}  
Assume that $k=8$ and $X = \bF_1$. A generator $\varphi$ for $MW(Y, \bP^1)$ gives the symplectomorphism $\check{\varphi}$ of $M$ described by the sequence of nodal slides in Figure \ref{fig:nodal_slide_example_3}, postcomposed by the $SL_2(\bZ)$ action encoded by the labeling of the invariant rays. The map $\check{\varphi}^8$ has compact support, and, by Theorem \ref{thm:nodal_slide_recombination}, is mirror to $\varphi^8$. 
\begin{figure}[htb]
\begin{center}
\includegraphics[scale=0.55]{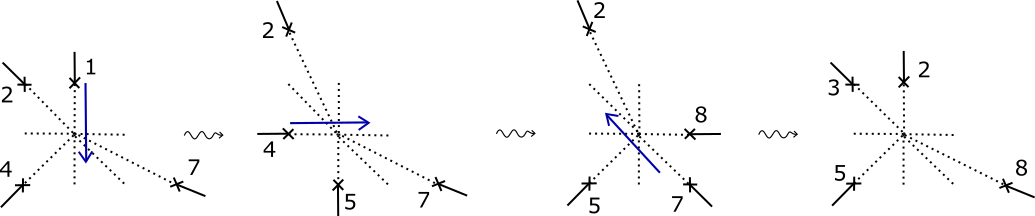}
\caption{Nodal slides for Example \ref{ex:k=8case}; the $SL_2$ action is determined by taking to $i$th ray in the right-hand side diagram to the $i+1$th.}
\label{fig:nodal_slide_example_3}
\end{center}
\end{figure}
\end{example}

\subsection{Relations between different families of symplectomorphisms}\label{sec:relations}

Homological mirror symmetry immediately hands us symplectic versions of Lemma \ref{lem:autoequivalence_relations} at the level of Fukaya categories.  The first two relations automatically have geometric counterparts: given a Lagrangian sphere $S$ and a symplectomorphism $f$, $\tau_{f(S)} = f \circ \tau_S \circ f^{-1} \in \pi_0 \Symp^\gr_c M$.  We check the other two.

\begin{lemma}\label{lem:translation_slide_relation}
Suppose that $\sigma_L \in \pi_0 \Symp^\gr_c M$ is the Lagrangian translation associated to $[L -L_0] \in H_2(M) /  (\bZ \cdot \gamma^\vee) $ and $\check{\varphi}\in \pi_0 \Symp^\gr_c M$ a nodal slide recombination. Then 
$$
\sigma_{\varphi_\ast L} = \varphi  \circ \sigma_{ L} \circ \varphi^{-1} \quad \in \pi_0 \Symp_c M.
$$
Similarly with non-compactly supported Lagrangian translations.
\end{lemma}

\begin{proof}
This follows from iteratively applying Lemma \ref{lem:nodal_slide_Lag_section}. Similarly with non-compactly supported Lagrangian translations.
\end{proof}

Suppose $C \subset Y \backslash D$ is a $(-2)$ curve and $a \in \bZ$. Recall that as autoequivalences of derived categories,
$$T_{\cO_C (a-1)} \circ T_{\cO_C (a)} = \underline{\phantom{...}}\otimes \cO_Y (C).$$
We will prove a symplectic analogue of this relation in Lemma \ref{lem:two_twists_Lag_translation}. Let's start with the set-up.

Choose an almost-toric fibration $\pi: M \to \bR^2$ as in Proposition \ref{prop:Lagrangian_spheres}, and let $s_{a} = \sigma^{a+1}_E s_C$ be the sphere mirror to $i_\ast \cO_C(a)$ constructed therein; these all project to a segment $\gamma$ between two critical points of $\pi$. This fibration comes with a prefered reference Lagrangian section $L_0$; using this, let $\sigma_{L_C}$ be the Lagrangian translation mirror to tensoring with $\cO_Y(C)$ from Theorem \ref{thm:mirror_autos}. 

 Let $B_\loc \subset \bR^2$ be a small open neighbourhood of $\gamma$, and let $M_\loc = \pi^{-1}(B_\loc)$.  (This should be thought of as mirror to the formal neighbourhood of a $(-2)$ curve, or to the resolution of an $A_1$ singularity, cf.~\cite{Chan-Ueda}.) $M_\loc$ is the total space of a Lefschetz fibration well-known to experts, which we briefly recall.
First, consider
$$M'_\loc =  \{ xy + (z-1)(z-2) = 0 \} \subset \bC^2 \times \bC^\ast
$$
equipped with the Kaehler form with potential $|x|^2 + |y|^2 + (\log |z|)^2$. Let $f: M'_\loc \to \bC^\ast$ be projection to $z$, a conic Lefschetz fibration with two critical fibres. There is a symplectic $S^1$ action given by $(x,y,z) \mapsto (e^{i\theta}x, e^{-i\theta}y,z)$. Following  \cite[Section 5.1]{Auroux_Gokova}, let $\delta_{z_0}(x_0,y_0)$ be the signed symplectic area between the equator $\{ |x| = |y| \} \subset f^{-1}(z_0)$ and the $S^1$ orbit of $(x_0,y_0,z_0)$. Then the map $\pi': M'_\loc \to \bR^2$ given by $\pi'(x,y,z)= (|z|, \delta_z(x,y))$ defines a singular Lagrangian torus fibration: symplectic parallel transport with respect to $f$ is $S^1$-equivariant and preserves $\delta$. There are two nodal critical fibres, the preimages of $(1,0)$ and $(2,0)$, and one then readily gets that after suitably truncating, $M'_\loc$ agrees with $M_\loc$. Moreover, the spheres $s_a$ can be represented as matching cycles with respect to $f$, as in Figure \ref{fig:formal_nhood_spheres}.

\begin{figure}[htb]
\begin{center}
\includegraphics[scale=0.35]{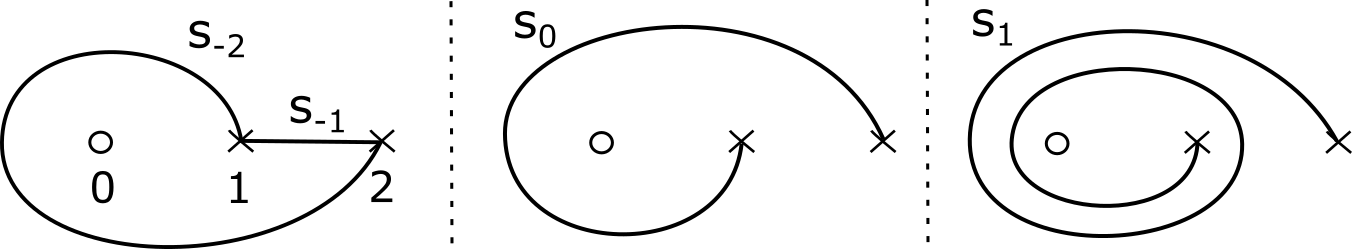}
\caption{Lagrangian spheres $s_a$: matching paths in the base of the Lefschetz fibration $f: M'_\loc \to \bR^2$. The matching path for $s_a$ can be obtained from the one for $s_{a-1}$ by applying a full negative twist in the segment from 0 to 1.}
\label{fig:formal_nhood_spheres}
\end{center}
\end{figure}

\begin{lemma}\label{lem:two_twists_Lag_translation}
Suppose that we are given $\sigma_{L_C} \in \pi_0 \Symp_c (M)$ and Lagrangian spheres $s_a$, $a \in \bZ$ as above. Then
$$
\tau_{s_{a-1}} \circ \tau_{s_a} = \sigma_{L_C} \in  \pi_0 \Symp_c (M).
$$
\end{lemma}

\begin{proof}
It's enough to prove this for $a=0$. 
Set $\rho= \tau_{s_{-1}} \circ \tau_{s_0}$.
The local model for the Dehn twist in the matching cycle $s_a$ is well known: it is induced by the half-twist $\tw_a$  in the associated matching path (as an isotopy class of automorphism of $\bC^\ast$ relative to the marked points),  see e.g.~\cite{Seidel_more, Seidel_LES, AMP}. In the case at hand, recall that we have a two-parameter family of Lagrangian tori, with each one fibred over one of the concentric circles $\eta_c = \{ |z| = c \}$, for some ${c \in \bR_+}$. The family $\{ \eta_c \}$ foliates the base of the Lefschetz fibration $f$. On the other hand, the image of the foliation under $\tw_{-1} \circ \tw_0$ is isotopic relative to marked points to the original one. See Figure \ref{fig:formal_nhood_map}. (This is familiar to experts, see e.g.~\cite[Remark 11.11]{Seidel_lectures} for the observation for a single $\eta_c$.) 
\begin{figure}[htb]
\begin{center}
\includegraphics[scale=0.3]{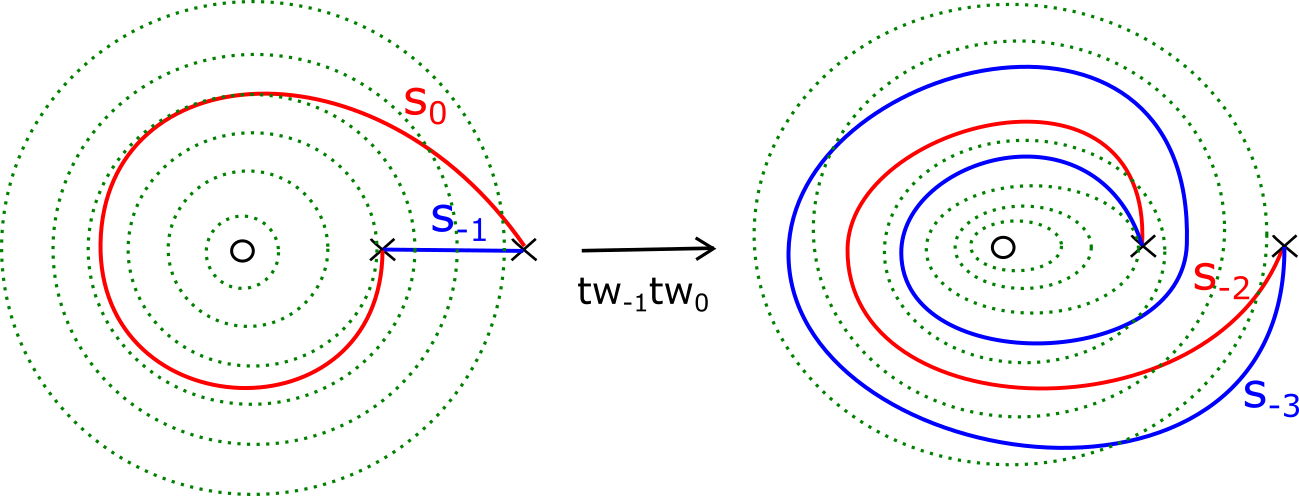}
\caption{Action of $\tw_{-1} \circ \tw_0$ on the base of the Lefschetz fibration $f: M'_\loc \to \bR^2$, with the images of $s_0$, $s_{-1}$, and a sample of circles $\gamma_c$.}
\label{fig:formal_nhood_map}
\end{center}
\end{figure}
This implies that after post-composing with a Hamiltonian isotopy, we can arrange for $\tau_{s_{-1}} \circ \tau_{s_0}$ to respect the almost-toric fibration $M_\loc \to B_\loc$. The Arnol'd--Liouville theorem then strongly constrains the map: it must be the Lagrangian translation determined by the image of a reference section $L_0$ (the argument extends over the nodal fibres similarly to the proof of Proposition \ref{prop:Lag_translations}). Reading off the action on homology is then enough to identify this translation with $\sigma_{L_C}$. 
\end{proof}

\section{Applications}\label{sec:applications}

\subsection{New symplectomorphisms}\label{sec:new_symplectos}
Fix a log CY surface $(Y_e, D)$ with mirror $M$. Let $Q$, $\Adm$ and $W$ be as before, and let $Q_\symp$, $\Adm/W_{\symp}$ and  $\Br_\symp$ be the subgroups of $\pi_0 \Symp_c^\gr (M)$ generated by respectively, compactly supported Lagrangian translations, nodal slide recombinations, and  Dehn twists in arbitrary Lagrangian spheres. (For the latter we do not restrict ourselves to the `known' spheres constructed in Proposition \ref{prop:Lagrangian_spheres}.) 

\begin{theorem} \label{thm:new_symplectos} 
Given the above setting, we have the following:

\begin{enumerate}

\item The map 
$Q \rtimes \Adm/W \to \pi_0 \Symp_c^\gr(M) $  with image $Q_\symp \rtimes \Adm/W_{\symp}$ is faithful. 

\item 
Let $\langle \Phi \rangle \leq  Q$ be the sublattice generated by classes of roots. Then $$\Br_\symp \cap (Q_\symp \rtimes \Adm/W_{\symp}) = \langle \Phi \rangle_\symp \rtimes \{ 1\} $$ 

\end{enumerate}

In particular, whenever  $\Adm \neq W$ or $Q \neq \langle \Phi \rangle$, we get elements of $\pi_0 \Symp_c^\gr (M)$ which are not products of Dehn twists. 

\end{theorem}

\begin{proof}
(1). We have a well-defined geometric action by Lemma \ref{lem:translation_slide_relation}; faithfulness holds at the level of the action on the homology group $H_2(M; \bZ)$. 

(2). After the HMS identification, the $K$ theory classes of Lagrangian spheres are the same as the ones in Remark \ref{rmk:classes_sphericals}. 
Also, recall that on the mirror side, $\Br_Z \cap \Aut (Y) = \{ \Id \} \in \Auteq (Y)$. The claim then follows from mirrors of the relations between autoequivalences in Lemma  \ref{lem:autoequivalence_relations}, notably Lemma \ref{lem:two_twists_Lag_translation}. 
\end{proof}

\begin{remark} The classification in 
Remark \ref{rmk:classes_sphericals} allows us to exclude arbitrary spherical twists, irrespective of whether they are given geometrically by Dehn twists in embedded Lagrangian spheres. Twists in other Lagrangians with periodic geodesic flow are ruled out for cohomological and dimension reasons. Less is known in general about the categorical action of fibred twists (see however \cite{Perutz, WW, Mak-Wu}). The only known fibred twists for our family of manifolds $M$ are  boundary twists in the semi-definite (simple elliptic) cases; these are given by a finite-time action of the Reeb flow on the boundary. 
(For hypersurface simple elliptic singularities, they are products of Dehn twists \cite[Section 4.c]{Seidel_graded}.)
 In general, they act trivially on relative homology, and also on $\cW(M)$. We expect them to correspond to elements of $Q$ which vanish in $\bar{Q}$. 
\end{remark}

\begin{example} Nodal slide recombinations: $\Adm / W$ is non-trivial in the following cases.

\begin{itemize}

\item Small values of $k$: for $k \leq 6$, $\Adm / W$ is always trivial, and there are no interesting nodal slide recombinations. On the other hand,  for $k=7$, $\Adm / W$ is always infinite; this is also true when $k=8$ and $\pi_1(U) = 0$.

\item In general, note that if $\Adm_Y/W_Y$ is non trivial for a log CY pair $(Y,D)$, and a log CY pair $(Y',D')$ is obtained from $(Y,D)$ through any sequence of toric and interior blow-ups, then $\Adm_{Y'} / W_{Y'}$ is also non trivial -- indeed, pullback gives an injection $\Adm_Y/W_Y \to \Adm_{Y'} / W_{Y'}$. In particular, the examples for small values of $k$ can be used to generate many more. 

\item When $Y=Y_e$, $\Aut(Y_e)$ acts with rational polyhedral domain on the nef cone \cite{Li_thesis}. This implies that $\Adm/W$ is infinite if and only if the nef cone is not rational polyhedral.

\end{itemize}

\end{example}

\begin{example} $\bar{Q}$ is not typically generated by classes of simple roots. Some specific examples include:

\begin{itemize}

\item Semi-definite cases with $D$ a cycle of $(-2)$ curves and $k \geq 7$  \label{ex:k=8ctd}, except when $k=8$ and $\pi_1 (M) = \bZ/2$. 

\item Whenever the rank of $Q$ is one and $D$ is negative (semi-)definite.

\item Examples 4.3 and 4.4. in \cite{Friedman_amplecone}. Both have the property that $Q$ has rank 2 and the root system is empty.

\end{itemize}

\end{example}

\begin{remark}\label{rmk:semidef_compactify} 
In the semi-definite case, a compactly supported symplectomorphism of $M$ induces a symplectomorphism of its compactification $X$, a del Pezzo surface with Kaehler anti-canonical form. However, in our cases the induced symplectomophisms of $X$ act trivially on homology, and thus are isotopic to the identity \cite{Gromov, LP, Evans}. 
\end{remark}

\subsection{Lagrangian spheres and symplectic Torelli group}\label{sec:applications_to_sphericals}

For small values of $k$, a lot is explicitly known about $(-2)$ curves in $Y \backslash D$; as soon as a toric model for $(Y,D)$ involves two interior blow ups on the same  component of $\bar{D}$, there is at least one $(-2)$ curve.
 Moreover, suppose that $D$ is negative definite and does not contain a $(-1)$ curve. Then for $k \leq 5$, there are finitely many $(-2)$ curves on $Y_e$, which give a basis of $Q$ \cite{Looijenga}. For $k=6$,  the $(-2)$ curves span $Q$, with exactly one relation between them; in particular, there's also finitely many of them \cite{Li_thesis, Simonetti_thesis}.  
For $k=7$, there are always infinitely many $(-2)$ curves; for $k=8$ there are infinitely many whenever the fundamental group of the complement $U_e$ is trivial. (For $k \geq 9$ little has been done systematically, though individual cases can be done by hand.)

\subsubsection{Infinitely many $(-2)$ curves} Assume that $Y \backslash D$ contains infinitely many $(-2)$ curves (necessarily countable). This implies we are in the negative definite case.

\begin{corollary}\label{cor:infinite_spheres}
Let $(Y,D) = (Y_e, D)$ be a log Calabi-Yau surface with distinguished complex structure, and let $w: M \to \bC$ be its mirror, i.e.~a Milnor fibre of the dual cusp of $D$. Assume that $Y$ contains infinitely many $(-2)$ curves. Then:

\begin{enumerate}
\item There is a countable infinite collection of Lagrangian spheres in $M$, none of which is contained in the subcategory of $D^b\cW(M)$ split-generated by the others. 

\item The Dehn twist in any one of them is not contained in the subgroup of $\pi_0 \Symp_c (M)$ generated by any of the others; similarly with squares of Dehn twists.

\end{enumerate}

\end{corollary}

\begin{proof} 
Let $C_i$, $i \in \bN$, be an infinite collection of $(-2)$ curves in $Y$; wlog they all lie in $U$. Let $S_i = s_{C_i}$ be the Lagrangian sphere mirror to $i_\ast \cO_{C_i}(-1)$, from Proposition \ref{prop:Lagrangian_spheres}. Let $p_i \in C_i \backslash \cup_{j \neq i} C_j$. By considering morphisms with $i_\ast \cO_{p_i}$, we see that $i_\ast \cO_{C_i}(-1)$ cannot be contained by the subcategory split-generated by $\{ i_\ast \cO_{C_j}(-1) \}_{j \neq i}$. The homological mirror symmetry theorem then implies the first point. For the second one, one can proceed similarly, noting that $i_\ast \cO_{p_i}$ is fixed by the spherical twist in $ i_\ast \cO_{C_j}(-1)$ for $j \neq i$, but not under the twist  in $ i_\ast \cO_{C_i}(-1)$, nor its square. 
\end{proof}

\begin{remark}
Using  \cite{Sheridan-Smith}, the same statements and proofs hold for K3 surfaces whose mirror K3s containing infinitely many embedded $(-2)$ curves. 
\end{remark}

\emph{Discussion.} All of the spheres in Corollary \ref{cor:infinite_spheres}  can be thought of as vanishing cycles for the cusp singularity dual to the one with cycle $D$, with Dehn twists in them induced by monodromy maps in the moduli spaces from Section \ref{sec:monodromies}. 
In particular, 
Corollary \ref{cor:infinite_spheres} should be contrasted with the case of (arbitrary) hypersurface singularities, for which Picard--Lefschetz theory tells us that any distinguished collection of vanishing cycles generates all others as objects of the Fukaya category, and that Dehn twists in them generate the monodromy group. 
(Note that the full compact Fukaya category of $M$ can never be split generated by spherical objects: this is immediate by using HMS and by considering supports of mirror objects.) 
On the other hand, by \cite{Li_thesis},  $\Aut(Y)$ acts on the collection of $(-2)$ curves with finitely many orbits (note that there's no a priori uniform bound); thus there are finitely many orbits of known Lagrangian spheres in  $D^b \cW(M)$  under the action of  compactly supported symplectomorphisms. 
There could be more than one orbit: for instance, 
in the $k=6$ semi-definite case, $Y_e$ is the total space of  elliptic fibration where the boundary $D$ is a cycle of 6 $(-2)$ curves, and the interior $(-2)$ curves form an $I_2$ and an $I_3$ fibre, giving two orbits of Lagrangians spheres in $D^b\cW(M)$, where $M$ is mirror to $Y_e \backslash D$. (Note that $\Adm = W$ here.)

\subsection{Semi-definite case: monodromy and mirror to the $SL_2(\bZ)$ action} \label{sec:monodromy_semidef_symplectic_side}
Assume that $D$ is negative semi-definite, i.e.~that $M$ is the Milnor fibre of a simple elliptic singularity.  We re-visit the monodromy perspective from the introduction in this case. We have that $M = X_0 \backslash E_0$, where $X_0$ is a del Pezzo surface, with $E_0$ smooth anticanonical. The natural moduli space of complex structures for $M$ consists of pairs $(X, E)$ where $X$ is a del Pezzo surface deformation equivalent to $X_0$, and $E \subset X$ a smooth anticanonical divisor. 

First, we want to carefully check that any loop in this moduli space induces a symplectomorphism of $M$, well-behaved near $\partial M$ in the cases where it's not compactly supported. Assume we have a smooth loop of pairs  $\{ (X_t, E_t) \}_{t \in [0,1]}$ where the $X_t$ are del Pezzo surfaces, and $E_t \subset X_t$ are smooth anticanonical elliptic curves. $M$ is equipped with $\omega_0$,  the restriction of a Kaehler form Poincar\'e dual to $E_0$; by making coherent choices of global sections of a smoothly varying very ample line bundle, we get a smooth loop  of triples $\{ (X_t, E_t, \omega_t ) \}_{t \in [0,1]}$, where $\omega_t \in \Omega^2(X_t)$ is a Kaehler form such that $[\omega_t] = P.D[E_t] \in H^2(X_t; \bZ)$. Let $f_t: X_0 \to X_t$ parametrise the family. Set $M_t := X_t \backslash E_t$, and $M = M_0 = M_1$. By assumption,  $\omega_t|_{M_t}$ is exact, and using the tubular neighbourhood theorem for symplectic submanifolds, we can choose a primitive $\theta_t$ for $\omega_t|_{M_t}$ such that on a collar neighbourhood of $\partial M_t$, the corresponding Liouville vector field radially scales each fibre of the symplectic normal bundle of $E_t$ (inwards with respect to the deleted zero-section $E_t$). Moreover, after isotopy, we can assume that $f_t^{-1}(E_t) = E_0$, and that $f_t^\ast \theta_t = \theta_0$ in a neighbourhood of $E_t$; note that $f_1$ can be arranged to restrict to the identity on this neighbourhood precisely when it induces the identity on $H_1(E_0; \bZ)$. A Moser argument gives a symplectomorphism $f: M \to M$, smoothly isotopic to $f_1$ rel boundary, which is the compactly supported whenever $(f_1)_\ast$ is the identity on $H_1(E_0; \bZ)$. 
Choosing a grading, the map $f$ acts on the wrapped Fukaya category; using the framework of stops from \cite{Sylvan,GPS_sectorial}, we expect to arrange an action on the directed Fukaya category.

Consider $H_2(M; \bZ)$. The ``tube over a cycle" map gives a primitive embedding $H_1(E;\bZ) \subset H_2(M; \bZ)$. This sublattice can be characterised as the radical of the degenerate intersection form on $H_2(M; \bZ)$. Now suppose that  $g$ is an arbitrary symplectomorphism of $M$. It induces an automorphism of  $H_2(M; \bZ)$ respecting the intersection form, and so restricts to an automorphism of  $H_1(E;\bZ)$. 
If $g$ is the result of a monodromy construction as above, we recover the original action on $H_1(E; \bZ)$. 
Dually, one could start with the action of $g$ on $H_2(M, \partial M; \bZ)$, and get the (dual) action on $H^1(E; \bZ)$ via cap products. Alternatively, if $g$ induces an autoequivalence of the directed Fukaya category of $M$, this could be obtained by looking at the induced element of $ \Aut K(D^b \cF^\to (w)) $ and then considering the action on the radical to get an element in $\Aut H^1(E; \bZ)$. 
Call the resulting map $\widetilde{\Theta}: \Auteq D^b \cF^\to (w) \to \Aut H^1(E; \bZ)$. 
The map $\widetilde{\Theta}$ is surjective: by considering paths of smooth elliptic curves in $\bP^2$ (or $\bP^1 \times \bP^1$) and blowing up points on those curves, we can arrange to get arbitrary elements of $SL_2(\bZ)$ (cf.~\cite[Section 5.4]{Keating_tori}). 

Finally, recall that on the mirror side, we have the map $\Theta: \Auteq D(Y) \to \End K(F)$ of equation \ref{eq:auteq_elliptic}, where $F$ denotes a smooth fibre of the elliptic fibration $\varpi: Y \backslash D \to \bC$. $K(F)$ is generated by $[\cO_{pt}]$ and $[\cO_F]$. 
Now the radical of the intersection pairing on  $H_2( Y \backslash D; \bZ)$ is the primitive sublattice $\langle \gamma, [F] \rangle$, where $\gamma$ is the `linking torus' for a node of $D$. Homological mirror symmetry induces an identification of this lattice with the sublattice $H_1(E;\bZ) \subset H_2(M; \bZ)$; and by construction $\widetilde{\Theta}$ is mirror to $\Theta$. (On the $B$ side, $K(F)$ should be thought of as dual to $\langle \gamma, [F] \rangle = H_1(E; \bZ) \subset H_2(M; \bZ)$, with  $\gamma$ corresponding to $[\cO_{pt}]$ and $F$ to $[\cO_F]$). In particular, we expect that there is a non-compactly supported symplectomorphism mirror to the fibrewise Fourier-Mukai transform described in Section \ref{sec:semi_def_D}.

\begin{remark}\label{rmk:Hicks} This should be compared with \cite[Section 5]{Hicks}: in the case where $X = \bC \bP^2$, 
Hicks constructs a symplectomorphism $g$ of $M = \bC \bP^2 \backslash E$ (where $E$ is smooth anticanonical), by using an automorphism of the Hesse pencil on $\bC \bP^2$, which interchanges an SYZ fibre and a  tropical Lagrangian torus,  $L_{T^2}$ in his notation.  The mirror log CY surface $Y$, denoted $\check{X}_{9111}$ in \cite{Hicks}, is the total space of a rational elliptic fibration with an $I_9$ fibre above infinity and three other singular fibres, all nodal. \cite[Section 5]{Hicks} provides evidence that $g$ is mirror to a fibrewise Fourier--Mukai transform on $Y$.  
\end{remark}

\subsection{Semi-definite case: connection with the special Lagrangian viewpoint} \label{sec:semi-def_hyperkaehler}

We briefly discuss connections between our results and work of Collins--Jacob--Lin in the semi-definite case \cite{CJL, CJL2}. 

Suppose $X$ is a del Pezzo surface, $E$ a smooth anticanonical divisor, and $M = X \backslash E$. There exists a meromorphic volume form $\Omega$ on $X$ with a simple pole along $E$; this is unique up to scaling by a constant. Moreover, $M$ has a complete Ricci flat metric \cite{Tian-Yau}, say $g_{TY}$. Let $\omega_{TY}$ be the associated Kaehler form (automatically exact). 

\begin{theorem}\cite{CJL}\label{thm:CJL_fibration}
Fix any primitive homology class $\alpha \in H_1(E, \bZ)$. Then 
$M$ admits a special Lagrangian fibration $\pi_\alpha: M \to \bR^2$ with respect to $\omega_{TY}$, with special Lagrangian torus fibre the image of $\alpha$ under the tube map $H_1 (E, \bZ) \to H_2(M, \bZ)$. 
\end{theorem}

Note that once $\alpha$ has been specified, the (almost-)toric fibration underlying the special one should be unique up to nodal slides.  

Let $(\alpha_\gamma, \alpha_F)$ be the basis of $H_1(E; \bZ)$ corresponding to $(\gamma, [F]) \subset H_2(Y \backslash U; \bZ)$ under the mirror identification. (As before, $\gamma $ is the class of a linking torus for a node of $D$, and $[F]$ the class of a fibre of the elliptic fibration $\varpi$ on the mirror $Y_e$.) 
Intuitively, exactly one of the fibrations above should be thought of as the SYZ fibration: the one $\pi_\gamma$ with fibre in class $\Tube(\alpha_\gamma) \in H_2(M; \bZ)$. Recall from Section \ref{sec:monodromy_semidef_symplectic_side} that monodromy considerations give non-compactly supported symplectomorphisms inducing any $SL_2(\bZ)$ action on $H_1(E; \bZ)$. We expect that for arbitrary $\alpha$,  $\pi_\alpha$ can be obtained  as a Lagrangian fibration by precomposing $\pi_\gamma$ with a suitable non-compactly supported symplectomorphism. In the case of $\pi_{\alpha_F}$, this should be the mirror to the fibrewise Fourier-Mukai transform.

Collins--Jacob--Lin also study properties of the hyperkaehler structure associated to these Ricci flat metrics. Assume that the holomorphic volume form on $M$ is normalised so that the special torus fibres have argument zero. Hyperkaehler rotation on $M$ (from the $I$-structures to the $J$-structures with the standard conventions) yields a complex manifold $U_{\HK}$, and $\pi$ induces a holomophic fibration $\varpi_{\HK}: U_{\HK} \to \bC$; by \cite[Theorem 6.4]{CJL}, $U_{\HK}$ admits a compactification to a relatively minimal elliptic fibration to $\bP^1$ with an $I_{k}$ fibre at infinity, where $k$ is the degree of the del Pezzo surface $X$. In particular, this implies that $Y$ is in the deformation component of the homological mirror $Y_e$ to $M$. 

While hyperkaehler rotation should generally \emph{not} be thought of as mirror symmetry, an analysis in this case shows that the two can be set up to agree for a distinguished moduli space point: following \cite[Section 5]{CJL2} (and \cite{McMullen}), we can find 
 a smooth del Pezzo surface $X'$ deformation equivalent to $X$; a smooth anticanonical elliptic curve $E \subset X'$; and a holomorphic volume form $\Omega^{M'}$ on $M'$ such that:
\begin{enumerate}
\item Hyperkaehler rotation on $M$ (from the $I$ structures to the $J$ ones) gives the distinguished complex structure $U_e = Y_e \backslash D$. 

\item Under this hyperkaehler rotation, the special Lagrangian fibration on $M = X \backslash E$ with fibre class $\Tube( \alpha_F )$, say $\pi_F$, becomes an elliptic fibration $\varpi^0: U_e \to \bC$ which compactifies to $\varpi: Y_e \to \bP^1$.

\item The special Lagrangian fibration on $M$ with fibre class $\Tube( \alpha_\gamma )$ has fibres with argument $\pi / 2$. After hyperkaehler rotation, this gives a special Lagrangian fibration on $U_e$. 
\end{enumerate}
Cf.~\cite[Theorem 5.4]{CJL2}. This set-up can be used to obtain symplectomorphisms of $M$:

\begin{lemma}\label{lem:MW_symplectic} \cite[Proposition 5.10]{CJL2}
Suppose we are in the setting above. Let $\psi \in MW(Y_e, \bP^1)$ be an element of the Mordell-Weil group of $Y$. Then as a diffeomorphism of $M$, with the symplectic form given by hyperkaehler rotation, $\psi$ is in fact a symplectomorphism of $M$, say $\rho_\psi$. 
\end{lemma}

\begin{proof} This is essentially immediate from hyperkaehler identities. 
The map $\psi$ is holomorphic on $U = U_e$ (whose complex structure is $J$), so it must at worst scale $\Omega_J = \Omega^U$. Moreover, the scaling factor must in fact be one: $\psi$ acts by translation on each fibre, and uniformising the fibration near a smooth fibre one sees that the scaling factor must be one.  (See \cite[Lemma 15.4.4]{Huybrechts_K3} for details on an analogous argument, or the proof of  \cite[Proposition 5.10]{CJL2} for a more analytic viewpoint.) 
Now split the equation $\psi^\ast \Omega_J =  \Omega_J$ into real and imaginary parts to see that $\psi$ preserves the symplectic form on $M$. 
\end{proof}

Suppose that $\psi$ is a $k$th power in the Mordell-Weil group, where $k$ still denotes the degree of the del Pezzo surface $Y$. The map $\rho_\psi$ above doesn't have compact support: this is merely true asymptotically, in a sense which could be made precise. However, we observe that the ideas of Section \ref{sec:translations_and_tensors} can be used in this setting to improve this. We first note the following:

\begin{lemma}\label{lem:hk_MW_is_Lag_translation}
Let $s_0$ be the choice of reference holomorphic section for the group law in $MW(Y_e, \bP^1)$, and $s_\psi$ the one for $\psi$. Then their images under hyperkaehler rotation are Lagrangian sections of $\pi_F$, say $L_e$ and $L_\psi$, respectively. Moreover, away from critical fibres (and on fibres with a single critical point), $\rho_\psi$ is precisely equal to the Lagrangian translation of $\pi_\w: M \to \bR^2$ which takes $L_e$ to $L_\psi$. 
\end{lemma}

\begin{proof}
Hyperkaehler identities imply that holomorphic sections of $\w$ become Lagrangian sections of $\pi_F$. By Lemma \ref{lem:MW_symplectic}, the map $\rho_\psi$ is a symplectomorphism of $M$, preserving the fibres of $\pi_\w$. By \cite[Lemma 2.5]{Symington}, this implies that on any Arnol'd--Liouville chart $V \times T^2$, $V \subset \bR^2$, it is of the form $(q, p) \mapsto (q , p+ F(q))$, where $F$ is the derivative of a smooth function $V \to \bR$. The claim is then immediate.
\end{proof}

\begin{corollary}\label{cor:MW_symp_compact}
Let $\psi \in k \cdot MW(Y_e, \bP^1)$ be an element of the Mordell-Weil group of $Y_e$ fixing $D$ pointwise, where $k$ is the degree of the del Pezzo surface $Y_e$. Then $\rho_\psi$ is Hamiltonian isotopic, via a fibre-preserving isotopy of $\pi_F: M \to \bR^2$, to a compactly supported symplectomorphism, say $\rho^c_\psi$. 
\end{corollary}

\begin{proof}
As $\psi \in k \cdot MW(Y_e, \bP^1)$, the two sections $s_0$ and $s_\psi$ defining $\psi$ intersect at infinity, at an interior point of a component of $D = \w^{-1} (\infty)$. Thus the sections $s_0$ and $s_\psi$ are smoothly isotopic near infinity; hence $L_e$ and $L_\psi$ are too. It then follows from Corollary \ref{cor:isotoping_Lag_sections} after a fibre-preserving Hamiltonian isotopy $L_e$ and $L_\psi$ can be taken to agree outside a compact set. (Recall  the obstruction is the symplectic area of an annulus; one observes that it can be capped off with two Lagrangian discs, given by $L_e$ and $L_\psi$, to get a closed chain; and as $\omega \in \Omega^2(M)$ is exact, the obstruction must thus vanish.)
Lagrangian translation between $L_e$ and the images of $L_\psi$ under isotopy then gives the desired isotopy of $\rho_\psi$ to a compactly supported symplectomorphism.
\end{proof}

As before, this gives non-trivial symplectomorphisms for two cases: $k=7$, or $k=8$ with $X = \bF_1$. By construction, the $K$-theory actions of $\psi$ and $\rho^c_\psi$ are mirror. For the $k=8$ case, as there are no interior $(-2)$ curves in $Y_e$, Proposition \ref{prop:detecting_id} immediately implies that the maps are HMS mirrors. (For the $k=7$ case, there is an interior $I_2$ fibre, and no a priori reason for the two maps not to differ by a Torelli autoequivalence generated by twists in the associated spherical objects.)

\bibliography{bib}{}
\bibliographystyle{alpha}

\end{document}